\documentclass[12pt,fleqn]{article}

\usepackage{amscd,amsmath,amssymb,amsthm}
\usepackage{geometry} 
\usepackage[pdfpagemode=UseNone,pdfstartview=FitH,hypertex]{hyperref}

\newcommand{\A}{{\mathcal A}}
\newcommand{\As}[1][]{A_s #1}
\newcommand{\abs}[1]{\left|#1\right|}
\newcommand{\Bs}[1][]{B_s #1}
\newcommand{\bdry}[1]{\partial #1}
\newcommand{\bgdnorm}[2][]{\big\|#2\big\|_{#1}^\ast}
\newcommand{\bgset}[1]{\big\{#1\big\}}
\newcommand{\D}{{\mathcal D}}
\newcommand{\dint}{\ds{\int}}
\newcommand{\dist}[2]{\text{dist}\, (#1,#2)}
\newcommand{\dnorm}[2][]{\left\|#2\right\|_{#1}^\ast}
\newcommand{\ds}[1]{\displaystyle #1}
\newcommand{\eps}{\varepsilon}
\newcommand{\F}{{\mathcal F}}
\newcommand{\hquad}{\hspace{0.08in}}
\newcommand{\id}{id}
\newcommand{\incl}{\hookrightarrow}
\newcommand{\ins}[1]{\hquad \text{#1} \hquad}
\newcommand{\isom}{\approx}
\newcommand{\loc}{\text{loc}}
\newcommand{\M}{{\mathcal M}}
\newcommand{\N}{\mathbb N}
\newcommand{\norm}[2][]{\left\|#2\right\|_{#1}}
\renewcommand{\O}{\text{O}}
\renewcommand{\o}{\text{o}}
\newcommand{\PS}[1]{$(\text{PS})_{#1}$}
\newcommand{\pnorm}[2][]{\if #1'' \left|#2\right|_p \else \left|#2\right|_{#1} \fi}
\newcommand{\QED}{\mbox{\qedhere}}
\newcommand{\R}{\mathbb R}
\newcommand{\RP}{\R \text{P}}
\newcommand{\restr}[2]{\left.#1\right|_{#2}}
\newcommand{\seq}[1]{\left(#1\right)}
\newcommand{\set}[1]{\left\{#1\right\}}
\newcommand{\wto}{\rightharpoonup}
\newcommand{\Z}{\mathbb Z}

\DeclareMathOperator{\im}{im}

\newenvironment{enumroman}{\begin{enumerate}

}{\end{enumerate}}

\newtheorem{corollary}{Corollary}[section]
\newtheorem{lemma}[corollary]{Lemma}
\newtheorem{proposition}[corollary]{Proposition}
\newtheorem{theorem}[corollary]{Theorem}

\theoremstyle{definition}
\newtheorem{definition}[corollary]{Definition}

\theoremstyle{remark}
\newtheorem{example}[corollary]{Example}
\newtheorem{remark}[corollary]{Remark}

\numberwithin{equation}{section}

\title{\bf Variational methods for scaled functionals with applications to the Schr\"{o}dinger--Poisson--Slater equation}
\author{\bf Carlo Mercuri\\
Dipartimento di Scienze Fisiche, Informatiche e Matematiche\\
Universit\`{a} di Modena e Reggio Emilia\\
Via Campi 213/b, 41125 Modena, Italy\\
\em carlo.mercuri@unimore.it\\
[\medskipamount]
\bf Kanishka Perera\\
Department of Mathematics\\
Florida Institute of Technology\\
150 W University Blvd, Melbourne, FL 32901-6975, USA\\
\em kperera@fit.edu}
\date{}

\begin{document}

\maketitle

\begin{abstract}
We develop novel variational methods for solving scaled equations that do not have the mountain pass geometry, classical linking geometry based on linear subspaces, or $\Z_2$ symmetry, and therefore cannot be solved using classical variational arguments. Our contributions here include new critical group estimates for scaled functionals, nonlinear saddle point and linking geometries based on scaling, a notion of local linking based on scaling, and scaling-based multiplicity results for symmetric functionals. We develop these methods in an abstract setting involving scaled operators and scaled eigenvalue problems. Applications to subcritical and critical Schr\"{o}dinger--Poisson--Slater equations are given.
\end{abstract}

\newpage

{\small \tableofcontents}

\section{Introduction}

In this paper we provide new variational insights on a broad class of possibly nonlinear and nonlocal partial differential equations, such as, in the spirit of Berestycki and Lions \cite{MR695535,MR695536}, the zero-mass equation
\begin{equation} \label{152}
- \Delta u + \left(\frac{1}{4 \pi |x|} \star u^2\right) u = f(|x|,u) \quad \text{in } \R^3,
\end{equation}
by taking advantage of their scaling properties. To this aim we introduce new minimax principles that are applicable in instances where the mountain pass geometry, linking geometry based on linear subspaces, or $\Z_2$ symmetry may not be present, making most of the variational approaches based on variants of the classical work of Ambrosetti and Rabinowitz \cite{MR0370183} not directly applicable.

We begin by introducing an abstract notion of a scaling in a Banach space, identifying a set of axioms that allow us to develop novel variational methods and prove new abstract critical point theorems. Our abstract results include new critical group estimates for scaled functionals, nonlinear saddle point and linking geometries based on scaling, a notion of local linking based on scaling, and scaling-based multiplicity results for symmetric functionals. Our axiomatic approach to scaled operators also leads to a new notion of scaled nonlinear eigenvalue problems. We construct minimax eigenvalues of these eigenvalue problems using the $\Z_2$-cohomological index of Fadell and Rabinowitz \cite{MR0478189}. The choice of using this index instead of the classical genus is for the purposes of obtaining critical group estimates and constructing scaling-based linking sets.

When applied to equation \eqref{152}, our axiomatic approach yields a natural classification of different nonlinear regimes in terms of scaling properties of the equation. We give a range of applications that include existence and multiplicity results in different regimes, obtained with a unified approach to nonlinearities which may be either Sobolev-subcritical or critical (we also consider some supercritical nonlinearities). Other applications will be given in forthcoming contributions, some of which are mentioned at the end of this introduction, in a section on open problems.

\subsection{Motivations}

The nonlinear and nonlocal Schr\"odinger equation
\begin{equation} \label{NLSt}
i\, \partial_t \psi = - \Delta \psi + V_{\mathrm{ext}}(x)\, \psi + \left(\frac{1}{4 \pi |x|} \star |\psi|^2\right) \psi - |\psi|^{q-2}\, \psi, \quad (x,t) \in \R^3 \times \R,
\end{equation}
as well as its stationary version
\begin{equation} \label{spsx}
- \Delta u + (V_{\mathrm{ext}}(x) + \omega)\, u + \left(\frac{1}{4 \pi |x|} \star |u|^2\right) u = |u|^{q-2}\, u, \quad x \in \R^3,
\end{equation}
obtained from \eqref{NLSt} via the ansatz $\psi(x,t) = e^{i \omega t}\, u(x)$, are approximations of the Hartree--Fock model for a quantum many-body system of electrons under the effect of the external potential $V_{\mathrm{ext}}$ (see, e.g., Le Bris and Lions \cite{MR2149087} and references therein) and arise, as shown more recently, in models for semiconductors (see, e.g., Sanchez and Soler \cite{MR2032129}).

According to different contexts, equations \eqref{NLSt} and \eqref{spsx} are known under different names, such as {\em Schr\"odinger--Poisson's} (see, e.g., Ambrosetti \cite{MR2465993}), {\em Schr\"odinger--Poisson--Slater's} (see, e.g., Bokanowski et al.\! \cite{MR2013491}), {\em Schr\"odin\-ger--Poisson--$X_\alpha$'s} (see, e.g., Mauser \cite{MR1836081}), or {\em Maxwell--Schr\"o\-din\-ger--Poisson's} (see, e.g., Benci and Fortunato \cite{MR1659454} and Catto et al.\! \cite{MR3078678}). The function $|u|^2 : \R^3 \to \R$ in equation \eqref{spsx} represents the density of electrons in the original many-body system, whereas the convolution term represents the Coulombic repulsion between the electrons. The local term $|u|^{q-2}\, u$ was introduced by Slater \cite{Slater} with $q = 8/3$ as a local approximation of the exchange potential in the Hartree--Fock model (see, e.g., Bokanowski et al.\! \cite{MR2013491}, Bokanowski and Mauser \cite{MR1702877}, and Mauser \cite{MR1836081}). When $V_{\mathrm{ext}} \equiv 0$, as remarked by Ruiz \cite[p.351]{MR2679375}, equation \eqref{152} with $f(|x|,u) = |u|^{q-2}\, u$ can be regarded as a singular limit of
\begin{equation} \label{spsx2}
- \Delta v + v + \lambda \left(\frac{1}{4 \pi |x|} \star |v|^2\right) v = |v|^{q-2}\, v, \quad x \in \R^3,
\end{equation}
where $\lambda > 0$ and $q \in (2,3)$, noting that with the rescaling $u(x) = \omega^{1/(q-2)}\, v(\sqrt{\omega}\, x)$, with $\omega = \lambda^{(q-2)/2\,(3-q)}$, \eqref{spsx2} reduces to \eqref{spsx}.

The equation \eqref{spsx} is also related to the Thomas--Fermi--Dirac--von\;Weizs\"acker (TFDW) model of Density Functional Theory (DFT) (see, e.g., Lieb \cite{ MR629207,MR641371}, Le Bris and Lions \cite{MR2149087}, Lu and Otto \cite{MR3251907}, Frank et al.\! \cite{MR3762278}, and references therein). In the context of the TFDW model, the local nonlinearity takes the form $f(|x|,u) = u^{5/3} - u^{7/3}$, for which our Corollary \ref{Corollary 4} provides the existence of infinitely many solutions at negative energy levels. Finally, we stress that the mathematical choice of considering in equation \eqref {152} a more general local nonlinearity $f(|x|,u)$ is partly motivated by physical considerations again arising from DFT and quantum chemistry models. These, such as Kohn-Sham's, involve general local nonlinearities of the type we consider, which correspond to homogeneous electron gas approximations of the so-called exchange-correlation potential, which is not explicitly known (see, e.g., Anantharaman and Canc\`es \cite{MR2569902} and references therein).

\subsection{Overview of the theory}

\begin{definition}
Let $W$ be a reflexive Banach space. A scaling on $W$ is a continuous mapping $W \times [0,\infty) \to W,\, (u,t) \mapsto u_t$ satisfying
\begin{enumerate}
\item[$(H_1)$] $(u_{t_1})_{t_2} = u_{t_1 t_2}$ for all $u \in W$ and $t_1, t_2 \ge 0$,
\item[$(H_2)$] $(\tau u)_t = \tau u_t$ for all $u \in W$, $\tau \in \R$, and $t \ge 0$,
\item[$(H_3)$] $u_0 = 0$ and $u_1 = u$ for all $u \in W$,
\item[$(H_4)$] $u_t$ is bounded on bounded sets in $W \times [0,\infty)$,
\item[$(H_5)$] $\exists s > 0$ such that $\norm{u_t} = \O(t^s)$ as $t \to \infty$, uniformly in $u$ on bounded sets.
\end{enumerate}
\end{definition}

Denote by $W^\ast$ the dual of $W$. Recall that $q \in C(W,W^\ast)$ is a potential operator if there is a functional $Q \in C^1(W,\R)$, called a potential for $q$, such that $Q' = q$. By replacing $Q$ with $Q - Q(0)$ if necessary, we may assume that $Q(0) = 0$.

\begin{definition}
A scaled operator is an odd potential operator $\As \in C(W,W^\ast)$ that maps bounded sets into bounded sets and satisfies
\[
\As(u_t) v_t = t^s \As(u) v \quad \forall u, v \in W,\, t \ge 0.
\]
\end{definition}

If $\As$ is a scaled operator, then its potential $I_s$ with $I_s(0) = 0$ is even, bounded on bounded sets, and has the scaling property
\begin{equation} \label{112}
I_s(u_t) = t^s I_s(u) \quad \forall u \in W,\, t \ge 0
\end{equation}
(see Proposition \ref{Proposition 1}).

Let $\As$ be a scaled operator satisfying
\begin{enumerate}
\item[$(H_6)$] $\As(u) u > 0$ for all $u \in W \setminus \set{0}$,
\item[$(H_7)$] every sequence $\seq{u_j}$ in $W$ such that $u_j \wto u$ and $\As(u_j)(u_j - u) \to 0$ has a subsequence that converges strongly to $u$.
\end{enumerate}
We consider the question of existence and multiplicity of solutions to nonlinear operator equations of the form
\begin{equation} \label{117}
\As(u) = f(u) \quad \text{in } W^\ast,
\end{equation}
where $f \in C(W,W^\ast)$ is a potential operator. Solutions of this equation coincide with critical points of the $C^1$-functional
\[
\Phi(u) = I_s(u) - F(u), \quad u \in W,
\]
where $F$ is the potential of $f$ with $F(0) = 0$. We will develop novel variational methods based on the scaling for studying critical points of functionals of this type.

For example, consider the Schr\"{o}dinger--Poisson--Slater type equation
\[
- \Delta u + \left(\frac{1}{4 \pi |x|} \star u^2\right) u = f(|x|,u) \quad \text{in } \R^3,
\]
where $f$ is a Carath\'{e}odory function on $[0,\infty) \times \R$ satisfying the growth condition
\[
|f(|x|,t)| \le a_1\, |t|^{q_1 - 1} + a_2\, |t|^{q_2 - 1} + a(x) \quad \text{for a.a.\! } x \in \R^3 \text{ and all } t \in \R
\]
for some constants $a_1, a_2 > 0$, $18/7 < q_1 < q_2 \le 6$, and $a \in L^r(\R^3)$ with $r \in [6/5,18/11)$. As pointed out in Lions \cite{MR636734} (see also Ruiz \cite{MR2679375}), the right space to look for solutions of this equation is
\[
E(\R^3) = \set{u \in D^{1,2}(\R^3) : \int_{\R^3} \int_{\R^3} \frac{u^2(x)\, u^2(y)}{|x - y|}\, dx\, dy < \infty}
\]
endowed with the norm
\[
\norm{u} = \left[\int_{\R^3} |\nabla u|^2\, dx + \left(\int_{\R^3} \int_{\R^3} \frac{u^2(x)\, u^2(y)}{|x - y|}\, dx\, dy\right)^{1/2}\right]^{1/2}.
\]
We work in the subspace $E_r(\R^3)$ of radial functions in $E(\R^3)$. Equation \eqref{152} can be written as $\As(u) = \widetilde{f}(u)$ in the dual of $E_r(\R^3)$, where the operators $\As$ and $\widetilde{f}$ are given by
\[
\As(u) v = \int_{\R^3} \nabla u \cdot \nabla v\, dx + \frac{1}{4 \pi} \int_{\R^3} \int_{\R^3} \frac{u^2(x)\, u(y)\, v(y)}{|x - y|}\, dx\, dy, \qquad \widetilde{f}(u) v = \int_{\R^3} f(|x|,u)\, v\, dx
\]
for $u, v \in E_r(\R^3)$. The space $E_r(\R^3)$ has the scaling
\[
u_t(x) = t^2\, u(tx), \quad x \in \R^3
\]
with $s = 3$, and $\As$ is a scaled operator with respect to this scaling. Here the potentials of $\As$ and $\widetilde{f}$ are given by
\[
I_s(u) = \frac{1}{2} \int_{\R^3} |\nabla u|^2\, dx + \frac{1}{16 \pi} \int_{\R^3} \int_{\R^3} \frac{u^2(x)\, u^2(y)}{|x - y|}\, dx\, dy, \qquad \widetilde{F}(u) = \int_{\R^3} F(|x|,u)\, dx
\]
for $u \in E_r(\R^3)$, respectively, where $F(|x|,t) = \int_0^t f(|x|,\tau)\, d\tau$ is the primitive of $f$. So the corresponding variational functional is
\[
\Phi(u) = \frac{1}{2} \int_{\R^3} |\nabla u|^2\, dx + \frac{1}{16 \pi} \int_{\R^3} \int_{\R^3} \frac{u^2(x)\, u^2(y)}{|x - y|}\, dx\, dy - \int_{\R^3} F(|x|,u)\, dx, \quad u \in E_r(\R^3).
\]

\subsubsection{Scaled eigenvalue problems}

First we will consider the scaled eigenvalue problem
\begin{equation} \label{114}
\As(u) = \lambda \Bs(u) \quad \text{in } W^\ast,
\end{equation}
where $\Bs$ is a scaled operator satisfying
\begin{enumerate}
\item[$(H_8)$] $\Bs(u) u > 0$ for all $u \in W \setminus \set{0}$,
\item[$(H_9)$] if $u_j \wto u$ in $W$, then $\Bs(u_j) \to \Bs(u)$ in $W^\ast$,
\end{enumerate}
and $\lambda \in \R$. We say that $\lambda$ is an eigenvalue if there is a $u \in W \setminus \set{0}$, called an eigenfunction associated with $\lambda$, satisfying equation \eqref{114}. Then $u_t$ is also an eigenfunction associated with $\lambda$ for any $t > 0$ since
\[
\As(u_t) v = \As(u_t) (v_{1/t})_t = t^s \As(u) v_{1/t} = t^s \lambda \Bs(u) v_{1/t} = \lambda \Bs(u_t) (v_{1/t})_t = \lambda \Bs(u_t) v
\]
for all $v \in W$. We will call the set $\sigma(\As,\Bs)$ of all eigenvalues the spectrum of the pair of scaled operators $(\As,\Bs)$. We note that $\sigma(\As,\Bs) \subset (0,\infty)$ by $(H_6)$ and $(H_8)$.

The potential $J_s$ of $\Bs$ with $J_s(0) = 0$ is even, bounded on bounded sets, and has the scaling property
\begin{equation} \label{116}
J_s(u_t) = t^s J_s(u) \quad \forall u \in W,\, t \ge 0
\end{equation}
(see Proposition \ref{Proposition 1}). We assume that $I_s$ and $J_s$ satisfy
\begin{enumerate}
\item[$(H_{10})$] $I_s$ is coercive, i.e., $I_s(u) \to \infty$ as $\norm{u} \to \infty$,
\item[$(H_{11})$] the equation $I_s(tu) = 1$ has a unique positive solution $t$ for each $u \in W \setminus \set{0}$,
\item[$(H_{12})$] every solution of equation \eqref{114} satisfies
    \[
    I_s(u) = \lambda\, J_s(u).
    \]
\end{enumerate}
Then the eigenvalue problem \eqref{114} has the following variational formulation. Let
\[
\Psi(u) = \frac{1}{J_s(u)}, \quad u \in W \setminus \set{0},
\]
let
\[
\M_s = \bgset{u \in W : I_s(u) = 1},
\]
and let $\widetilde{\Psi} = \restr{\Psi}{\M_s}$. Then $\M_s$ is a complete, symmetric, and bounded $C^1$-Finsler manifold, and eigenvalues of problem \eqref{114} coincide with critical values of $\widetilde{\Psi}$ (see Proposition \ref{Proposition 12}).

Using this variational characterization of the eigenvalues of problem \eqref{114}, we will construct an unbounded sequence of minimax eigenvalues. Although this can be done in a standard way using the Krasnoselskii's genus, we prefer to use the $\Z_2$-cohomological index of Fadell and Rabinowitz \cite{MR0478189} in order to obtain nontrivial critical groups and construct linking sets related to the eigenvalues (see Definition \ref{Definition 1}). Let $\F$ denote the class of symmetric subsets of $\M_s$, and let $i(M)$ be the cohomological index of $M \in \F$. For $k \ge 1$, let
\[
\F_k = \bgset{M \in \F : i(M) \ge k}
\]
and set
\[
\lambda_k := \inf_{M \in \F_k}\, \sup_{u \in M}\, \widetilde{\Psi}(u).
\]
We will prove the following theorem.

\begin{theorem} \label{Theorem 21}
Assume $(H_1)$--$(H_{12})$. Then $\lambda_k \nearrow \infty$ is a sequence of eigenvalues of \eqref{114}.
\begin{enumroman}
\item The first eigenvalue is given by
    \[
    \lambda_1 = \min_{u \in \M_s}\, \widetilde{\Psi}(u) > 0.
    \]
\item If $\lambda_k = \dotsb = \lambda_{k+m-1} = \lambda$, then $i(E_\lambda) \ge m$, where $E_\lambda$ is the set of eigenfunctions associated with $\lambda$ that lie on $\M_s$.
\item \label{Theorem 21.iii} If $\lambda_k < \lambda < \lambda_{k+1}$, then
    \[
    i(\widetilde{\Psi}^{\lambda_k}) = i(\M_s \setminus \widetilde{\Psi}_\lambda) = i(\widetilde{\Psi}^\lambda) = i(\M_s \setminus \widetilde{\Psi}_{\lambda_{k+1}}) = k,
    \]
    where $\widetilde{\Psi}^a = \bgset{u \in \M_s : \widetilde{\Psi}(u) \le a}$ and $\widetilde{\Psi}_a = \bgset{u \in \M_s : \widetilde{\Psi}(u) \ge a}$ for $a \in \R$.
\end{enumroman}
\end{theorem}

The variational functional associated with problem \eqref{114} is
\[
\Phi_\lambda(u) = I_s(u) - \lambda\, J_s(u), \quad u \in W.
\]
By \eqref{112} and \eqref{116}, $\Phi_\lambda$ has the scaling property
\[
\Phi_\lambda(u_t) = t^s\, \Phi_\lambda(u) \quad \forall u \in W,\, t \ge 0.
\]
When $\lambda \notin \sigma(\As,\Bs)$, the origin is the only critical point of $\Phi_\lambda$ and its critical groups there are given by
\[
C^l(\Phi_\lambda,0) = H^l(\Phi_\lambda^0,\Phi_\lambda^0 \setminus \set{0}), \quad l \ge 0,
\]
where $\Phi_\lambda^0 = \bgset{u \in W : \Phi_\lambda(u) \le 0}$ and $H^\ast$ denotes cohomology with $\Z_2$ coefficients. We will prove the following theorem, where $\widetilde{H}^\ast$ denotes reduced cohomology.

\begin{theorem}
Assume $(H_1)$--$(H_{12})$ and let $\lambda \in \R \setminus \sigma(\As,\Bs)$.
\begin{enumroman}
\item If $\lambda < \lambda_1$, then $C^l(\Phi_\lambda,0) \isom \delta_{l0}\, \Z_2$, where $\delta$ is the Kronecker delta.
\item If $\lambda > \lambda_1$, then $C^l(\Phi_\lambda,0) \isom \widetilde{H}^{l-1}(\widetilde{\Psi}^\lambda)$, in particular, $C^0(\Phi_\lambda,0) = 0$.
\item If $\lambda_k < \lambda < \lambda_{k+1}$, then $C^k(\Phi_\lambda,0) \ne 0$.
\end{enumroman}
In particular, $C^l(\Phi_\lambda,0)$ is nontrivial for some $l$.
\end{theorem}

For example, based on its scaling properties, an appropriate eigenvalue problem to consider for the Schr\"{o}dinger--Poisson--Slater equation is
\begin{equation} \label{115}
- \Delta u + \left(\frac{1}{4 \pi |x|} \star u^2\right) u = \lambda\, |u| u \quad \text{in } \R^3.
\end{equation}
Here
\[
J_s(u) = \frac{1}{3} \int_{\R^3} |u|^3\, dx, \quad u \in E_r(\R^3),
\]
and $(H_{12})$ follows from the Poho\v{z}aev type identity
\[
\frac{1}{2} \int_{\R^3} |\nabla u|^2\, dx + \frac{5}{16 \pi} \int_{\R^3} \int_{\R^3} \frac{u^2(x)\, u^2(y)}{|x - y|}\, dx\, dy = \lambda \int_{\R^3} |u|^3\, dx
\]
for solutions of equation \eqref{115} (see Ianni and Ruiz \cite[Proposition 2.5]{MR2902293}).

\subsubsection{Scaling-based abstract critical point theorems}

Theorem \ref{Theorem 21} \ref{Theorem 21.iii} gives
\[
i(\widetilde{\Psi}^{\lambda_k}) = i(\M_s \setminus \widetilde{\Psi}_{\lambda_{k+1}}) = k
\]
when $\lambda_k < \lambda < \lambda_{k+1}$. We will prove a series of scaling-based abstract critical point theorems that can be applied to nonlinear equations of the type \eqref{117} using this fact.

To state these theorems, let $I \in C(W,\R)$ be an even functional satisfying
\[
I(u_t) = t^s I(u) \quad \forall u \in W,\, t \ge 0
\]
and
\[
I(u) > 0 \quad \forall u \in W \setminus \set{0}.
\]
Let
\[
\M = \bgset{u \in W : I(u) = 1}
\]
and note that $\M$ is a closed and symmetric set since $I$ is continuous and even. We assume that $\M$ is bounded and that each ray starting from the origin intersects $\M$ at exactly one point. First we will prove the following scaled saddle point theorem.

\begin{theorem} \label{Theorem 22}
Let $\Phi$ be a $C^1$-functional on $W$. Let $A_0$ and $B_0$ be disjoint nonempty closed symmetric subsets of $\M$ such that
\[
i(A_0) = i(\M \setminus B_0) < \infty,
\]
let $R > 0$, and let
\begin{gather*}
X = \set{u_t : u \in A_0,\, 0 \le t \le R},\\
A = \set{u_R : u \in A_0},\\
B = \set{u_t : u \in B_0,\, t \ge 0}.
\end{gather*}
Assume that
\[
\sup_{u \in A}\, \Phi(u) \le \inf_{u \in B}\, \Phi(u), \qquad \sup_{u \in X}\, \Phi(u) < \infty.
\]
Let $\Gamma = \bgset{\gamma \in C(X,W) : \gamma(X) \text{ is closed and } \restr{\gamma}{A} = \id}$ and set
\[
c := \inf_{\gamma \in \Gamma}\, \sup_{u \in \gamma(X)}\, \Phi(u).
\]
If $\Phi$ satisfies the {\em \PS{c}} condition, then $c$ is a critical value of $\Phi$.
\end{theorem}

\begin{remark}
Theorem \ref{Theorem 22} generalizes the classical saddle point theorem of Rabinowitz \cite{MR0488128,MR0501092}. To see this, take $u_t = tu$ and $I(u) = \norm{u}$. Then $\M = \bgset{u \in W : \norm{u} = 1}$ is the unit sphere in $W$. Let $W = W_1 \oplus W_2$ be a direct sum decomposition with both $W_1$ and $W_2$ nontrivial and $W_1$ finite dimensional, and take $A_0 = W_1 \cap \M$ and $B_0 = W_2 \cap \M$. Then
\[
i(A_0) = i(\M \setminus B_0) = \dim W_1
\]
and
\begin{gather*}
X = \set{u \in W_1 : \norm{u} \le R},\\
A = \set{u \in W_1 : \norm{u} = R},\\
B = W_2.
\end{gather*}
So Theorem \ref{Theorem 22} reduces to the classical saddle point theorem in this special case.
\end{remark}

We can define a continuous projection $\pi : W \setminus \set{0} \to \M$ by
\[
\pi(u) = u_{t_u}, \quad t_u = \frac{1}{I(u)^{1/s}}.
\]
We will prove the following scaled linking theorem.

\begin{theorem} \label{Theorem 23}
Let $\Phi$ be a $C^1$-functional on $W$. Let $A_0$ and $B_0$ be disjoint nonempty closed symmetric subsets of $\M$ such that
\[
i(A_0) = i(\M \setminus B_0) < \infty,
\]
let $R > \rho > 0$, let $e \in \M \setminus A_0$, and let
\begin{gather*}
X = \set{(\pi((1 - \tau)u + \tau e))_t : u \in A_0,\, \tau \in [0,1],\, 0 \le t \le R},\\
A = \set{u_t : u \in A_0,\, 0 \le t \le R} \cup \set{(\pi((1 - \tau)u + \tau e))_R : u \in A_0,\, \tau \in [0,1]},\\
B = \set{u_\rho : u \in B_0}.
\end{gather*}
Assume that
\[
\sup_{u \in A}\, \Phi(u) \le \inf_{u \in B}\, \Phi(u), \qquad \sup_{u \in X}\, \Phi(u) < \infty.
\]
Let $\Gamma = \bgset{\gamma \in C(X,W) : \gamma(X) \text{ is closed and } \restr{\gamma}{A} = \id}$ and set
\[
c := \inf_{\gamma \in \Gamma}\, \sup_{u \in \gamma(X)}\, \Phi(u).
\]
If $\Phi$ satisfies the {\em \PS{c}} condition, then $c$ is a critical value of $\Phi$.
\end{theorem}

The special case $u_t = tu$ of this theorem was proved in Yang and Perera \cite{MR3616328}.

\begin{remark}
Theorem \ref{Theorem 23} generalizes the classical linking theorem of Rabinowitz \cite{MR0488128,MR0501092}. To see this, take $u_t = tu$ and $I(u) = \norm{u}$. Then $\M = \bgset{u \in W : \norm{u} = 1}$ is the unit sphere in $W$. Let $W = W_1 \oplus W_2$ be a direct sum decomposition with both $W_1$ and $W_2$ nontrivial and $W_1$ finite dimensional, and take $A_0 = W_1 \cap \M$ and $B_0 = W_2 \cap \M$. Then
\[
i(A_0) = i(\M \setminus B_0) = \dim W_1
\]
and
\begin{gather*}
X = \set{u + \tau e : u \in W_1,\, \tau \ge 0,\, \norm{u + \tau e} \le R},\\
A = \set{u \in W_1 : \norm{u} \le R} \cup \set{u + \tau e : u \in W_1,\, \tau \ge 0,\, \norm{u + \tau e} = R},\\
B = \set{u \in W_2 : \norm{u} = \rho}.
\end{gather*}
So Theorem \ref{Theorem 23} reduces to the classical linking theorem in this special case.
\end{remark}

Next we introduce a notion of local linking based on scaling.

\begin{definition} \label{Definition 3}
Let $\Phi$ be a $C^1$-functional on $W$. We will say that $\Phi$ has a scaled local linking near the origin in dimension $k \ge 1$ if there are disjoint nonempty closed symmetric subsets $A_0$ and $B_0$ of $\M$ with
\[
i(A_0) = i(\M \setminus B_0) = k
\]
and $\rho > 0$ such that
\begin{equation} \label{119}
\begin{cases}
\Phi(u_t) \le 0 & \forall u \in A_0,\, 0 \le t \le \rho\\[7.5pt]
\Phi(u_t) > 0 & \forall u \in B_0,\, 0 < t \le \rho.
\end{cases}
\end{equation}
\end{definition}

The special case $u_t = tu$ of this definition was given in Perera \cite{MR1700283} (see also Perera et al.\! \cite[Definition 3.33]{MR2640827} and Degiovanni et al.\! \cite{MR2661274}). The usefulness of this notion lies in the following theorem.

\begin{theorem}
If $\Phi$ has a scaled local linking near the origin in dimension $k$, then $C^k(\Phi,0)$ is nontrivial.
\end{theorem}

We will prove the following theorem on the existence of nontrivial critical points for functionals with a scaled local linking.

\begin{theorem}
Let $\Phi$ be a $C^1$-functional on $W$ that satisfies the {\em \PS{}} condition. Assume that $\Phi$ has a scaled local linking near the origin in dimension $k$ and that $\widetilde{H}^{k-1}(\Phi^a) = 0$ for some $a < 0$, where $\Phi^a = \bgset{u \in W : \Phi(u) \le a}$ and $\widetilde{H}^\ast$ denotes reduced cohomology. Then $\Phi$ has a nontrivial critical point.
\end{theorem}

In particular, we have the following corollary.

\begin{corollary}
Let $\Phi$ be a $C^1$-functional on $W$ that satisfies the {\em \PS{}} condition. Assume that $\Phi$ has a scaled local linking near the origin and that $\Phi^a$ is contractible for some $a < 0$. Then $\Phi$ has a nontrivial critical point.
\end{corollary}

\begin{remark}
Definition \ref{Definition 3} generalizes the original definition of Li and Liu \cite{MR986153,MR802575} (see also Li and Willem \cite{MR1312028}). To see this, take $u_t = tu$ and $I(u) = \norm{u}$. Then $\M = \bgset{u \in W : \norm{u} = 1}$ is the unit sphere in $W$. Let $W = W_1 \oplus W_2$ be a direct sum decomposition with $\dim W_1 = k \ge 1$ and $W_2$ nontrivial, and take $A_0 = W_1 \cap \M$ and $B_0 = W_2 \cap \M$. Then
\[
i(A_0) = i(\M \setminus B_0) = k
\]
and \eqref{119} reduces to
\[
\begin{cases}
\Phi(u) \le 0 & \forall u \in W_1,\, \norm{u} \le \rho\\[7.5pt]
\Phi(u) > 0 & \forall u \in W_2,\, 0 < \norm{u} \le \rho.
\end{cases}
\]
So Definition \ref{Definition 3} reduces to the original definition in this special case.
\end{remark}

Finally we will prove a scaling-based multiplicity result for symmetric functionals that only satisfy a local \PS{} condition and is therefore applicable to scaled equations with critical growth. Let $\Phi \in C^1(W,\R)$ be an even functional, i.e., $\Phi(-u) = \Phi(u)$ for all $u \in W$. Assume that $\exists c^\ast > 0$ such that $\Phi$ satisfies the \PS{c} condition for all $c \in (0,c^\ast)$. Let $\Gamma$ denote the group of odd homeomorphisms of $W$ that are the identity outside $\Phi^{-1}(0,c^\ast)$. Let $\A^\ast$ denote the class of symmetric subsets of $W$ and let
\[
\M_\rho = \set{u \in W : I(u) = \rho^s} = \set{u_\rho : u \in \M}
\]
for $\rho > 0$. Recall that the pseudo-index of $M \in \A^\ast$ related to $i$, $\M_\rho$, and $\Gamma$ is defined by
\[
i^\ast(M) = \min_{\gamma \in \Gamma}\, i(\gamma(M) \cap \M_\rho)
\]
(see Benci \cite{MR84c:58014}). We will prove the following multiplicity result.

\begin{theorem}
Let $A_0$ and $B_0$ be symmetric subsets of $\M$ such that $A_0$ is compact, $B_0$ is closed, and
\[
i(A_0) \ge k + m - 1, \qquad i(\M \setminus B_0) \le k - 1
\]
for some $k, m \ge 1$. Let $R > \rho > 0$ and let
\begin{gather*}
X = \set{u_t : u \in A_0,\, 0 \le t \le R},\\
A = \set{u_R : u \in A_0},\\
B = \set{u_\rho : u \in B_0}.
\end{gather*}
Assume that
\[
\sup_{u \in A}\, \Phi(u) \le 0 < \inf_{u \in B}\, \Phi(u), \qquad \sup_{u \in X}\, \Phi(u) < c^\ast.
\]
For $j = k,\dots,k + m - 1$, let $\A_j^\ast = \set{M \in \A^\ast : M \text{ is compact and } i^\ast(M) \ge j}$ and set
\[
c_j^\ast := \inf_{M \in \A_j^\ast}\, \max_{u \in M}\, \Phi(u).
\]
Then $0 < c_k^\ast \le \dotsb \le c_{k+m-1}^\ast < c^\ast$, each $c_j^\ast$ is a critical value of $\Phi$, and $\Phi$ has $m$ distinct pairs of associated critical points.
\end{theorem}

The special case $u_t = tu$ of this theorem was proved in Yang and Perera \cite{MR3616328}.

\subsection{Applications to the Schr\"{o}dinger--Poisson--Slater equation}

We will give applications of our abstract theory to the Schr\"{o}dinger--Poisson--Slater type equation
\begin{equation} \label{52}
- \Delta u + \left(\frac{1}{4 \pi |x|} \star u^2\right) u = f(|x|,u) \quad \text{in } \R^3,
\end{equation}
where $f$ is a Carath\'{e}odory function on $[0,\infty) \times \R$ satisfying the growth condition
\begin{equation} \label{53}
|f(|x|,t)| \le a_1\, |t|^{q_1 - 1} + a_2\, |t|^{q_2 - 1} + a(x) \quad \text{for a.a.\! } x \in \R^3 \text{ and all } t \in \R
\end{equation}
for some constants $a_1, a_2 > 0$, $18/7 < q_1 < q_2 \le 6$, and $a \in L^r(\R^3)$ with $r \in [6/5,18/11)$. As pointed out in Lions \cite{MR636734} (see also Ruiz \cite{MR2679375}), the right space to look for solutions of this equation is
\[
E(\R^3) = \set{u \in D^{1,2}(\R^3) : \int_{\R^3} \int_{\R^3} \frac{u^2(x)\, u^2(y)}{|x - y|}\, dx\, dy < \infty}
\]
endowed with the norm
\[
\norm{u} = \left[\int_{\R^3} |\nabla u|^2\, dx + \left(\int_{\R^3} \int_{\R^3} \frac{u^2(x)\, u^2(y)}{|x - y|}\, dx\, dy\right)^{1/2}\right]^{1/2}.
\]
We work in the subspace $E_r(\R^3)$ of radial functions in $E(\R^3)$. It was shown in Ruiz \cite[Theorem 1.2]{MR2679375} that $E_r(\R^3)$ is embedded in $L^q(\R^3)$ continuously for $q \in (18/7,6]$ and compactly for $q \in (18/7,6)$, and these ranges are optimal by Mercuri et al. \cite[Theorem 4 and Theorem 5]{MR3568051} (see also Bellazzini et al.  \cite{MR3852465}). So, in view of our growth condition \eqref{53}, the functional
\[
\Phi(u) = \frac{1}{2} \int_{\R^3} |\nabla u|^2\, dx + \frac{1}{16 \pi} \int_{\R^3} \int_{\R^3} \frac{u^2(x)\, u^2(y)}{|x - y|}\, dx\, dy - \int_{\R^3} F(|x|,u)\, dx, \quad u \in E_r(\R^3),
\]
where $F(|x|,t) = \int_0^t f(|x|,\tau)\, d\tau$ is the primitive of $f$, is well-defined, $C^1$, and its critical points correspond to solutions of equation \eqref{52}.

The pure power equation
\begin{equation} \label{54}
- \Delta u + \left(\frac{1}{4 \pi |x|} \star u^2\right) u = \lambda\, |u|^{q-2}\, u \quad \text{in } \R^3,
\end{equation}
where $\lambda > 0$ and $q \in (18/7,6)$, was studied in Ruiz \cite{MR2679375} and Ianni and Ruiz \cite{MR2902293}. For $q \in (18/7,3)$, it was shown in \cite[Theorem 1.3]{MR2679375} that the corresponding functional
\[
\Phi_\lambda(u) = \frac{1}{2} \int_{\R^3} |\nabla u|^2\, dx + \frac{1}{16 \pi} \int_{\R^3} \int_{\R^3} \frac{u^2(x)\, u^2(y)}{|x - y|}\, dx\, dy - \frac{\lambda}{q} \int_{\R^3} |u|^q\, dx, \quad u \in E_r(\R^3)
\]
is coercive and has a minimizer at a negative level for all $\lambda > 0$. For $q \in (3,6)$, it was shown in \cite[Theorem 1.2]{MR2902293} that equation \eqref{54} has infinitely many radial solutions for all $\lambda > 0$ as a result of the $\Z_2$ symmetry of $\Phi_\lambda$. It was also noted in \cite{MR2902293} that in the borderline case of $q = 3$, if $u$ is a solution of the equation
\begin{equation} \label{55}
- \Delta u + \left(\frac{1}{4 \pi |x|} \star u^2\right) u = \lambda\, |u| u \quad \text{in } \R^3,
\end{equation}
then the entire $1$-parameter family of functions $t^2\, u(t\, \cdot),\, t \in \R$ are solutions and hence this equation may be thought of as a nonlinear eigenvalue problem. Moreover, a nondecreasing and unbounded sequence of positive eigenvalues with radial eigenfunctions was constructed using the Krasnoselskii genus in \cite[Theorem 1.3]{MR2902293}.

The eigenvalue problem \eqref{55} is not only nonlinear, but also nonhomogeneous. However, it has the scale invariance described above, so we will refer to it as a scaled eigenvalue problem. We will also refer to equation \eqref{54} as subscaled or superscaled depending on whether $q < 3$ or $q > 3$. More generally, we classify equation \eqref{52} as
\begin{enumroman}
\item subscaled if
    \[
    \lim_{|t| \to \infty}\, \frac{f(|x|,t)}{|t|\, t} = 0 \quad \text{uniformly a.e.},
    \]
\item asymptotically scaled if
    \[
    \lim_{|t| \to \infty}\, \frac{f(|x|,t)}{|t|\, t} = \lambda \quad \text{uniformly a.e.}
    \]
    for some $\lambda \in (0,\infty)$,
\item superscaled if
    \[
    \lim_{|t| \to \infty}\, \frac{f(|x|,t)}{|t|\, t} = \infty \quad \text{uniformly a.e.}
    \]
\end{enumroman}
We will prove existence and multiplicity results in all three cases by obtaining critical group estimates and constructing new linking geometries.

The eigenvalue problem \eqref{55} will play a central role in our results for equation \eqref{52}. However, eigenvalues based on the genus are not useful for obtaining nontrivial critical groups or for constructing linking sets, so we will use a new sequence of minimax eigenvalues $\seq{\lambda_k}$ based on the $\Z_2$-cohomological index of Fadell and Rabinowitz \cite{MR0478189} (see Definition \ref{Definition 1} and Theorem \ref{Theorem 1}). Eigenvalues based on the cohomological index were first introduced in Perera \cite{MR1998432} in the context of the $p$-Laplacian operator.

\subsubsection{Subcritical case} \label{1.2.1}

Our first contribution here is towards understanding the local geometry of the functional $\Phi$ near the origin when
\[
f(|x|,t) = \lambda\, |t| t + \o(t^2) \text{ as } t \to 0
\]
uniformly a.e.\! for some $\lambda \in \R$. We show that some critical group of $\Phi$ at the origin is nontrivial if $\lambda$ is not an eigenvalue of problem \eqref{55}. Recall that the critical groups are defined by
\[
C^l(\Phi,0) = H^l(\Phi^0,\Phi^0 \setminus \set{0}), \quad l \ge 0,
\]
where $\Phi^0 = \bgset{u \in E_r(\R^3) : \Phi(u) \le 0}$ and $H^\ast$ denotes cohomology with $\Z_2$ coefficients. We write
\begin{equation} \label{64}
f(|x|,t) = \lambda\, |t| t + g(|x|,t),
\end{equation}
and assume that
\begin{equation} \label{65}
|g(|x|,t)| \le a_3\, |t|^{q_3 - 1} + a_4\, |t|^{q_4 - 1} \quad \text{for a.a.\! } x \in \R^3 \text{ and all } t \in \R
\end{equation}
for some constants $a_3, a_4 > 0$ and $3 < q_3 < q_4 < 6$. We have the following theorem.

\begin{theorem} \label{Theorem 10}
Assume that \eqref{64} and \eqref{65} hold and that $\lambda$ is not an eigenvalue of problem \eqref{55}.
\begin{enumroman}
\item If $\lambda < \lambda_1$, then $C^l(\Phi,0) \isom \delta_{l0}\, \Z_2$, where $\delta$ is the Kronecker delta.
\item \label{Theorem 10.ii} If $\lambda > \lambda_1$, then $C^0(\Phi,0) = 0$.
\item \label{Theorem 10.iii} If $\lambda_k < \lambda < \lambda_{k+1}$, then $C^k(\Phi,0) \ne 0$.
\end{enumroman}
In particular, $C^l(\Phi,0)$ is nontrivial for some $l$.
\end{theorem}

This is the first result in the literature on the critical groups of $\Phi$. Part \ref{Theorem 10.ii} shows that when $\lambda > \lambda_1$, the origin is not a strict local minimizer and hence $\Phi$ does not have the mountain pass geometry. However, the nontrivial critical group in part \ref{Theorem 10.iii} can be used to obtain multiple nontrivial solutions of subscaled equations when $\lambda > \lambda_1$ as our next result shows.

\begin{theorem} \label{Theorem 11}
Assume that $f$ satisfies \eqref{53} with $18/7 < q_1 < q_2 < 3$, \eqref{64} and \eqref{65} hold, and that $\lambda$ is not an eigenvalue of problem \eqref{55}.
\begin{enumroman}
\item If $\lambda > \lambda_1$, then equation \eqref{52} has a nontrivial solution.
\item If $\lambda > \lambda_2$, then equation \eqref{52} has two nontrivial solutions.
\end{enumroman}
\end{theorem}

For example, consider the equation
\begin{equation} \label{71}
- \Delta u + \left(\frac{1}{4 \pi |x|} \star u^2\right) u = \frac{\lambda\, |u| u}{1 + |u|} \quad \text{in } \R^3.
\end{equation}
The nonlinearity $f(t) = \lambda\, |t| t/(1 + |t|)$ satisfies
\[
|f(t)| \le |\lambda|\, |t|^{5/3} \quad \forall t \in \R
\]
since $|t|^{1/3} \le 2/3 + |t|/3 \le 1 + |t|$ by the Young's inequality, and $g(t) = f(t) - \lambda\, |t| t = - \lambda t^3/(1 + |t|)$ satisfies
\[
|g(t)| \le |\lambda|\, |t|^3 \quad \forall t \in \R.
\]
So the following corollary is immediate from Theorem \ref{Theorem 11}.

\begin{corollary}
Assume that $\lambda$ is not an eigenvalue of problem \eqref{55}.
\begin{enumroman}
\item If $\lambda > \lambda_1$, then equation \eqref{71} has a nontrivial solution.
\item If $\lambda > \lambda_2$, then equation \eqref{71} has two nontrivial solutions.
\end{enumroman}
\end{corollary}

Next we prove an existence result in the asymptotically scaled case. We write $f$ as in \eqref{64} and assume that
\begin{equation} \label{66}
|g(|x|,t)| \le a_5\, |t|^{q_5 - 1} + h(x) \quad \text{for a.a.\! } x \in \R^3 \text{ and all } t \in \R
\end{equation}
for some constant $a_5 > 0$, $q_5 \in (18/7,3)$, and $h \in L^r(\R^3)$ with $r \in [6/5,18/11)$. Since the eigenvalue problem \eqref{55} is not linear, here $\Phi$ does not have the classical saddle point geometry based on linear subspaces. First we show that $\Phi$ has a certain saddle point like geometry based on the scale invariance property of problem \eqref{55} (see Proposition \ref{Proposition 8}). Next we prove a new saddle point type theorem that can capture this geometry to produce a critical point (see Theorem \ref{Theorem 4}). Then we apply our saddle point theorem to prove the following existence result.

\begin{theorem} \label{Theorem 12}
Assume that \eqref{64} and \eqref{66} hold and that $\lambda$ is not an eigenvalue of problem \eqref{55}. Then equation \eqref{52} has a solution.
\end{theorem}

In particular, we have the following alternative.

\begin{corollary}
Either the equation
\[
- \Delta u + \left(\frac{1}{4 \pi |x|} \star u^2\right) u = \lambda\, |u| u \quad \text{in } \R^3
\]
has a nontrivial solution, or the equation
\[
- \Delta u + \left(\frac{1}{4 \pi |x|} \star u^2\right) u = \lambda\, |u| u + h(x) \quad \text{in } \R^3
\]
has a solution for all $h \in L^r(\R^3)$ with $r \in [6/5,18/11)$.
\end{corollary}

Next we prove the existence of a nontrivial solution in the superscaled case when \eqref{64} and \eqref{65} hold. We assume that $G(|x|,t) = \int_0^t g(|x|,\tau)\, d\tau$ satisfies
\begin{equation} \label{61}
c_0\, |t|^q \le G(|x|,t) \le \frac{1}{q}\, g(|x|,t)\, t \quad \text{for a.a.\! } x \in \R^3 \text{ and all } t \in \R
\end{equation}
for some constant $c_0 > 0$ and $q \in (4,6)$. A model equation is
\[
- \Delta u + \left(\frac{1}{4 \pi |x|} \star u^2\right) u = \lambda\, |u| u + |u|^{q-2}\, u \quad \text{in } \R^3,
\]
but we do not assume that $g$ is odd in $t$. For example, we may consider the equation
\[
- \Delta u + \left(\frac{1}{4 \pi |x|} \star u^2\right) u = \lambda\, |u| u + a\, (u^+)^{q-1} - b\, (u^-)^{q-1} \quad \text{in } \R^3,
\]
where $u^\pm = \max \set{\pm u,0}$ are the positive and negative parts of $u$, respectively, and $a, b > 0$ are different constants. We have the following theorem.

\begin{theorem} \label{Theorem 13}
Assume that \eqref{64}, \eqref{65}, and \eqref{61} hold and that $\lambda$ is not an eigenvalue of problem \eqref{55}. Then equation \eqref{52} has a nontrivial solution at a positive energy level.
\end{theorem}

When $\lambda > \lambda_1$, Theorem \ref{Theorem 10} \ref{Theorem 10.ii} shows that $\Phi$ does not have the mountain pass geometry, and it also does not have the classical linking geometry based on linear subspaces since the eigenvalue problem \eqref{55} is not linear. Furthermore, Ljusternik-Schnirelman theory cannot be applied to this equation since we do not assume that $g$ is an odd function of $t$. So the existence of a nontrivial solution does not follow from classical variational arguments. We will show that $\Phi$ has a certain nonlinear linking geometry based on the scaling property of problem \eqref{55} (see Proposition \ref{Proposition 9}), and prove a new linking theorem based on this geometry to produce a nontrivial critical point (see Theorem \ref{Theorem 5}).

Next we remove the assumption that $\lambda$ is not an eigenvalue in Theorem \ref{Theorem 11} when $G$ is negative and in Theorem \ref{Theorem 13}, by introducing a notion of local linking based on scaling (see Definition \ref{Definition 2}) and proving a new abstract result on the existence of a nontrivial critical point based on this notion (see Theorem \ref{Theorem 17} and Corollary \ref{Corollary 2}). For the sake of simplicity, we assume that $f$ is independent of $x$ and consider the equation
\begin{equation} \label{100}
- \Delta u + \left(\frac{1}{4 \pi |x|} \star u^2\right) u = f(u) \quad \text{in } \R^3,
\end{equation}
where $f$ is a continuous function on $\R$ satisfying the growth condition
\begin{equation} \label{153}
|f(t)| \le a_1\, |t|^{q_1 - 1} + a_2\, |t|^{q_2 - 1} \quad \forall t \in \R
\end{equation}
for some constants $a_1, a_2 > 0$ and $18/7 < q_1 < q_2 < 6$. As before, we write
\begin{equation} \label{154}
f(t) = \lambda\, |t| t + g(t),
\end{equation}
and assume that
\begin{equation} \label{101}
|g(t)| \le a_3\, |t|^{q_3 - 1} + a_4\, |t|^{q_4 - 1} \quad \forall t \in \R
\end{equation}
for some constants $a_3, a_4 > 0$ and $3 < q_3 < q_4 < 6$. Let $G(t) = \int_0^t g(\tau)\, d\tau$. We have the following theorem in the subscaled case.

\begin{theorem} \label{Theorem 19}
Assume that $f$ satisfies \eqref{153} with $18/7 < q_1 < q_2 < 3$, \eqref{154} and \eqref{101} hold, and that $G(t) < 0$ for all $t \in \R \setminus \set{0}$.
\begin{enumroman}
\item If $\lambda > \lambda_1$, then equation \eqref{100} has a nontrivial solution.
\item If $\lambda > \lambda_2$, then equation \eqref{100} has two nontrivial solutions.
\end{enumroman}
\end{theorem}

In particular, equation \eqref{71} has a nontrivial solution for all $\lambda > \lambda_1$ and two nontrivial solutions for all $\lambda > \lambda_2$.

We have the following theorem in the superscaled case.

\begin{theorem} \label{Theorem 18}
Assume that \eqref{154} and \eqref{101} hold and that
\begin{equation} \label{102}
c_0\, |t|^q \le G(t) \le \frac{1}{q}\, g(t)\, t \quad \forall t \in \R
\end{equation}
for some constant $c_0 > 0$ and $q \in (4,6)$. Then equation \eqref{100} has a nontrivial solution for all $\lambda \in \R$.
\end{theorem}

In particular, we have the following corollaries.

\begin{corollary}
The equation
\[
- \Delta u + \left(\frac{1}{4 \pi |x|} \star u^2\right) u = \lambda\, |u| u + |u|^{q-2}\, u \quad \text{in } \R^3
\]
has a nontrivial solution for all $\lambda \in \R$.
\end{corollary}

\begin{corollary}
The equation
\[
- \Delta u + \left(\frac{1}{4 \pi |x|} \star u^2\right) u = \lambda\, |u| u + a\, (u^+)^{q-1} - b\, (u^-)^{q-1} \quad \text{in } \R^3,
\]
where $u^\pm = \max \set{\pm u,0}$ and $a, b > 0$ are constants, has a nontrivial solution for all $\lambda \in \R$.
\end{corollary}

\subsubsection{Critical case} \label{1.2.2}

Next, in the spirit of Brezis and Nirenberg \cite{MR709644}, we consider the critical growth equation
\begin{equation} \label{79}
- \Delta u + \left(\frac{1}{4 \pi |x|} \star u^2\right) u = \lambda\, |u|^{q-2}\, u + u^5 \quad \text{in } \R^3,
\end{equation}
where $\lambda > 0$ and $q \in [3,6)$, which has the added difficulty of lack of compactness (see also Ruf and Gazzola \cite{MR1441856}). It was shown in Ianni and Ruiz \cite[Corollary 2.6]{MR2902293} that this equation has no nontrivial solution in $E(\R^3) \cap H^2_\loc(\R^3)$ when $\lambda = 0$. We will prove some multiplicity results when $\lambda > 0$.

For $q = 3$, we will show that equation \eqref{79} has $m$ distinct pairs of nontrivial solutions for all $\lambda$ in a suitably small left neighborhood of any eigenvalue of problem \eqref{55} with multiplicity $m \ge 1$.

\begin{theorem} \label{Theorem 25}
Let $\lambda_k = \cdots = \lambda_{k+m-1} < \lambda_{k+m}$ for some $k, m \ge 1$. Then $\exists \delta_k > 0$ such that the equation
\begin{equation} \label{185}
- \Delta u + \left(\frac{1}{4 \pi |x|} \star u^2\right) u = \lambda\, |u| u + u^5 \quad \text{in } \R^3
\end{equation}
has $m$ distinct pairs of nontrivial solutions at positive energy levels for all $\lambda \in (\lambda_k - \delta_k,\lambda_k)$.
\end{theorem}

In particular, we have the following corollary when $m = 1$.

\begin{corollary} \label{Corollary 6}
If $\lambda_k < \lambda_{k+1}$, then $\exists \delta_k > 0$ such that equation \eqref{185} has a pair of nontrivial solutions at a positive energy level for all $\lambda \in (\lambda_k - \delta_k,\lambda_k)$.
\end{corollary}

\begin{remark}
The existence of a single positive solution of equation \eqref{185} for some $\lambda > 0$ was proved in Liu et al.\! \cite[Theorem 1.4]{MR3912770}. In contrast, Corollary \ref{Corollary 6} gives a nontrivial solution for each $\lambda \in \bigcup_{k=1}^\infty (\lambda_k - \delta_k,\lambda_k)$.
\end{remark}

For $q \in (3,6)$, we will show that equation \eqref{79} has arbitrarily many solutions for all sufficiently large $\lambda$.

\begin{theorem} \label{Theorem 15}
Let $q \in (3,6)$. Then for any $m \ge 1$, $\exists \Lambda_m > 0$ such that equation \eqref{79} has $m$ distinct pairs of nontrivial solutions at positive energy levels for all $\lambda > \Lambda_m$. In particular, the number of solutions goes to infinity as $\lambda \to \infty$.
\end{theorem}

\begin{remark}
The existence of a single positive solution of equation \eqref{79} for $q \in (3,4]$ and sufficiently large $\lambda > 0$ and for $q \in (4,6)$ and all $\lambda > 0$ was proved in Liu et al.\! \cite[Theorem 1.1]{MR3912770}.
\end{remark}

Proofs of Theorem \ref{Theorem 25} and Theorem \ref{Theorem 15} will be based on some new abstract multiplicity results for scaled functionals that only satisfy a local \PS{} condition (see Theorem \ref{Theorem 14} and Corollary \ref{Corollary 1}). The case where $q \in (18/7,3)$ will be taken up at the end of the next section (see Theorem \ref{Theorem 26}).

\subsubsection{Equations subscaled near zero} \label{1.2.3}

Finally we consider the class of equations
\begin{equation} \label{300}
- \Delta u + \left(\frac{1}{4 \pi |x|} \star u^2\right) u = |u|^{\sigma - 2}\, u + g(|x|,u) \quad \text{in } \R^3,
\end{equation}
where $\sigma \in (18/7,3)$ and $g$ is a Carath\'{e}odory function on $[0,\infty) \times \R$ satisfying the growth condition
\begin{equation} \label{301}
|g(|x|,t)| \le a_6\, |t|^{q_6 - 1} + a_7\, |t|^{q_7 - 1} \quad \text{for a.a.\! } x \in \R^3 \text{ and all } t \in \R
\end{equation}
for some constants $a_6, a_7 > 0$ and $\sigma < q_6 < q_7 \le 6$. We assume that $g$ is odd in $t$, i.e.,
\begin{equation} \label{302}
g(|x|,-t) = - g(|x|,t) \quad \text{for a.a.\! } x \in \R^3 \text{ and all } t \in \R,
\end{equation}
and that its primitive $G(|x|,t) = \int_0^t g(|x|,\tau)\, d\tau$ satisfies
\begin{equation} \label{305}
G(|x|,t) \le C\, |t|^{\widetilde{\sigma}} \quad \text{for a.a.\! } x \in \R^3 \text{ and all } t \in \R
\end{equation}
for some constant $C > 0$ and $\widetilde{\sigma} \in (18/7,3)$. We will show that this class of equations have infinitely many solutions at negative energy levels when the associated variational functional
\begin{multline*}
\Phi(u) = \frac{1}{2} \int_{\R^3} |\nabla u|^2\, dx + \frac{1}{16 \pi} \int_{\R^3} \int_{\R^3} \frac{u^2(x)\, u^2(y)}{|x - y|}\, dx\, dy - \frac{1}{\sigma} \int_{\R^3} |u|^\sigma\, dx - \int_{\R^3} G(|x|,u)\, dx,\\[7.5pt]
u \in E_r(\R^3)
\end{multline*}
satisfies the \PS{} condition at negative levels.

\begin{theorem} \label{Theorem 20}
Assume that $\sigma \in (18/7,3)$, \eqref{301}--\eqref{305} hold, and that $\Phi$ satisfies the {\em \PS{c}} condition for all $c < 0$. Then equation \eqref{300} has a sequence of solutions $\seq{u_k}$ such that $\Phi(u_k) < 0$ for all $k$ and $\Phi(u_k) \nearrow 0$.
\end{theorem}

In particular, we have the following corollary.

\begin{corollary} \label{Corollary 3}
The equation
\[
- \Delta u + \left(\frac{1}{4 \pi |x|} \star u^2\right) u = |u|^{\sigma - 2}\, u \quad \text{in } \R^3
\]
has infinitely many solutions at negative energy levels for all $\sigma \in (18/7,3)$.
\end{corollary}

\begin{remark}
The existence of a single positive solution of the equation in Corollary \ref{Corollary 3} was proved in Ruiz \cite[Theorem 1.3]{MR2679375}.
\end{remark}

Let
\[
f(|x|,t) = |t|^{\sigma - 2}\, t + g(|x|,t), \quad (x,t) \in \R^3 \times \R.
\]
If
\begin{equation} \label{307}
f(|x|,t_0) \equiv 0
\end{equation}
for some $t_0 > 0$, then it is enough to assume that \eqref{301}--\eqref{305} hold for $|t| \le t_0$.

\begin{theorem} \label{Theorem 29}
Assume that $\sigma \in (18/7,3)$, \eqref{307} holds, and that \eqref{301}--\eqref{305} hold for $|t| \le t_0$. Then equation \eqref{300} has a sequence of solutions $\seq{u_k}$ such that $|u_k| < t_0$ and $\Phi(u_k) < 0$ for all $k$ and $\Phi(u_k) \nearrow 0$.
\end{theorem}

We have the following immediate corollaries.

\begin{corollary} \label{Corollary 4}
The equation
\[
- \Delta u + \left(\frac{1}{4 \pi |x|} \star u^2\right) u = |u|^{\sigma - 2}\, u - |u|^{q-2}\, u \quad \text{in } \R^3
\]
has infinitely many solutions at negative energy levels for all $\sigma \in (18/7,3)$ and $q \in (\sigma,\infty)$.
\end{corollary}

\begin{remark}
This corollary holds for $\sigma = 8/3$ and $q = 10/3$, which is the case of the so-called Thomas--Fermi--Dirac--von\;Weizs\"acker model, as mentioned at the beginning of this introduction.
\end{remark}

\begin{corollary} \label{Corollary 5}
The equation
\[
- \Delta u + \left(\frac{1}{4 \pi |x|} \star u^2\right) u = |u|^{\sigma - 2}\, u + \lambda\, |u| u - |u|^{q-2}\, u \quad \text{in } \R^3
\]
has infinitely many solutions at negative energy levels for all $\sigma \in (18/7,3)$, $\lambda \in \R$, and $q \in (3,\infty)$.
\end{corollary}

\begin{corollary}
The equation
\[
- \Delta u + \left(\frac{1}{4 \pi |x|} \star u^2\right) u = |u|^{\sigma - 2}\, u - \lambda\, |u| u + |u|^{q-2}\, u \quad \text{in } \R^3
\]
has infinitely many solutions at negative energy levels for all $\sigma \in (18/7,3)$, $\lambda \ge 2$, and $q \in (3,\infty)$.
\end{corollary}

The assumption \eqref{305} in Theorem \ref{Theorem 20} implies that the functional $\Phi$ is bounded from below, and this fact is used in the proof to show that certain minimax levels are finite. We also have the following theorem where the functional is unbounded from below when $\lambda > \lambda_1$, but possesses infinitely many critical points at negative levels nonetheless.

\begin{theorem} \label{Theorem 24}
If $\lambda$ is not an eigenvalue of problem \eqref{55}, then the equation
\[
- \Delta u + \left(\frac{1}{4 \pi |x|} \star u^2\right) u = |u|^{\sigma - 2}\, u + \lambda\, |u| u \quad \text{in } \R^3
\]
has infinitely many solutions at negative energy levels for all $\sigma \in (18/7,3)$.
\end{theorem}

Next we consider the critical growth equation
\begin{equation} \label{179}
- \Delta u + \left(\frac{1}{4 \pi |x|} \star u^2\right) u = \mu\, |u|^{\sigma - 2}\, u + u^5 \quad \text{in } \R^3,
\end{equation}
where $\mu > 0$ and $\sigma \in (18/7,3)$. We will show that this equation has infinitely many solutions at negative energy levels for all sufficiently small $\mu$.

\begin{theorem} \label{Theorem 26}
Let $\sigma \in (18/7,3)$. Then $\exists \mu^\ast > 0$ such that equation \eqref{179} has infinitely many solutions at negative energy levels for all $\mu \in (0,\mu^\ast)$.
\end{theorem}

\begin{remark}
The existence of a single positive solution of equation \eqref{179} for sufficiently small $\mu > 0$ was proved in Liu et al.\! \cite[Theorem 1.5]{MR3912770}.
\end{remark}

Finally we consider the case where the term $|u|^{\sigma - 2}\, u$ has the opposite sign. We will prove the following theorem in the subcritical case.

\begin{theorem} \label{Theorem 27}
If $\sigma \in (18/7,3)$ and $\lambda > \lambda_k$ is not an eigenvalue of problem \eqref{55}, then the equation
\begin{equation} \label{201}
- \Delta u + \left(\frac{1}{4 \pi |x|} \star u^2\right) u = \lambda\, |u| u - |u|^{\sigma - 2}\, u \quad \text{in } \R^3
\end{equation}
has $k$ distinct pairs of nontrivial solutions at positive energy levels.
\end{theorem}

We have the following theorem in the critical case.

\begin{theorem} \label{Theorem 28}
If $\sigma \in (18/7,3)$ and $\lambda > \lambda_k$, then $\exists \mu^\ast > 0$ such that the equation
\begin{equation} \label{202}
- \Delta u + \left(\frac{1}{4 \pi |x|} \star u^2\right) u = \lambda\, |u| u - \mu\, |u|^{\sigma - 2}\, u + u^5 \quad \text{in } \R^3
\end{equation}
has $k$ distinct pairs of nontrivial solutions at positive energy levels for all $\mu \in (0,\mu^\ast)$.
\end{theorem}

\subsection{Open problems}

Our present work is part of a broader programme that includes many other scaled equations involving other differential operators, as well as generalizations in multiple directions of the results obtained here for Schr\"{o}dinger--Poisson--Slater equations. Within the framework of the latter, we list a few questions, some of which we will hopefully address in future contributions.

\subsubsection{Supercritical equations}

Can the bounded-truncation technique introduced by Berestycki and Lions in \cite{MR695535,MR695536} be adapted to extend our existence and multiplicity results to Sobolev-supercritical nonlinearities of the form
\[
f(|x|,u) = |u|^{p-2}\, u - |u|^{q-2}\, u
\]
with $q > p \ge 6$?

\subsubsection{More general kernels and higher dimensions}

A way to capture Sobolev-supercritical nonlinear regimes is to set all our equations in $\R^N$ and replace the so-called Hartree term
\[
\int_{\R^3} \int_{\R^3} \frac{|u(x)|^2\, |u(y)|^2}{|x - y|}\, dx\, dy
\]
that we have considered here with the more general form
\[
\int_{\R^N} \int_{\R^N} \frac{|u(x)|^p\, |u(y)|^p}{|x - y|^{N - \alpha}}\, d\mu(x)\, d\mu(y),
\]
where $p > 1$, $\alpha \in (0,N)$, and $d\mu(x) = \rho(x)\, dx$ with $\rho$, e.g., locally bounded and homogeneous of some order $k \ge 0$. Results in this direction can be obtained by combining techniques developed in this paper with radial and nonradial embeddings of Coulomb-Sobolev spaces studied in the case $\rho \equiv 1$ by Bellazzini et al.\! in \cite{MR3852465} and Mercuri et al.\! in \cite{MR3568051}. In \cite{MR3852465,MR3568051} it was shown that these spaces embed into $L^q$ spaces in an optimal range of the parameters $N$, $\alpha$, and $p$ that allows the case $q > 2^\ast$. Scaling properties for related equations involving $d\mu(x) = \rho(x)\, dx$ have been studied by Dutko et al.\! in \cite{MR4292779}.

\section{Abstract theory} \label{Section 2}

\subsection{Scaled operators}

Let $W$ be a reflexive Banach space with a continuous mapping $W \times [0,\infty) \to W,\, (u,t) \mapsto u_t$ satisfying
\begin{enumerate}
\item[$(H_1)$] $(u_{t_1})_{t_2} = u_{t_1 t_2}$ for all $u \in W$ and $t_1, t_2 \ge 0$,
\item[$(H_2)$] $(\tau u)_t = \tau u_t$ for all $u \in W$, $\tau \in \R$, and $t \ge 0$,
\item[$(H_3)$] $u_0 = 0$ and $u_1 = u$ for all $u \in W$,
\item[$(H_4)$] $u_t$ is bounded on bounded sets in $W \times [0,\infty)$,
\item[$(H_5)$] $\exists s > 0$ such that $\norm{u_t} = \O(t^s)$ as $t \to \infty$, uniformly in $u$ on bounded sets.
\end{enumerate}
We will refer to this mapping as a scaling on $W$.

Taking $u = 0$ and $\tau = 0$ in $(H_2)$ gives
\begin{equation} \label{15}
0_t = 0 \quad \forall t \ge 0.
\end{equation}
By $(H_1)$ and $(H_3)$,
\begin{equation} \label{5}
(u_t)_{1/t} = u_1 = u \quad \forall u \in W,\, t > 0,
\end{equation}
in particular,
\begin{equation} \label{7}
u_t \ne 0 \quad \forall u \in W \setminus \set{0},\, t > 0.
\end{equation}

Denote by $W^\ast$ the dual of $W$. Recall that $q \in C(W,W^\ast)$ is a potential operator if there is a functional $Q \in C^1(W,\R)$, called a potential for $q$, such that $Q' = q$. By replacing $Q$ with $Q - Q(0)$ if necessary, we may assume that $Q(0) = 0$.

\begin{definition}
We denote by $\A_s$ the class of odd potential operators $q \in C(W,W^\ast)$ that map bounded sets into bounded sets and satisfy
\begin{equation} \label{1}
q(u_t) v_t = t^s q(u) v \quad \forall u, v \in W,\, t \ge 0.
\end{equation}
We will refer to an element of $\A_s$ as a scaled operator.
\end{definition}

We have the following proposition.

\begin{proposition} \label{Proposition 1}
If $q \in C(W,W^\ast)$ is a potential operator and $Q$ is its potential with $Q(0) = 0$, then
\[
Q(u) = \int_0^1 q(\tau u) u\, d\tau \quad \forall u \in W.
\]
In particular, $Q$ is even if $q$ is odd, and $Q$ is bounded on bounded sets if $q$ maps bounded sets into bounded sets. If $q \in \A_s$, then $Q$ satisfies
\[
Q(u_t) = t^s Q(u) \quad \forall u \in W,\, t \ge 0.
\]
\end{proposition}

\begin{proof}
We have
\[
Q(u) = \int_0^1 \frac{d}{d\tau}\, Q(\tau u)\, d\tau = \int_0^1 Q'(\tau u) u\, d\tau = \int_0^1 q(\tau u) u\, d\tau.
\]
If $q \in \A_s$, then this gives
\[
Q(u_t) = \int_0^1 q(\tau u_t) u_t\, d\tau = \int_0^1 \frac{1}{\tau}\, q((\tau u)_t) (\tau u)_t\, d\tau = t^s \int_0^1 q(\tau u) u\, d\tau = t^s Q(u)
\]
by $(H_2)$ and \eqref{1}.
\end{proof}

We consider the question of existence and multiplicity of solutions to nonlinear operator equations of the form
\begin{equation} \label{2}
\As(u) = f(u)
\end{equation}
in $W^\ast$, where $\As \in \A_s$ satisfies
\begin{enumerate}
\item[$(H_6)$] $\As(u) u > 0$ for all $u \in W \setminus \set{0}$,
\item[$(H_7)$] every sequence $\seq{u_j}$ in $W$ such that $u_j \wto u$ and $\As(u_j)(u_j - u) \to 0$ has a subsequence that converges strongly to $u$,
\end{enumerate}
and $f \in C(W,W^\ast)$ is a potential operator. By Proposition \ref{Proposition 1}, the potential $I_s$ of $\As$ with $I_s(0) = 0$ given by
\begin{equation} \label{3}
I_s(u) = \int_0^1 \As(\tau u) u\, d\tau
\end{equation}
is even, bounded on bounded sets, and satisfies
\begin{equation} \label{12}
I_s(u_t) = t^s I_s(u) \quad \forall u \in W,\, t \ge 0.
\end{equation}
By \eqref{3} and $(H_6)$, $I_s(u) > 0$ for all $u \in W \setminus \set{0}$. Let
\[
\Phi(u) = I_s(u) - F(u), \quad u \in W,
\]
where $F$ is the potential of $f$ with $F(0) = 0$ given by
\[
F(u) = \int_0^1 f(\tau u) u\, d\tau.
\]
Then
\[
\Phi'(u) = I_s'(u) - F'(u) = \As(u) - f(u),
\]
so solutions of equation \eqref{2} coincide with critical points of the $C^1$-functional $\Phi$. We note that $\Phi(0) = 0$ and that $\Phi$ is even when $f$ is odd.

The following proposition shows that every bounded \PS{} sequence of $\Phi$ has a convergent subsequence when $f$ is compact.

\begin{proposition} \label{Proposition 3}
If $f$ is a compact operator, then every bounded sequence $\seq{u_j}$ in $W$ such that $\Phi'(u_j) \to 0$ has a convergent subsequence.
\end{proposition}

\begin{proof}
Since $W$ is reflexive, $\seq{u_j}$ converges weakly to some $u \in W$ for a renamed subsequence. Since $f$ is compact, $f(u_j)$ converges in $W^\ast$ for a further subsequence. Then
\[
\As(u_j)(u_j - u) = (\Phi'(u_j) + f(u_j))(u_j - u) \to 0,
\]
so $\seq{u_j}$ has a subsequence that converges strongly to $u$ by $(H_7)$.
\end{proof}

\subsection{Scaled eigenvalue problems}

First we consider the eigenvalue problem
\begin{equation} \label{4}
\As(u) = \lambda \Bs(u)
\end{equation}
in $W^\ast$, where $\As, \Bs \in \A_s$ satisfy $(H_6)$, $(H_7)$, and
\begin{enumerate}
\item[$(H_8)$] $\Bs(u) u > 0$ for all $u \in W \setminus \set{0}$,
\item[$(H_9)$] if $u_j \wto u$ in $W$, then $\Bs(u_j) \to \Bs(u)$ in $W^\ast$,
\end{enumerate}
and $\lambda \in \R$. We say that $\lambda$ is an eigenvalue if there is a $u \in W \setminus \set{0}$, called an eigenfunction associated with $\lambda$, satisfying equation \eqref{4}. Then $u_t$ is also an eigenfunction associated with $\lambda$ for any $t > 0$ since $u_t \ne 0$ by \eqref{7} and
\[
\As(u_t) v = \As(u_t) (v_{1/t})_t = t^s \As(u) v_{1/t} = t^s \lambda \Bs(u) v_{1/t} = \lambda \Bs(u_t) (v_{1/t})_t = \lambda \Bs(u_t) v
\]
for all $v \in W$ by \eqref{5}. Moreover,
\[
\lambda = \frac{\As(u) u}{\Bs(u) u} > 0
\]
by $(H_6)$ and $(H_8)$. We will call the set $\sigma(\As,\Bs)$ of all eigenvalues the spectrum of the pair of scaled operators $(\As,\Bs)$. We will refer to equation \eqref{4} as a scaled eigenvalue problem.

\subsubsection{Variational setting}

We have $I_s'(0) = \As(0) = 0$ since $\As$ is odd, so the origin is a critical point of $I_s$. It is the only critical point of $I_s$ since
\[
I_s'(u) u = \As(u) u > 0 \quad \forall u \in W \setminus \set{0}
\]
by $(H_6)$. So $I_s(0) = 0$ is the only critical value of $I_s$, and hence it follows from the implicit function theorem that
\[
\M_s = \bgset{u \in W : I_s(u) = 1}
\]
is a $C^1$-Finsler manifold. Moreover, $\M_s$ is complete and symmetric since $I_s$ is continuous and even. We assume that
\begin{enumerate}
\item[$(H_{10})$] $I_s$ is coercive, i.e., $I_s(u) \to \infty$ as $\norm{u} \to \infty$,
\item[$(H_{11})$] the equation $I_s(tu) = 1$ has a unique positive solution $t$ for each $u \in W \setminus \set{0}$.
\end{enumerate}
We note that $(H_{10})$ implies that $\M_s$ is a bounded manifold, while $(H_{11})$ implies that the ray $\set{tu : t \ge 0}$ through $u$ intersects $\M_s$ at exactly one point for each $u \in W \setminus \set{0}$.

For $u \in W \setminus \set{0}$, set
\begin{equation} \label{13}
t_u = \frac{1}{I_s(u)^{1/s}}, \quad \widetilde{u} = u_{t_u}.
\end{equation}
Then
\[
I_s(\widetilde{u}) = t_u^s\, I_s(u) = 1
\]
by \eqref{12} and hence $\widetilde{u} \in \M_s$. For $u \in \M_s$, $t_u = 1$ and hence $\widetilde{u} = u$ by $(H_3)$. We will refer to the continuous mapping $\pi : W \setminus \set{0} \to \M_s,\, u \mapsto \widetilde{u}$ as the projection onto $\M_s$. We have
\begin{equation} \label{18}
u = \widetilde{u}_{1/t_u}
\end{equation}
by \eqref{5}.

By Proposition \ref{Proposition 1}, the potential $J_s$ of $\Bs$ with $J_s(0) = 0$ given by
\begin{equation} \label{8}
J_s(u) = \int_0^1 \Bs(\tau u) u\, d\tau
\end{equation}
is even, bounded on bounded sets, and satisfies
\begin{equation} \label{6}
J_s(u_t) = t^s J_s(u) \quad \forall u \in W,\, t \ge 0.
\end{equation}

\begin{proposition} \label{Proposition 4}
If $u_j \wto u$ in $W$, then $J_s(u_j) \to J_s(u)$.
\end{proposition}

\begin{proof}
By \eqref{8},
\[
J_s(u_j) = \int_0^1 \Bs(\tau u_j) u_j\, d\tau.
\]
For all $\tau \in [0,1]$, $\tau u_j \wto \tau u$ and hence $\Bs(\tau u_j) \to \Bs(\tau u)$ by $(H_9)$, so $\Bs(\tau u_j) u_j \to \Bs(\tau u) u$. Moreover, $\Bs(\tau u_j) u_j$ is bounded since $\seq{u_j}$ is a bounded sequence and $\Bs$ maps bounded sets into bounded sets. So
\[
\int_0^1 \Bs(\tau u_j) u_j\, d\tau \to \int_0^1 \Bs(\tau u) u\, d\tau = J_s(u). \QED
\]
\end{proof}

Finally we make the following structural assumption about the operators $\As$ and $\Bs$:
\begin{enumerate}
\item[$(H_{12})$] every solution of equation \eqref{4} satisfies
    \[
    I_s(u) = \lambda\, J_s(u).
    \]
\end{enumerate}
In many applications, $(H_{12})$ follows from a suitable Poho\v{z}aev type identity.

By \eqref{8} and $(H_8)$, $J_s(u) > 0$ for all $u \in W \setminus \set{0}$. So the functional
\[
\Psi(u) = \frac{1}{J_s(u)}, \quad u \in W \setminus \set{0}
\]
is positive and its restriction to $\M_s$,
\[
\widetilde{\Psi} = \restr{\Psi}{\M_s},
\]
is $C^1$. Since
\[
I_s'(u) = \As(u), \qquad \Psi'(u) = - \frac{J_s'(u)}{J_s(u)^2} = - \Psi(u)^2\, \Bs(u),
\]
the norm of $\widetilde{\Psi}'(u)$ as an element of the cotangent space $T_u^\ast \M_s$ at $u \in \M_s$ is given by
\begin{equation} \label{9}
\bgdnorm[u]{\widetilde{\Psi}'(u)} = \min_{\mu \in \R}\, \bgdnorm{\mu I_s'(u) + \Psi'(u)} = \min_{\mu \in \R}\, \bgdnorm{\mu \As(u) - \widetilde{\Psi}(u)^2\, \Bs(u)},
\end{equation}
where $\dnorm{\cdot}$ is the norm in $W^\ast$ (see, e.g., Perera et al.\! \cite[Proposition 3.54]{MR2640827}).

\begin{proposition} \label{Proposition 12}
Eigenvalues of problem \eqref{4} coincide with critical values of $\widetilde{\Psi}$, i.e., $\lambda$ is an eigenvalue if and only if there is a $u \in \M_s$ such that $\widetilde{\Psi}'(u) = 0$ and $\widetilde{\Psi}(u) = \lambda$.
\end{proposition}

\begin{proof}
By \eqref{9}, $\widetilde{\Psi}'(u) = 0$ if and only if
\begin{equation} \label{10}
\mu \As(u) = \widetilde{\Psi}(u)^2\, \Bs(u)
\end{equation}
for some $\mu \in \R$. If \eqref{10} holds, then $\Bs(u) \ne 0$ by $(H_8)$ and $\widetilde{\Psi}(u) > 0$, so $\mu \ne 0$ and hence
\[
\mu I_s(u) = \widetilde{\Psi}(u)^2\, J_s(u)
\]
by $(H_{12})$. Since $u \in \M_s$, this gives $\mu = \widetilde{\Psi}(u)$, so \eqref{10} reduces to \eqref{4} with $\lambda = \widetilde{\Psi}(u)$. Conversely, if $\lambda$ is an eigenvalue and $u \in W \setminus \set{0}$ is an associated eigenfunction, then by replacing $u$ with $\widetilde{u}$ if necessary (see \eqref{13}), we may assume that $u \in \M_s$. Then $(H_{12})$ gives
\[
\lambda = \frac{I_s(u)}{J_s(u)} = \widetilde{\Psi}(u),
\]
so \eqref{4} implies \eqref{10} with $\mu = \widetilde{\Psi}(u)$.
\end{proof}

\begin{proposition} \label{Proposition 5}
$\widetilde{\Psi}$ satisfies the {\em \PS{}} condition.
\end{proposition}

\begin{proof}
Let $c \in \R$ and let $\seq{u_j} \subset \M_s$ be a \PS{c} sequence of $\widetilde{\Psi}$, i.e.,
\[
\widetilde{\Psi}(u_j) \to c, \qquad \bgdnorm[u_j]{\widetilde{\Psi}'(u_j)} \to 0.
\]
Since $\M_s$ is a bounded manifold by $(H_{10})$, $\seq{u_j}$ is bounded and hence converges weakly to some $u \in W$ for a renamed subsequence. Then $J_s(u_j) \to J_s(u)$ by Proposition \ref{Proposition 4}. Since
\[
\widetilde{\Psi}(u_j) = \frac{1}{J_s(u_j)} \to c,
\]
it follows that $c > 0$ and $J_s(u) > 0$, in particular, $u \ne 0$.

By \eqref{9}, $\bgdnorm[u_j]{\widetilde{\Psi}'(u_j)} \to 0$ implies that
\begin{equation} \label{11}
\mu_j\, \As(u_j) + \widetilde{\Psi}(u_j)^2\, \Bs(u_j) \to 0
\end{equation}
for some sequence $\seq{\mu_j} \subset \R$. Applying \eqref{11} to $u$ and noting that $\As(u_j) u$ is bounded since $\As$ maps bounded sets into bounded sets and
\[
\widetilde{\Psi}(u_j)^2\, \Bs(u_j) u \to c^2\, \Bs(u) u > 0
\]
by $(H_9)$ and $(H_8)$ shows that $\mu_j$ is bounded away from zero. Now applying \eqref{11} to $u_j - u$ shows that $\As(u_j)(u_j - u) \to 0$, so $u_j \to u$ for a further subsequence by $(H_7)$.
\end{proof}

Proposition \ref{Proposition 5} implies that the set
\[
E_\lambda = \bgset{u \in \M_s : \widetilde{\Psi}'(u) = 0 \ins{and} \widetilde{\Psi}(u) = \lambda}
\]
of eigenfunctions associated with $\lambda$ that lie on $\M_s$ is compact and that the spectrum
\[
\sigma(\As,\Bs) = \bgset{\lambda \in \R : \exists u \in \M_s \ins{such that} \widetilde{\Psi}'(u) = 0 \ins{and} \widetilde{\Psi}(u) = \lambda}
\]
is closed.

\subsubsection{Minimax eigenvalues}

We now construct an unbounded sequence of minimax eigenvalues of problem \eqref{4}. Although this can be done using the Krasnoselskii's genus, we prefer to use the $\Z_2$-cohomological index of Fadell and Rabinowitz \cite{MR0478189} in order to obtain nontrivial critical groups and construct linking sets.

\begin{definition}[Fadell and Rabinowitz \cite{MR0478189}] \label{Definition 1}
Let $\A$ denote the class of symmetric subsets of $W \setminus \set{0}$. For $A \in \A$, let $\overline{A} = A/\Z_2$ be the quotient space of $A$ with each $u$ and $-u$ identified, let $f : \overline{A} \to \RP^\infty$ be the classifying map of $\overline{A}$, and let $f^\ast : H^\ast(\RP^\infty) \to H^\ast(\overline{A})$ be the induced homomorphism of the Alexander-Spanier cohomology rings. The cohomological index of $A$ is defined by
\[
i(A) = \begin{cases}
0 & \text{if } A = \emptyset\\[5pt]
\sup \set{m \ge 1 : f^\ast(\omega^{m-1}) \ne 0} & \text{if } A \ne \emptyset,
\end{cases}
\]
where $\omega \in H^1(\RP^\infty)$ is the generator of the polynomial ring $H^\ast(\RP^\infty) = \Z_2[\omega]$.
\end{definition}

\begin{example}
The classifying map of the unit sphere $S^N$ in $\R^{N+1},\, N \ge 0$ is the inclusion $\RP^N \incl \RP^\infty$, which induces isomorphisms on the cohomology groups $H^l$ for $l \le N$, so $i(S^N) = N + 1$.
\end{example}

The following proposition summarizes the basic properties of the cohomological index.

\begin{proposition}[Fadell and Rabinowitz \cite{MR0478189}] \label{Proposition 7}
The index $i : \A \to \N \cup \set{0,\infty}$ has the following properties:
\begin{enumerate}
\item[$(i_1)$] Definiteness: $i(A) = 0$ if and only if $A = \emptyset$.
\item[$(i_2)$] Monotonicity: If there is an odd continuous map from $A$ to $B$ (in particular, if $A \subset B$), then $i(A) \le i(B)$. Thus, equality holds when the map is an odd homeomorphism.
\item[$(i_3)$] Dimension: $i(A) \le \dim W$.
\item[$(i_4)$] Continuity: If $A$ is closed, then there is a closed neighborhood $N \in \A$ of $A$ such that $i(N) = i(A)$. When $A$ is compact, $N$ may be chosen to be a $\delta$-neighborhood $N_\delta(A) = \set{u \in W : \dist{u}{A} \le \delta}$.
\item[$(i_5)$] Subadditivity: If $A$ and $B$ are closed, then $i(A \cup B) \le i(A) + i(B)$.
\item[$(i_6)$] Stability: If $\Sigma A$ is the suspension of $A \ne \emptyset$, obtained as the quotient space of $A \times [-1,1]$ with $A \times \set{1}$ and $A \times \set{-1}$ collapsed to different points, then $i(\Sigma A) = i(A) + 1$.
\item[$(i_7)$] Piercing property: If $C$, $C_0$, and $C_1$ are closed and $\varphi : C \times [0,1] \to C_0 \cup C_1$ is a continuous map such that $\varphi(-u,t) = - \varphi(u,t)$ for all $(u,t) \in C \times [0,1]$, $\varphi(C \times [0,1])$ is closed, $\varphi(C \times \set{0}) \subset C_0$, and $\varphi(C \times \set{1}) \subset C_1$, then $i(\varphi(C \times [0,1]) \cap C_0 \cap C_1) \ge i(C)$.
\item[$(i_8)$] Neighborhood of zero: If $U$ is a bounded closed symmetric neighborhood of $0$, then $i(\bdry{U}) = \dim W$.
\end{enumerate}
\end{proposition}

Let $\F$ denote the class of symmetric subsets of $\M_s$. For $k \ge 1$, let
\[
\F_k = \bgset{M \in \F : i(M) \ge k}
\]
and set
\[
\lambda_k := \inf_{M \in \F_k}\, \sup_{u \in M}\, \widetilde{\Psi}(u).
\]
We have the following theorem (see Perera et al.\! \cite[Proposition 3.52 and Proposition 3.53]{MR2640827}).

\begin{theorem} \label{Theorem 1}
Assume $(H_1)$--$(H_{12})$. Then $\lambda_k \nearrow \infty$ is a sequence of eigenvalues of \eqref{4}.
\begin{enumroman}
\item \label{Theorem 1.i} The first eigenvalue is given by
    \[
    \lambda_1 = \min_{u \in \M_s}\, \widetilde{\Psi}(u) > 0.
    \]
\item If $\lambda_k = \dotsb = \lambda_{k+m-1} = \lambda$, then $i(E_\lambda) \ge m$.
\item \label{Theorem 1.iii} If $\lambda_k < \lambda < \lambda_{k+1}$, then
    \[
    i(\widetilde{\Psi}^{\lambda_k}) = i(\M_s \setminus \widetilde{\Psi}_\lambda) = i(\widetilde{\Psi}^\lambda) = i(\M_s \setminus \widetilde{\Psi}_{\lambda_{k+1}}) = k,
    \]
    where $\widetilde{\Psi}^a = \bgset{u \in \M_s : \widetilde{\Psi}(u) \le a}$ and $\widetilde{\Psi}_a = \bgset{u \in \M_s : \widetilde{\Psi}(u) \ge a}$ for $a \in \R$.
\end{enumroman}
\end{theorem}

\begin{remark}
The spectrum $\sigma(\As,\Bs)$ may possibly contain points other than those of the sequence $\seq{\lambda_k}$.
\end{remark}

\subsubsection{Critical groups}

The variational functional associated with the scaled eigenvalue problem \eqref{4} is
\begin{equation} \label{16}
\Phi_\lambda(u) = I_s(u) - \lambda\, J_s(u), \quad u \in W.
\end{equation}
We note that
\begin{equation} \label{14}
\Phi_\lambda(u_t) = t^s\, \Phi_\lambda(u) \quad \forall u \in W,\, t \ge 0
\end{equation}
by \eqref{12} and \eqref{6}.

When $\lambda \notin \sigma(\As,\Bs)$, the origin is the only critical point of $\Phi_\lambda$ and its critical groups there are given by
\[
C^l(\Phi_\lambda,0) = H^l(\Phi_\lambda^0,\Phi_\lambda^0 \setminus \set{0}), \quad l \ge 0,
\]
where $\Phi_\lambda^0 = \bgset{u \in W : \Phi_\lambda(u) \le 0}$ and $H^\ast$ denotes cohomology with $\Z_2$ coefficients. We have the following theorem, where $\widetilde{H}^\ast$ denotes reduced cohomology.

\begin{theorem} \label{Theorem 2}
Assume $(H_1)$--$(H_{12})$ and let $\lambda \in \R \setminus \sigma(\As,\Bs)$.
\begin{enumroman}
\item \label{Theorem 2.i} If $\lambda < \lambda_1$, then $C^l(\Phi_\lambda,0) \isom \delta_{l0}\, \Z_2$.
\item \label{Theorem 2.ii} If $\lambda > \lambda_1$, then $C^l(\Phi_\lambda,0) \isom \widetilde{H}^{l-1}(\widetilde{\Psi}^\lambda)$, in particular, $C^0(\Phi_\lambda,0) = 0$.
\item \label{Theorem 2.iii} If $\lambda_k < \lambda < \lambda_{k+1}$, then $C^k(\Phi_\lambda,0) \ne 0$.
\end{enumroman}
In particular, $C^l(\Phi_\lambda,0)$ is nontrivial for some $l$.
\end{theorem}

\begin{proof}
If $u \in \Phi_\lambda^0$, then $u_t \in \Phi_\lambda^0$ for all $t \ge 0$ by \eqref{14}. This together with $(H_3)$ implies that $\Phi_\lambda^0$ contracts to $\set{0}$ via
\[
\Phi_\lambda^0 \times [0,1] \to \Phi_\lambda^0, \quad (u,\tau) \mapsto u_{1 - \tau}
\]
and $\Phi_\lambda^0 \setminus \set{0}$ deformation retracts to $\Phi_\lambda^0 \cap \M_s$ via
\[
(\Phi_\lambda^0 \setminus \set{0}) \times [0,1] \to \Phi_\lambda^0 \setminus \set{0}, \quad (u,\tau) \mapsto u_{1 - \tau + \tau\, t_u}
\]
(see \eqref{13}). It follows that
\[
C^l(\Phi_\lambda,0) \isom \begin{cases}
\delta_{l0}\, \Z_2 & \text{if } \Phi_\lambda^0 \cap \M_s = \emptyset\\[7.5pt]
\widetilde{H}^{l-1}(\Phi_\lambda^0 \cap \M_s) & \text{if } \Phi_\lambda^0 \cap \M_s \ne \emptyset
\end{cases}
\]
(see, e.g., Perera et al.\! \cite[Proposition 2.4]{MR2640827}). Since
\[
\Phi_\lambda(u) = 1 - \frac{\lambda}{\widetilde{\Psi}(u)}
\]
for $u \in \M_s$, $\Phi_\lambda^0 \cap \M_s = \widetilde{\Psi}^\lambda$, so we have
\[
C^l(\Phi_\lambda,0) \isom \begin{cases}
\delta_{l0}\, \Z_2 & \text{if } \widetilde{\Psi}^\lambda = \emptyset\\[7.5pt]
\widetilde{H}^{l-1}(\widetilde{\Psi}^\lambda) & \text{if } \widetilde{\Psi}^\lambda \ne \emptyset.
\end{cases}
\]
Since $\widetilde{\Psi}^\lambda = \emptyset$ if and only if $\lambda < \lambda_1$ by Theorem \ref{Theorem 1} \ref{Theorem 1.i}, \ref{Theorem 2.i} and \ref{Theorem 2.ii} follow from this. If $\lambda_k < \lambda < \lambda_{k+1}$, then $i(\widetilde{\Psi}^\lambda) = k$ by Theorem \ref{Theorem 1} \ref{Theorem 1.iii} and hence $\widetilde{H}^{k-1}(\widetilde{\Psi}^\lambda) \ne 0$ by Perera et al.\! \cite[Proposition 2.14 ({\em iv})]{MR2640827}, so \ref{Theorem 2.iii} follows from \ref{Theorem 2.ii}.
\end{proof}

\subsubsection{Perturbed equations}

Now we consider the perturbed equation
\[
\As(u) = \lambda \Bs(u) + g(u)
\]
in $W^\ast$, where $\As, \Bs \in \A_s$ satisfy $(H_6)$--$(H_{12})$, $g \in C(W,W^\ast)$ is a compact potential operator satisfying
\begin{equation} \label{17}
g(u_t) v_t = \o(t^s) \norm{v} \text{ as } t \to 0
\end{equation}
uniformly in $u$ on bounded sets for all $v \in W$, and $\lambda \in \R$. The variational functional associated with this equation is
\[
\Phi(u) = \Phi_\lambda(u) - G(u), \quad u \in W,
\]
where $\Phi_\lambda$ is as in \eqref{16} and
\[
G(u) = \int_0^1 g(\tau u) u\, d\tau
\]
is the potential of $g$ with $G(0) = 0$.

For any $v \in W$ and $t > 0$,
\begin{equation} \label{39}
g(0) v = g(0_t) (v_{1/t})_t = \o(t^s) \|v_{1/t}\| = \o(1) \text{ as } t \to 0
\end{equation}
by \eqref{15}, \eqref{5}, \eqref{17}, and $(H_5)$. So $g(0) = 0$ and hence the origin is a critical point of $\Phi$. We have the following theorem on its critical groups there.

\begin{theorem} \label{Theorem 3}
Assume $(H_1)$--$(H_{12})$ and \eqref{17}, and let $\lambda \in \R \setminus \sigma(\As,\Bs)$. Then the origin is an isolated critical point of $\Phi$ and
\[
C^l(\Phi,0) \isom C^l(\Phi_\lambda,0) \quad \forall l \ge 0.
\]
In particular, $C^l(\Phi,0)$ is nontrivial for some $l$.
\end{theorem}

The proof of this theorem will be based on the invariance of critical groups under homotopies that preserve the isolatedness of the critical point as stated in the following proposition (see Chang and Ghoussoub \cite{MR1422006} or Corvellec and Hantoute \cite{MR1926378}).

\begin{proposition} \label{Proposition 2}
Let $\Phi_\tau,\, \tau \in [0,1]$ be a family of $C^1$-functionals on a Banach space $W$ and let $u_0 \in W$ be a critical point of each $\Phi_\tau$. Assume that there is a closed neighborhood $U$ of $u_0$ such that
\begin{enumroman}
\item the map $[0,1] \to C^1(U,\R),\, \tau \mapsto \Phi_\tau$ is continuous,
\item $U$ contains no other critical point of any $\Phi_\tau$,
\item each $\Phi_\tau$ satisfies the {\em \PS{}} condition on $U$.
\end{enumroman}
Then
\[
C^l(\Phi_0,u_0) \isom C^l(\Phi_1,u_0) \quad \forall l \ge 0.
\]
\end{proposition}

\begin{proof}[Proof of Theorem \ref{Theorem 3}]
We apply Proposition \ref{Proposition 2} to the family of functionals
\[
\Phi_\tau(u) = \Phi_\lambda(u) - (1 - \tau)\, G(u), \quad u \in W,\, \tau \in [0,1]
\]
in a sufficiently small closed neighborhood of the origin, noting that $\Phi_0 = \Phi$ and $\Phi_1 = \Phi_\lambda$. Since $B_s$ and $g$ are compact, each $\Phi_\tau$ satisfies the \PS{} condition on bounded sets by Proposition \ref{Proposition 3}.

It only remains to verify that there is no nonzero critical point of any $\Phi_\tau$ in a sufficiently small neighborhood of the origin. Suppose this is not the case. Then there are sequences $\seq{u_j} \subset W \setminus \set{0}$ and $\seq{\tau_j} \subset [0,1]$ such that $u_j \to 0$ and $\Phi_{\tau_j}'(u_j) = 0$. Set
\[
t_j = t_{u_j} = \frac{1}{I_s(u_j)^{1/s}}, \quad \widetilde{u}_j = (u_j)_{t_j}, \quad \widetilde{t}_j = \frac{1}{t_j} = I_s(u_j)^{1/s}
\]
(see \eqref{13}). Then $\widetilde{u}_j \in \M_s$ and
\[
u_j = (\widetilde{u}_j)_{\widetilde{t}_j}
\]
by \eqref{18}. Since $\M_s$ is a bounded manifold by $(H_{10})$, $\seq{\widetilde{u}_j}$ is bounded and hence converges weakly to some $\widetilde{u} \in W$ for a renamed subsequence. Since $u_j \to 0$, $I_s(u_j) \to I_s(0) = 0$ and hence $\widetilde{t}_j \to 0$.

Since $\Phi_{\tau_j}'(u_j) = 0$ and $u_j = (\widetilde{u}_j)_{\widetilde{t}_j}$,
\[
\As((\widetilde{u}_j)_{\widetilde{t}_j}) v_{\widetilde{t}_j} = \lambda \Bs((\widetilde{u}_j)_{\widetilde{t}_j}) v_{\widetilde{t}_j} + (1 - \tau_j)\, g((\widetilde{u}_j)_{\widetilde{t}_j}) v_{\widetilde{t}_j}
\]
for any $v \in W$. Since $\As, \Bs \in \A_s$, $g$ satisfies \eqref{17}, $\widetilde{t}_j \to 0$, and $\seq{\widetilde{u}_j}$ is bounded, this reduces to
\[
\widetilde{t}_j^s\, \As(\widetilde{u}_j) v = \lambda \widetilde{t}_j^s\, \Bs(\widetilde{u}_j) v + \o(\widetilde{t}_j^s) \norm{v},
\]
or
\begin{equation} \label{19}
\As(\widetilde{u}_j) v = \lambda \Bs(\widetilde{u}_j) v + \o(1) \norm{v}.
\end{equation}
Taking $v = \widetilde{u}_j - \widetilde{u}$ and noting that $\Bs(\widetilde{u}_j) \to \Bs(\widetilde{u})$ by $(H_9)$ shows that $\As(\widetilde{u}_j)(\widetilde{u}_j - \widetilde{u}) \to 0$, so $\widetilde{u}_j \to \widetilde{u}$ for a further subsequence by $(H_7)$. Passing to the limit in \eqref{19} now gives
\[
\As(\widetilde{u}) = \lambda \Bs(\widetilde{u})
\]
since $v$ is arbitrary. Since $\widetilde{u}_j \in \M_s$ and $\M_s$ is closed, $\widetilde{u} \in \M_s$ and hence $\widetilde{u} \ne 0$. So $\lambda \in \sigma(\As,\Bs)$, a contradiction.
\end{proof}

\subsubsection{Subscaled equations}

Now we prove some existence and multiplicity results for the equation
\begin{equation} \label{34}
\As(u) = f(u)
\end{equation}
in $W^\ast$, where $\As \in \A_s$ satisfies $(H_6)$, $(H_7)$, and $(H_{10})$, and $f \in C(W,W^\ast)$ is a compact potential operator satisfying
\begin{equation} \label{35}
f(u_t) v_t = \o(t^s) \norm{v} \text{ as } t \to \infty
\end{equation}
uniformly in $u$ on bounded sets for all $v \in W$. We will refer to this equation as a subscaled equation.

The variational functional associated with equation \eqref{34} is
\[
\Phi(u) = I_s(u) - F(u), \quad u \in W,
\]
where $I_s$ is as in \eqref{3} and
\[
F(u) = \int_0^1 f(\tau u) u\, d\tau
\]
is the potential of $f$ with $F(0) = 0$. Since the compact operator $f$ maps bounded sets into precompact, and hence bounded, sets, $F$ is bounded on bounded sets. Since $I_s$ is also bounded on bounded sets, then so is $\Phi$.

\begin{lemma} \label{Lemma 2}
If \eqref{35} holds, then $\Phi$ is coercive.
\end{lemma}

\begin{proof}
Let $\norm{u} \to \infty$ and set
\begin{equation} \label{36}
t_u = \frac{1}{I_s(u)^{1/s}}, \quad \widetilde{u} = u_{t_u}, \quad \widetilde{t}_u = \frac{1}{t_u} = I_s(u)^{1/s}
\end{equation}
(see \eqref{13}). Then $\widetilde{u} \in \M_s$ and
\[
u = \widetilde{u}_{\widetilde{t}_u}
\]
by \eqref{18}. By $(H_{10})$, $I_s(u) \to \infty$ and $\M_s$ is a bounded manifold, so $\widetilde{t}_u \to \infty$ and $\widetilde{u}$ is bounded. So
\[
F(u) = F(\widetilde{u}_{\widetilde{t}_u}) = \int_0^1 f(\tau \widetilde{u}_{\widetilde{t}_u}) \widetilde{u}_{\widetilde{t}_u}\, d\tau = \int_0^1 f((\tau \widetilde{u})_{\widetilde{t}_u}) \widetilde{u}_{\widetilde{t}_u}\, d\tau = \o(\widetilde{t}_u^s) = \o(I_s(u))
\]
by $(H_2)$, \eqref{35}, and \eqref{36}. Then
\[
\Phi(u) = (1 + \o(1))\, I_s(u) \to \infty. \QED
\]
\end{proof}

Since $\Phi$ is bounded on bounded sets, Lemma \ref{Lemma 2} implies that $\Phi$ is bounded from below. It also implies that every \PS{} sequence of $\Phi$ is bounded, so $\Phi$ satisfies the \PS{} condition by Proposition \ref{Proposition 3}. So $\Phi$ has a minimizer, and consequently we have the following existence result.

\begin{theorem} \label{Theorem 7}
Assume $(H_1)$--$(H_7)$, $(H_{10})$, and \eqref{35}. Then equation \eqref{34} has a solution.
\end{theorem}

Now we assume that
\begin{equation} \label{37}
f(u) = \lambda \Bs(u) + g(u),
\end{equation}
where $\Bs \in \A_s$ satisfies $(H_8)$, $(H_9)$, and $(H_{12})$, $g \in C(W,W^\ast)$ is a compact potential operator satisfying
\begin{equation} \label{38}
g(u_t) v_t = \o(t^s) \norm{v} \text{ as } t \to 0
\end{equation}
uniformly in $u$ on bounded sets for all $v \in W$, and $\lambda \in \R$. Then equation \eqref{34} has the trivial solution $u = 0$ by \eqref{39}, and we seek others.

\begin{theorem} \label{Theorem 8}
Assume $(H_1)$--$(H_{12})$, \eqref{35}, \eqref{37}, and \eqref{38}, and let $\lambda \in \R \setminus \sigma(\As,\Bs)$.
\begin{enumroman}
\item \label{Theorem 8.i} If $\lambda > \lambda_1$, then equation \eqref{34} has a nontrivial solution.
\item \label{Theorem 8.ii} If $\lambda > \lambda_2$, then equation \eqref{34} has two nontrivial solutions.
\end{enumroman}
\end{theorem}

The proof of this theorem will be based on the following proposition, which is a special case of a result proved in Perera \cite{MR1749421} (see also Perera \cite{MR1700283}).

\begin{proposition} \label{Proposition 10}
Let $\Phi$ be a $C^1$-functional on a Banach space $W$ that is bounded from below and satisfies the {\em \PS{}} condition. Assume that $0$ is a critical point of $\Phi$ with $\Phi(0) = 0$ and $C^k(\Phi,0) \ne 0$ for some $k \ge 1$. Then $\Phi$ has a critical point $u_1 \ne 0$ with either $\Phi(u_1) < 0$ and $C^{k-1}(\Phi,u_1) \ne 0$, or $\Phi(u_1) > 0$ and $C^{k+1}(\Phi,u_1) \ne 0$.
\end{proposition}

\begin{proof}[Proof of Theorem \ref{Theorem 8}]
Let $u_0$ be the minimizer of $\Phi$. If $u_0$ is not an isolated minimizer, then $\Phi$ must have infinitely many critical points, so we may assume that $u_0$ is isolated. Then the critical groups of $\Phi$ at $u_0$ are given by
\begin{equation} \label{41}
C^l(\Phi,u_0) \isom \delta_{l0}\, \Z_2.
\end{equation}

\ref{Theorem 8.i} If $\lambda > \lambda_1$, then $\lambda_k < \lambda < \lambda_{k+1}$ for some $k \ge 1$. Then $C^k(\Phi_\lambda,0) \ne 0$ by Theorem \ref{Theorem 2} \ref{Theorem 2.iii}, where $\Phi_\lambda$ is as in \eqref{16}. Since $\lambda \notin \sigma(\As,\Bs)$, $C^l(\Phi,0) \isom C^l(\Phi_\lambda,0)$ for all $l \ge 0$ by Theorem \ref{Theorem 3}, so $C^k(\Phi,0) \ne 0$. Since $C^k(\Phi,u_0) = 0$ by \eqref{41}, it follows that $u_0 \ne 0$.

\ref{Theorem 8.ii} Since $C^k(\Phi,0) \ne 0$, Proposition \ref{Proposition 10} gives a critical point $u_1 \ne 0$ with either $\Phi(u_1) < 0$ and $C^{k-1}(\Phi,u_1) \ne 0$, or $\Phi(u_1) > 0$ and $C^{k+1}(\Phi,u_1) \ne 0$. If $\lambda > \lambda_2$, then $k \ge 2$ and hence $C^{k-1}(\Phi,u_0) = 0 = C^{k+1}(\Phi,u_0)$ by \eqref{41}, so $u_1 \ne u_0$.
\end{proof}

\subsection{Linking based on scaling}

The notion of a linking is useful for obtaining critical points via the minimax principle. Let $X$ be a closed subset of $W$, let $A$ be a proper subset of $X$ that is closed and bounded, and let $B$ be a nonempty closed subset of $W$ such that $A \cap B = \emptyset$. Let
\[
\Gamma = \bgset{\gamma \in C(X,W) : \gamma(X) \text{ is closed and } \restr{\gamma}{A} = \id}.
\]

\begin{definition}
We say that $A$ links $B$ with respect to $X$ if
\[
\gamma(X) \cap B \ne \emptyset \quad \forall \gamma \in \Gamma.
\]
\end{definition}

The following proposition is standard (see, e.g., Perera et al.\! \cite[Proposition 3.21]{MR2640827}).

\begin{proposition} \label{Proposition 6}
Let $\Phi$ be a $C^1$-functional on $W$ and let $X$, $A$, $B$, and $\Gamma$ be as above. Assume that $A$ links $B$ with respect to $X$ and
\[
\sup_{u \in A}\, \Phi(u) \le \inf_{u \in B}\, \Phi(u), \qquad \sup_{u \in X}\, \Phi(u) < \infty.
\]
Set
\[
c := \inf_{\gamma \in \Gamma}\, \sup_{u \in \gamma(X)}\, \Phi(u).
\]
Then
\[
\inf_{u \in B}\, \Phi(u) \le c \le \sup_{u \in X}\, \Phi(u).
\]
If $\Phi$ satisfies the {\em \PS{c}} condition, then $c$ is a critical value of $\Phi$.
\end{proposition}

\subsubsection{Scaled linking sets} \label{Section 2.3.1}

To construct linking sets that are applicable to equation \eqref{2}, let $I \in C(W,\R)$ be an even functional satisfying
\begin{equation} \label{21}
I(u_t) = t^s I(u) \quad \forall u \in W,\, t \ge 0
\end{equation}
and
\begin{equation} \label{22}
I(u) > 0 \quad \forall u \in W \setminus \set{0}.
\end{equation}
Let
\[
\M = \bgset{u \in W : I(u) = 1}
\]
and note that $\M$ is closed and symmetric since $I$ is continuous and even. We assume that $\M$ is a bounded set and that each ray starting from the origin intersects $\M$ at exactly one point. In view of \eqref{21} and \eqref{22}, we can define a continuous projection $\pi : W \setminus \set{0} \to \M$ by setting
\begin{equation} \label{60}
t_u = \frac{1}{I(u)^{1/s}}, \quad \pi(u) = u_{t_u}.
\end{equation}
We have the following proposition describing a new saddle point like geometry.

\begin{proposition} \label{Proposition 8}
Let $A_0$ and $B_0$ be disjoint nonempty closed symmetric subsets of $\M$ such that
\begin{equation} \label{20}
i(A_0) = i(\M \setminus B_0) < \infty,
\end{equation}
let $R > 0$, and let
\begin{gather*}
X = \set{u_t : u \in A_0,\, 0 \le t \le R},\\
A = \set{u_R : u \in A_0},\\
B = \set{u_t : u \in B_0,\, t \ge 0}.
\end{gather*}
Then $A$ links $B$ with respect to $X$.
\end{proposition}

\begin{proof}
If $A$ does not link $B$ with respect to $X$, then there is a continuous map $\gamma : X \to W \setminus B$ such that $\restr{\gamma}{A} = \id$. Denote by $CA_0$ the cone on $A_0$, obtained as the quotient space of $A_0 \times [0,1]$ with $A_0 \times \set{1}$ collapsed to a single point. Then the map $\eta : CA_0 \to W$ defined by
\[
\eta(u,\tau) = \pi(\gamma(u_{(1 - \tau)\, R})), \quad (u,\tau) \in A_0 \times [0,1]
\]
is continuous and $\restr{\eta}{A_0 \times \set{0}} = \id$. So $\eta$ can be extended to an odd continuous map $\widetilde{\eta} : \Sigma A_0 \to W$, where $\Sigma A_0$ is the suspension of $A_0$, by
\[
\widetilde{\eta}(u,\tau) = \begin{cases}
\eta(u,\tau) & \text{if } (u,\tau) \in A_0 \times [0,1]\\[5pt]
- \eta(-u,-\tau) & \text{if } (u,\tau) \in A_0 \times [-1,0).
\end{cases}
\]
Since
\[
\eta(CA_0) = \pi(\gamma(X)) \subset \pi(W \setminus B) = \M \setminus B_0
\]
and $\M \setminus B_0$ is symmetric,
\[
\widetilde{\eta}(\Sigma A_0) = \eta(CA_0) \cup -\eta(CA_0) \subset \M \setminus B_0.
\]
So
\[
i(\M \setminus B_0) \ge i(\Sigma A_0) = i(A_0) + 1
\]
by $(i_2)$ and $(i_6)$ of Proposition \ref{Proposition 7}, contradicting \eqref{20}.
\end{proof}

The following proposition describes a new linking geometry.

\begin{proposition} \label{Proposition 9}
Let $A_0$ and $B_0$ be disjoint nonempty closed symmetric subsets of $\M$ such that
\begin{equation} \label{23}
i(A_0) = i(\M \setminus B_0) < \infty,
\end{equation}
let $R > \rho > 0$, let $e \in \M \setminus A_0$, and let
\begin{gather*}
X = \set{(\pi((1 - \tau)u + \tau e))_t : u \in A_0,\, \tau \in [0,1],\, 0 \le t \le R},\\
A = \set{u_t : u \in A_0,\, 0 \le t \le R} \cup \set{(\pi((1 - \tau)u + \tau e))_R : u \in A_0,\, \tau \in [0,1]},\\
B = \set{u_\rho : u \in B_0}.
\end{gather*}
Then $A$ links $B$ with respect to $X$.
\end{proposition}

\begin{proof}
Let $\Sigma A_0$ be the suspension of $A_0$ and recall that
\begin{equation} \label{26}
i(\Sigma A_0) = i(A_0) + 1
\end{equation}
by Proposition \ref{Proposition 7} $(i_6)$. Let
\[
\widetilde{A} = \set{(\pi((1 - \tau)u + \tau e))_R : u \in A_0,\, \tau \in [0,1]} \cup \set{(\pi((1 - \tau)u - \tau e))_R : u \in A_0,\, \tau \in [0,1]}
\]
and note that $\widetilde{A}$ is closed since $A_0$ is closed (here $(1 - \tau)u \pm \tau e \ne 0$ since $\pm e \notin A_0$ and each ray starting from the origin intersects $\M$ at only one point). We have the odd continuous map
\[
\Sigma A_0 \to \widetilde{A}, \quad (u,\tau) \mapsto \begin{cases}
(\pi((1 - \tau)u + \tau e))_R & \text{if } (u,\tau) \in A_0 \times [0,1]\\[5pt]
(\pi((1 + \tau)u + \tau e))_R & \text{if } (u,\tau) \in A_0 \times [-1,0),
\end{cases}
\]
so
\begin{equation}
i(\Sigma A_0) \le i(\widetilde{A})
\end{equation}
by Proposition \ref{Proposition 7} $(i_2)$.

If $A$ does not link $B$ with respect to $X$, then there is a continuous map $\gamma : X \to W \setminus B$ such that $\gamma(X)$ is closed and $\restr{\gamma}{A} = \id$. Consider the map $\varphi : \widetilde{A} \times [0,1] \to W$ defined by
\[
\varphi(u,t) = \begin{cases}
\gamma(u_t) & \text{if } (u,t) \in (\widetilde{A} \cap A) \times [0,1]\\[5pt]
- \gamma(-u_t) & \text{if } (u,t) \in (\widetilde{A} \setminus A) \times [0,1].
\end{cases}
\]
Since $\gamma$ is the identity on the symmetric set $\set{u_t : u \in A_0,\, 0 \le t \le R}$, $\varphi$ is continuous. Clearly, $\varphi(-u,t) = - \varphi(u,t)$ for all $(u,t) \in \widetilde{A} \times [0,1]$. Since $\gamma(X)$ is closed, so is $\varphi(\widetilde{A} \times [0,1]) = \gamma(X) \cup -\gamma(X)$. By $(H_3)$ and since $\restr{\gamma}{A} = \id$, $\varphi(\widetilde{A} \times \set{0}) = \set{0}$ and $\varphi(\widetilde{A} \times \set{1}) = \widetilde{A}$. Noting that
\[
I(u) = R^s > \rho^s \quad \forall u \in \widetilde{A}
\]
by \eqref{21} and applying Proposition \ref{Proposition 7} $(i_7)$ with $C = \widetilde{A}$, $C_0 = \set{u \in W : I(u) \le \rho^s}$, and $C_1 = \set{u \in W : I(u) \ge \rho^s}$ gives
\begin{equation}
i(\widetilde{A}) \le i(\varphi(\widetilde{A} \times [0,1]) \cap \M_\rho),
\end{equation}
where
\[
\M_\rho = C_0 \cap C_1 = \set{u \in W : I(u) = \rho^s} = \set{u_\rho : u \in \M}.
\]
Since $\gamma(X) \subset W \setminus B$ and $W \setminus B$ is symmetric,
\[
\varphi(\widetilde{A} \times [0,1]) = \gamma(X) \cup -\gamma(X) \subset W \setminus B.
\]
So
\[
\varphi(\widetilde{A} \times [0,1]) \cap \M_\rho \subset (W \setminus B) \cap \M_\rho = \M_\rho \setminus B
\]
and hence
\begin{equation}
i(\varphi(\widetilde{A} \times [0,1]) \cap \M_\rho) \le i(\M_\rho \setminus B)
\end{equation}
by Proposition \ref{Proposition 7} $(i_2)$. Since the restriction of $\pi$ to $\M_\rho \setminus B$ is an odd homeomorphism onto $\M \setminus B_0$, Proposition \ref{Proposition 7} $(i_2)$ also gives
\begin{equation} \label{24}
i(\M_\rho \setminus B) = i(\M \setminus B_0).
\end{equation}
Combining \eqref{26}--\eqref{24} gives $i(A_0) + 1 \le i(\M \setminus B_0)$, contradicting \eqref{23}.
\end{proof}

\subsubsection{Scaled saddle point theorem}

Combining Proposition \ref{Proposition 6} with Proposition \ref{Proposition 8} gives us the following saddle point theorem.

\begin{theorem} \label{Theorem 4}
Let $\Phi$ be a $C^1$-functional on $W$. Let $A_0$ and $B_0$ be disjoint nonempty closed symmetric subsets of $\M$ such that
\begin{equation} \label{32}
i(A_0) = i(\M \setminus B_0) < \infty,
\end{equation}
let $R > 0$, and let
\begin{gather*}
X = \set{u_t : u \in A_0,\, 0 \le t \le R},\\
A = \set{u_R : u \in A_0},\\
B = \set{u_t : u \in B_0,\, t \ge 0}.
\end{gather*}
Assume that
\begin{equation} \label{33}
\sup_{u \in A}\, \Phi(u) \le \inf_{u \in B}\, \Phi(u), \qquad \sup_{u \in X}\, \Phi(u) < \infty.
\end{equation}
Let
\[
\Gamma = \bgset{\gamma \in C(X,W) : \gamma(X) \text{ is closed and } \restr{\gamma}{A} = \id}
\]
and set
\[
c := \inf_{\gamma \in \Gamma}\, \sup_{u \in \gamma(X)}\, \Phi(u).
\]
Then
\[
\inf_{u \in B}\, \Phi(u) \le c \le \sup_{u \in X}\, \Phi(u).
\]
If $\Phi$ satisfies the {\em \PS{c}} condition, then $c$ is a critical value of $\Phi$.
\end{theorem}

\subsubsection{Scaled linking theorem}

Combining Proposition \ref{Proposition 6} with Proposition \ref{Proposition 9} gives us the following linking theorem.

\begin{theorem} \label{Theorem 5}
Let $\Phi$ be a $C^1$-functional on $W$. Let $A_0$ and $B_0$ be disjoint nonempty closed symmetric subsets of $\M$ such that
\[
i(A_0) = i(\M \setminus B_0) < \infty,
\]
let $R > \rho > 0$, let $e \in \M \setminus A_0$, and let
\begin{gather*}
X = \set{(\pi((1 - \tau)u + \tau e))_t : u \in A_0,\, \tau \in [0,1],\, 0 \le t \le R},\\
A = \set{u_t : u \in A_0,\, 0 \le t \le R} \cup \set{(\pi((1 - \tau)u + \tau e))_R : u \in A_0,\, \tau \in [0,1]},\\
B = \set{u_\rho : u \in B_0}.
\end{gather*}
Assume that
\begin{equation} \label{47}
\sup_{u \in A}\, \Phi(u) \le \inf_{u \in B}\, \Phi(u), \qquad \sup_{u \in X}\, \Phi(u) < \infty.
\end{equation}
Let
\[
\Gamma = \bgset{\gamma \in C(X,W) : \gamma(X) \text{ is closed and } \restr{\gamma}{A} = \id}
\]
and set
\begin{equation} \label{48}
c := \inf_{\gamma \in \Gamma}\, \sup_{u \in \gamma(X)}\, \Phi(u).
\end{equation}
Then
\[
\inf_{u \in B}\, \Phi(u) \le c \le \sup_{u \in X}\, \Phi(u).
\]
If $\Phi$ satisfies the {\em \PS{c}} condition, then $c$ is a critical value of $\Phi$.
\end{theorem}

\subsubsection{Asymptotically scaled equations}

As an application of our saddle point theorem, Theorem \ref{Theorem 4}, we prove an existence result for the equation
\begin{equation} \label{25}
\As(u) = \lambda \Bs(u) + g(u)
\end{equation}
in $W^\ast$, where $\As, \Bs \in \A_s$ satisfy $(H_6)$--$(H_{12})$, $g \in C(W,W^\ast)$ is a compact potential operator satisfying
\begin{equation} \label{27}
g(u_t) v_t = \o(t^s) \norm{v} \text{ as } t \to \infty
\end{equation}
uniformly in $u$ on bounded sets for all $v \in W$, and $\lambda \in \R$. We will refer to this equation as an asymptotically scaled equation. We have the following theorem.

\begin{theorem} \label{Theorem 6}
Assume $(H_1)$--$(H_{12})$ and \eqref{27}, and let $\lambda \in \R \setminus \sigma(\As,\Bs)$. Then equation \eqref{25} has a solution.
\end{theorem}

The variational functional associated with equation \eqref{25} is
\[
\Phi(u) = I_s(u) - \lambda\, J_s(u) - G(u), \quad u \in W,
\]
where $I_s$ and $J_s$ are as in \eqref{3} and \eqref{8}, respectively, and
\[
G(u) = \int_0^1 g(\tau u) u\, d\tau
\]
is the potential of $g$ with $G(0) = 0$. Since the compact operator $g$ maps bounded sets into precompact, and hence bounded, sets, $G$ is bounded on bounded sets. Since $I_s$ and $J_s$ are also bounded on bounded sets, then so is $\Phi$.

\begin{lemma} \label{Lemma 1}
If \eqref{27} holds and $\lambda \in \R \setminus \sigma(\As,\Bs)$, then $\Phi$ satisfies the {\em \PS{}} condition.
\end{lemma}

\begin{proof}
We will show that every sequence $\seq{u_j}$ in $W$ such that $\Phi'(u_j) \to 0$ is bounded. The desired conclusion will then follow from Proposition \ref{Proposition 3} since $B_s$ and $g$ are compact. Suppose $\norm{u_j} \to \infty$ for a renamed subsequence. Set
\[
t_j = t_{u_j} = \frac{1}{I_s(u_j)^{1/s}}, \quad \widetilde{u}_j = (u_j)_{t_j}, \quad \widetilde{t}_j = \frac{1}{t_j} = I_s(u_j)^{1/s}
\]
(see \eqref{13}). Then $\widetilde{u}_j \in \M_s$ and
\[
u_j = (\widetilde{u}_j)_{\widetilde{t}_j}
\]
by \eqref{18}. Since $\M_s$ is a bounded manifold by $(H_{10})$, $\seq{\widetilde{u}_j}$ is bounded and hence converges weakly to some $\widetilde{u} \in W$ for a renamed subsequence. Since $\norm{u_j} \to \infty$, $I_s(u_j) \to \infty$ by $(H_{10})$ and hence $\widetilde{t}_j \to \infty$.

Since $\Phi'(u_j) \to 0$ and $u_j = (\widetilde{u}_j)_{\widetilde{t}_j}$,
\[
\As((\widetilde{u}_j)_{\widetilde{t}_j}) v_{\widetilde{t}_j} = \lambda \Bs((\widetilde{u}_j)_{\widetilde{t}_j}) v_{\widetilde{t}_j} + g((\widetilde{u}_j)_{\widetilde{t}_j}) v_{\widetilde{t}_j} + \o(1) \|v_{\widetilde{t}_j}\|
\]
for any $v \in W$. Since $\As, \Bs \in \A_s$, $g$ satisfies \eqref{27}, $\widetilde{t}_j \to \infty$, and $\seq{\widetilde{u}_j}$ is bounded, this reduces to
\[
\widetilde{t}_j^s\, \As(\widetilde{u}_j) v = \lambda \widetilde{t}_j^s\, \Bs(\widetilde{u}_j) v + \o(\widetilde{t}_j^s) \norm{v} + \o(1) \|v_{\widetilde{t}_j}\|,
\]
or
\begin{equation} \label{28}
\As(\widetilde{u}_j) v = \lambda \Bs(\widetilde{u}_j) v + \o(1) \left(\norm{v} + \frac{\|v_{\widetilde{t}_j}\|}{\widetilde{t}_j^s}\right).
\end{equation}
Taking $v = \widetilde{u}_j - \widetilde{u}$ and noting that $\Bs(\widetilde{u}_j) \to \Bs(\widetilde{u})$ by $(H_9)$ and $\|v_{\widetilde{t}_j}\|/\widetilde{t}_j^s$ is bounded by $(H_5)$ shows that $\As(\widetilde{u}_j)(\widetilde{u}_j - \widetilde{u}) \to 0$, so $\widetilde{u}_j \to \widetilde{u}$ for a further subsequence by $(H_7)$. Passing to the limit in \eqref{28} now gives
\[
\As(\widetilde{u}) = \lambda \Bs(\widetilde{u})
\]
since $v$ is arbitrary. Since $\widetilde{u}_j \in \M_s$ and $\M_s$ is closed, $\widetilde{u} \in \M_s$ and hence $\widetilde{u} \ne 0$. So $\lambda \in \sigma(\As,\Bs)$, a contradiction.
\end{proof}

We are now ready to prove Theorem \ref{Theorem 6}.

\begin{proof}[Proof of Theorem \ref{Theorem 6}]
By $(H_2)$ and \eqref{27},
\begin{equation} \label{30}
G(u_t) = \int_0^1 g(\tau u_t) u_t\, d\tau = \int_0^1 g((\tau u)_t) u_t\, d\tau = \o(t^s) \text{ as } t \to \infty
\end{equation}
uniformly in $u$ on bounded sets. We will make use of the parametrization
\[
W = \set{u_t : u \in \M_s,\, t \ge 0}
\]
(see \eqref{18}). By \eqref{12}, \eqref{6}, and \eqref{30},
\begin{equation} \label{29}
\Phi(u_t) = I_s(u_t) - \lambda\, J_s(u_t) - G(u_t) = t^s\, \bigg(1 - \frac{\lambda}{\widetilde{\Psi}(u)} + \o(1)\bigg) \text{ as } t \to \infty
\end{equation}
uniformly in $u \in \M_s$.

First suppose $\lambda < \lambda_1$. Then \eqref{29} together with Theorem \ref{Theorem 1} \ref{Theorem 1.i} gives
\[
\Phi(u_t) \ge t^s \left(1 - \frac{\lambda^+}{\lambda_1} + \o(1)\right) \text{ as } t \to \infty
\]
uniformly in $u \in \M_s$, where $\lambda^+ = \max \set{\lambda,0} < \lambda_1$, so $\Phi$ is bounded from below. Since $\Phi$ satisfies the \PS{} condition by Lemma \ref{Lemma 1}, then it has a minimizer.

Now suppose $\lambda > \lambda_1$. Since $\lambda \notin \sigma(\As,\Bs)$, then $\lambda_k < \lambda < \lambda_{k+1}$ for some $k \ge 1$, so
\begin{equation} \label{31}
i(\widetilde{\Psi}^{\lambda_k}) = i(\M_s \setminus \widetilde{\Psi}_{\lambda_{k+1}}) = k
\end{equation}
by Theorem \ref{Theorem 1} \ref{Theorem 1.iii}. We apply Theorem \ref{Theorem 4} with $A_0 = \widetilde{\Psi}^{\lambda_k}$ and $B_0 = \widetilde{\Psi}_{\lambda_{k+1}}$, noting that \eqref{32} holds by \eqref{31}. Let $R > 0$ and let
\begin{gather*}
X = \bgset{u_t : u \in \widetilde{\Psi}^{\lambda_k},\, 0 \le t \le R},\\
A = \bgset{u_R : u \in \widetilde{\Psi}^{\lambda_k}},\\
B = \bgset{u_t : u \in \widetilde{\Psi}_{\lambda_{k+1}},\, t \ge 0}.
\end{gather*}
We have
\[
\Phi(u_t) \ge t^s \left(1 - \frac{\lambda}{\lambda_{k+1}} + \o(1)\right) \text{ as } t \to \infty
\]
uniformly in $u \in \widetilde{\Psi}_{\lambda_{k+1}}$ by \eqref{29}. Since $\lambda < \lambda_{k+1}$, it follows from this that $\Phi$ is bounded from below on $B$. On the other hand, \eqref{29} also gives
\[
\Phi(u_t) \le - t^s \left(\frac{\lambda}{\lambda_k} - 1 + \o(1)\right) \text{ as } t \to \infty
\]
uniformly in $u \in \widetilde{\Psi}^{\lambda_k}$. Since $\lambda > \lambda_k$, it follows that the first inequality in \eqref{33} holds if $R$ is sufficiently large. The second inequality also holds since $X$ is bounded by $(H_4)$ and $\Phi$ is bounded on bounded sets. Since $\Phi$ satisfies the \PS{} condition, then it has a critical point by Theorem \ref{Theorem 4}.
\end{proof}

\subsubsection{Superscaled equations}

Now we consider the equation
\begin{equation} \label{40}
\As(u) = f(u)
\end{equation}
in $W^\ast$, where $\As \in \A_s$ satisfies $(H_6)$, $(H_7)$, and $(H_{10})$, and $f \in C(W,W^\ast)$ is a potential operator with potential
\[
F(u) = \int_0^1 f(\tau u) u\, d\tau
\]
satisfying
\begin{equation} \label{42}
\frac{F(u_t)}{t^s} \to \infty \quad \text{as } t \to \infty
\end{equation}
uniformly in $u$ on compact subsets of $\M_s$. We will refer to this equation as a superscaled equation.

The variational functional associated with equation \eqref{40} is
\[
\Phi(u) = I_s(u) - F(u), \quad u \in W,
\]
where $I_s$ is as in \eqref{3}. We have
\[
\frac{\Phi(u_t)}{t^s} = I_s(u) - \frac{F(u_t)}{t^s} \quad \forall u \in W,\, t > 0
\]
by \eqref{12}, and $I_s$ is bounded on bounded sets, so it follows from \eqref{42} that
\begin{equation} \label{46}
\Phi(u_t) \to - \infty \quad \text{as } t \to \infty
\end{equation}
uniformly in $u$ on compact subsets of $\M_s$. However, this does not guarantee that every \PS{} sequence of $\Phi$ is bounded. Moreover, even bounded \PS{} sequences may not have convergent subsequences since $f$ is not assumed to be a compact operator.

We assume that
\begin{equation} \label{49}
f(u) = \lambda \Bs(u) + g(u),
\end{equation}
where $\Bs \in \A_s$ satisfies $(H_8)$, $(H_9)$, and $(H_{12})$, $g \in C(W,W^\ast)$ is a potential operator \linebreak satisfying
\begin{equation} \label{43}
g(u_t) v_t = \o(t^s) \norm{v} \text{ as } t \to 0
\end{equation}
uniformly in $u$ on bounded sets for all $v \in W$, and $\lambda \in \R \setminus \sigma(\As,\Bs)$. Then equation \eqref{40} has the trivial solution $u = 0$ by \eqref{39}. We will use our linking theorem, Theorem \ref{Theorem 5}, to produce a minimax level $c > 0$ that is critical when $\Phi$ satisfies the \PS{c} condition.

We have
\[
\Phi(u) = I_s(u) - \lambda\, J_s(u) - G(u), \quad u \in W,
\]
where $J_s$ is as in \eqref{8} and
\[
G(u) = \int_0^1 g(\tau u) u\, d\tau
\]
is the potential of $g$ with $G(0) = 0$. By $(H_2)$ and \eqref{43},
\begin{equation} \label{50}
G(u_t) = \int_0^1 g(\tau u_t) u_t\, d\tau = \int_0^1 g((\tau u)_t) u_t\, d\tau = \o(t^s) \text{ as } t \to 0
\end{equation}
uniformly in $u$ on bounded sets. We assume that
\begin{equation} \label{45}
G(u) \ge 0 \quad \forall u \in W.
\end{equation}

We use the parametrization
\[
W = \set{u_t : u \in \M_s,\, t \ge 0}
\]
(see \eqref{18}). By \eqref{12}, \eqref{6}, and \eqref{50},
\begin{equation} \label{44}
\Phi(u_t) = I_s(u_t) - \lambda\, J_s(u_t) - G(u_t) = t^s\, \bigg(1 - \frac{\lambda}{\widetilde{\Psi}(u)} + \o(1)\bigg) \text{ as } t \to 0
\end{equation}
uniformly in $u \in \M_s$. If $\lambda < \lambda_1$, then this together with Theorem \ref{Theorem 1} \ref{Theorem 1.i} gives
\[
\Phi(u_t) \ge t^s \left(1 - \frac{\lambda^+}{\lambda_1} + \o(1)\right) \text{ as } t \to 0
\]
uniformly in $u \in \M_s$, where $\lambda^+ = \max \set{\lambda,0} < \lambda_1$, so $\Phi$ has the mountain pass geometry. So suppose $\lambda > \lambda_1$. Since $\lambda \notin \sigma(\As,\Bs)$, then $\lambda_k < \lambda < \lambda_{k+1}$ for some $k \ge 1$, so we have
\[
i(\M_s \setminus \widetilde{\Psi}_\lambda) = i(\M_s \setminus \widetilde{\Psi}_{\lambda_{k+1}}) = k
\]
by Theorem \ref{Theorem 1} \ref{Theorem 1.iii}. Since $\M_s \setminus \widetilde{\Psi}_\lambda$ is an open symmetric subset of $\M_s$, then it has a compact symmetric subset $C$ of index $k$ (see the proof of Proposition 3.1 in Degiovanni and Lancelotti \cite{MR2371112}). We apply Theorem \ref{Theorem 5} with $A_0 = C$ and $B_0 = \widetilde{\Psi}_{\lambda_{k+1}}$.

Let $R > \rho > 0$, let $e \in \M_s \setminus C$, and let
\begin{gather*}
X = \bgset{(\pi((1 - \tau)u + \tau e))_t : u \in C,\, \tau \in [0,1],\, 0 \le t \le R},\\
A = \bgset{u_t : u \in C,\, 0 \le t \le R} \cup \bgset{(\pi((1 - \tau)u + \tau e))_R : u \in C,\, \tau \in [0,1]},\\
B = \bgset{u_\rho : u \in \widetilde{\Psi}_{\lambda_{k+1}}}.
\end{gather*}
By \eqref{45}, \eqref{12}, and \eqref{6},
\[
\Phi(u_t) \le I_s(u_t) - \lambda\, J_s(u_t) = t^s\, \bigg(1 - \frac{\lambda}{\widetilde{\Psi}(u)}\bigg) \le 0
\]
for all $u \in C$ and $t \ge 0$. Since $\bgset{\pi((1 - \tau)u + \tau e) : u \in C,\, \tau \in [0,1]}$ is a compact subset of $\M_s$, it follows from \eqref{46} that $\Phi \le 0$ on $\bgset{(\pi((1 - \tau)u + \tau e))_R : u \in C,\, \tau \in [0,1]}$ if $R$ is sufficiently large. Then
\[
\sup_{u \in A}\, \Phi(u) \le 0.
\]
On the other hand, \eqref{44} gives
\[
\Phi(u_t) \ge t^s \left(1 - \frac{\lambda}{\lambda_{k+1}} + \o(1)\right) \text{ as } t \to 0
\]
uniformly in $u \in \widetilde{\Psi}_{\lambda_{k+1}}$. Since $\lambda < \lambda_{k+1}$, it follows that
\[
\inf_{u \in B}\, \Phi(u) > 0
\]
if $\rho$ is sufficiently small. So the first inequality in \eqref{47} holds. The second inequality also holds since $X$ is bounded by $(H_4)$ and $\Phi$ is bounded from above on bounded sets by \eqref{45}. Theorem \ref{Theorem 5} now gives the following result.

\begin{theorem} \label{Theorem 9}
Assume $(H_1)$--$(H_{12})$, \eqref{42}, \eqref{49}, \eqref{43}, and \eqref{45}, let $\lambda \in \R \setminus \sigma(\As,\Bs)$, and let $c$ be as in \eqref{48}. If $\Phi$ satisfies the {\em \PS{c}} condition, then equation \eqref{40} has a nontrivial solution at the level $c > 0$.
\end{theorem}

\subsection{Local linking based on scaling}

The notion of a local linking is useful for obtaining nontrivial critical points via Morse theory. Let $\Phi$ be a $C^1$-functional on $W$ and let $I$, $\M$, and $\pi$ be as in the beginning of Section \ref{Section 2.3.1}.

\begin{definition} \label{Definition 2}
We will say that $\Phi$ has a scaled local linking near the origin in dimension $k \ge 1$ if there are disjoint nonempty closed symmetric subsets $A_0$ and $B_0$ of $\M$ with
\begin{equation} \label{56}
i(A_0) = i(\M \setminus B_0) = k
\end{equation}
and $\rho > 0$ such that
\begin{equation} \label{57}
\begin{cases}
\Phi(u_t) \le 0 & \forall u \in A_0,\, 0 \le t \le \rho\\[7.5pt]
\Phi(u_t) > 0 & \forall u \in B_0,\, 0 < t \le \rho.
\end{cases}
\end{equation}
\end{definition}

The special case $u_t = tu$ of this definition was given in Perera \cite{MR1700283} (see also Perera et al.\! \cite[Definition 3.33]{MR2640827} and Degiovanni et al.\! \cite{MR2661274}). The usefulness of this notion lies in the following theorem.

\begin{theorem} \label{Theorem 16}
If $\Phi$ has a scaled local linking near the origin in dimension $k$, then
\[
C^k(\Phi,0) \ne 0.
\]
\end{theorem}

\begin{proof}
We have
\begin{equation} \label{59}
C^k(\Phi,0) = H^k(\Phi^0 \cap U,\Phi^0 \cap U \setminus \set{0}),
\end{equation}
where $\Phi^0 = \bgset{u \in W : \Phi(u) \le 0}$ and $U = \set{u_t : u \in \M,\, 0 \le t \le \rho}$.

Let
\[
\M_\rho = \set{u_\rho : u \in \M}, \qquad A_\rho = \set{u_\rho : u \in A_0}, \qquad B_\rho = \set{u_\rho : u \in B_0}
\]
and note that \eqref{56} together with Proposition \ref{Proposition 7} $(i_2)$ gives
\begin{equation} \label{58}
i(A_\rho) = i(\M_\rho \setminus B_\rho) = k
\end{equation}
since the restrictions of $\pi$ to $A_\rho$ and $\M_\rho \setminus B_\rho$ are odd homeomorphisms onto $A_0$ and $\M \setminus B_0$, respectively. Let
\[
A = \set{u_t : u \in A_0,\, 0 \le t \le \rho}, \qquad B = \set{u_t : u \in B_0,\, 0 \le t \le \rho}.
\]
By \eqref{57}, $A \subset \Phi^0 \cap U$ and $A_\rho \subset \Phi^0 \cap U \setminus \set{0} \subset U \setminus B$, so we have the following commutative diagram of reduced cohomology groups induced by inclusions.
\begin{equation*}
\begin{CD}
\widetilde{H}^{k-1}(A) @>>> \widetilde{H}^{k-1}(A_\rho) @<j^\ast<< \widetilde{H}^{k-1}(\M_\rho \setminus B_\rho)\\
@AAA @Al^\ast AA @Ak^\ast AA\\
\widetilde{H}^{k-1}(\Phi^0 \cap U) @>i^\ast>> \widetilde{H}^{k-1}(\Phi^0 \cap U \setminus \set{0}) @<<< \widetilde{H}^{k-1}(U \setminus B)
\end{CD}
\end{equation*}
By \eqref{58} and Perera et al.\! \cite[Proposition 2.14 ({\em iv})]{MR2640827}, $j^\ast \ne 0$. Since $U \setminus B$ deformation retracts to $\M_\rho \setminus B_\rho$ via
\[
(U \setminus B) \times [0,1] \to U \setminus B, \quad (u,\tau) \mapsto u_{1 - \tau + \tau\, t_u \rho}
\]
(see \eqref{60}), $k^\ast$ is an isomorphism. So the square on the right gives $l^\ast \ne 0$. On the other hand, since $A$ contracts to $\set{0}$ via
\[
A \times [0,1] \to A, \quad (u,\tau) \mapsto u_{1 - \tau},
\]
the square on the left gives $l^\ast i^\ast = 0$. It follows that $i^\ast$ is not onto.

Now consider the following portion of the exact sequence of the pair $(\Phi^0 \cap U,\Phi^0 \cap U \setminus \set{0})$.
\[
\begin{CD}
\widetilde{H}^{k-1}(\Phi^0 \cap U) @>i^\ast>> \widetilde{H}^{k-1}(\Phi^0 \cap U \setminus \set{0}) @>\delta>> H^k(\Phi^0 \cap U,\Phi^0 \cap U \setminus \set{0})
\end{CD}
\]
Since $\ker \delta = \im i^\ast$ and $i^\ast$ is not onto, $\delta \ne 0$ and hence $H^k(\Phi^0 \cap U,\Phi^0 \cap U \setminus \set{0}) \ne 0$. So the desired conclusion follows from \eqref{59}.
\end{proof}

We can now prove the following theorem on the existence of nontrivial critical points for functionals with a scaled local linking near the origin. For $a \in \R$, let $\Phi^a = \bgset{u \in W : \Phi(u) \le a}$.

\begin{theorem} \label{Theorem 17}
Let $\Phi$ be a $C^1$-functional on $W$ that satisfies the {\em \PS{}} condition. Assume that $\Phi$ has a scaled local linking near the origin in dimension $k$ and that $\widetilde{H}^{k-1}(\Phi^a) = 0$ for some $a < 0$. Then $\Phi$ has a nontrivial critical point.
\end{theorem}

\begin{proof}
Suppose $\Phi$ has no nontrivial critical point. Then $\Phi^0$ and $\Phi^a$ are deformation retracts of $W$ and $\Phi^0 \setminus \set{0}$, respectively, by the second deformation lemma (see, e.g., Perera et al.\! \cite[Lemma 3.9]{MR2640827}). So
\[
C^k(\Phi,0) = H^k(\Phi^0,\Phi^0 \setminus \set{0}) \isom H^k(W,\Phi^a).
\]
If $\Phi^a \ne \emptyset$, since $W$ is contractible,
\[
H^k(W,\Phi^a) \isom \widetilde{H}^{k-1}(\Phi^a) = 0
\]
by assumption (see, e.g., Perera et al.\! \cite[Proposition 2.4 $(ii)$]{MR2640827}). On the other hand, if $\Phi^a = \emptyset$, then
\[
H^k(W,\Phi^a) = H^k(W) = 0
\]
since $k \ge 1$. So $C^k(\Phi,0) = 0$, contrary to the conclusion of Theorem \ref{Theorem 16}.
\end{proof}

In particular, we have the following corollary, which we will use to obtain nontrivial solutions of superscaled Schr\"{o}dinger--Poisson--Slater equations that are resonant at zero (see Theorem \ref{Theorem 18}).

\begin{corollary} \label{Corollary 2}
Let $\Phi$ be a $C^1$-functional on $W$ that satisfies the {\em \PS{}} condition. Assume that $\Phi$ has a scaled local linking near the origin and that $\Phi^a$ is contractible for some $a < 0$. Then $\Phi$ has a nontrivial critical point.
\end{corollary}

\subsection{Multiplicity based on scaling}

In this section we prove a multiplicity result for even functionals that only satisfy a local \PS{} condition and is therefore applicable to scaled equations with critical growth. Let $\Phi \in C^1(W,\R)$ be an even functional, i.e., $\Phi(-u) = \Phi(u)$ for all $u \in W$. Assume that $\exists c^\ast > 0$ such that $\Phi$ satisfies the \PS{c} condition for all $c \in (0,c^\ast)$. Let $\Gamma$ denote the group of odd homeomorphisms of $W$ that are the identity outside $\Phi^{-1}(0,c^\ast)$. Let $\A^\ast$ denote the class of symmetric subsets of $W$, let $I$, $\M$, and $\pi$ be as in the beginning of Section \ref{Section 2.3.1}, and let
\[
\M_\rho = \set{u \in W : I(u) = \rho^s} = \set{u_\rho : u \in \M}
\]
for $\rho > 0$.

\begin{definition}[Benci \cite{MR84c:58014}]
The pseudo-index of $M \in \A^\ast$ related to $i$, $\M_\rho$, and $\Gamma$ is defined by
\begin{equation} \label{74}
i^\ast(M) = \min_{\gamma \in \Gamma}\, i(\gamma(M) \cap \M_\rho).
\end{equation}
\end{definition}

We have the following multiplicity result.

\begin{theorem} \label{Theorem 14}
Let $A_0$ and $B_0$ be symmetric subsets of $\M$ such that $A_0$ is compact, $B_0$ is closed, and
\begin{equation} \label{73}
i(A_0) \ge k + m - 1, \qquad i(\M \setminus B_0) \le k - 1
\end{equation}
for some $k, m \ge 1$. Let $R > \rho > 0$ and let
\begin{gather*}
X = \set{u_t : u \in A_0,\, 0 \le t \le R},\\
A = \set{u_R : u \in A_0},\\
B = \set{u_\rho : u \in B_0}.
\end{gather*}
Assume that
\begin{equation} \label{72}
\sup_{u \in A}\, \Phi(u) \le 0 < \inf_{u \in B}\, \Phi(u), \qquad \sup_{u \in X}\, \Phi(u) < c^\ast.
\end{equation}
For $j = k,\dots,k + m - 1$, let
\[
\A_j^\ast = \set{M \in \A^\ast : M \text{ is compact and } i^\ast(M) \ge j}
\]
and set
\[
c_j^\ast := \inf_{M \in \A_j^\ast}\, \max_{u \in M}\, \Phi(u).
\]
Then $0 < c_k^\ast \le \dotsb \le c_{k+m-1}^\ast < c^\ast$, each $c_j^\ast$ is a critical value of $\Phi$, and $\Phi$ has $m$ distinct pairs of associated critical points.
\end{theorem}

\begin{proof}
First we show that $c_k^\ast > 0$. Let $M \in \A_k^\ast$. Since the identity map is in $\Gamma$, \eqref{74} gives
\begin{equation} \label{75}
i(M \cap \M_\rho) \ge i^\ast(M) \ge k.
\end{equation}
On the other hand, since the restriction of $\pi$ to $\M_\rho \setminus B$ is an odd homeomorphism onto $\M \setminus B_0$, Proposition \ref{Proposition 7} $(i_2)$ gives
\begin{equation} \label{76}
i(\M_\rho \setminus B) = i(\M \setminus B_0) \le k - 1
\end{equation}
by \eqref{73}. Combining \eqref{75} and \eqref{76} gives $i(M \cap \M_\rho) > i(\M_\rho \setminus B)$, so $M$ intersects $B$ by Proposition \ref{Proposition 7} $(i_2)$ again. It follows that
\[
c_k^\ast \ge \inf_{u \in B}\, \Phi(u) > 0
\]
by \eqref{72}.

Next we show that $c_{k+m-1}^\ast < c^\ast$. Let $\gamma \in \Gamma$. Since $A_0$ is compact, so are $X$ and $A$, in particular, $A$ is closed. Consider the continuous map $\varphi : A \times [0,1] \to W$ defined by
\[
\varphi(u,t) = \gamma(u_t).
\]
By $(H_2)$ and since $\gamma$ is odd,
\[
\varphi(-u,t) = \gamma((-u)_t) = \gamma(-u_t) = - \gamma(u_t) = - \varphi(u,t) \quad \forall (u,t) \in A \times [0,1].
\]
We have $\varphi(A \times [0,1]) = \gamma(X)$. Since $X$ is compact, so is $\gamma(X)$, so $\varphi(A \times [0,1])$ is closed. By $(H_3)$,
\[
\varphi(A \times \set{0}) = \set{\gamma(0)} = \set{0}
\]
since $\gamma$ is odd. On the other hand, since $\Phi \le 0$ on $A$ by \eqref{72} and $\gamma$ is the identity outside $\Phi^{-1}(0,c^\ast)$, $\gamma$ is the identity on $A$ and hence $(H_3)$ also gives
\[
\varphi(A \times \set{1}) = \gamma(A) = A.
\]
Noting that
\[
I(u) = R^s > \rho^s \quad \forall u \in A
\]
by \eqref{21} and applying Proposition \ref{Proposition 7} $(i_7)$ with $C = A$, $C_0 = \set{u \in W : I(u) \le \rho^s}$, and $C_1 = \set{u \in W : I(u) \ge \rho^s}$ now gives
\begin{equation} \label{77}
i(\gamma(X) \cap \M_\rho) = i(\varphi(A \times [0,1]) \cap C_0 \cap C_1) \ge i(A).
\end{equation}
Since the restriction of $\pi$ to $A$ is an odd homeomorphism onto $A_0$, Proposition \ref{Proposition 7} $(i_2)$ gives
\begin{equation} \label{78}
i(A) = i(A_0) \ge k + m - 1
\end{equation}
by \eqref{73}. Combining \eqref{77} and \eqref{78} gives $i(\gamma(X) \cap \M_\rho) \ge k + m - 1$. Since $\gamma \in \Gamma$ is arbitrary, it follows that $i^\ast(X) \ge k + m - 1$. So $X \in \A_{k+m-1}^\ast$ and hence
\[
c_{k+m-1}^\ast \le \max_{u \in X}\, \Phi(u) < c^\ast
\]
by \eqref{72}.

The rest now follows from standard results in critical point theory (see, e.g., Perera et al.\! \cite[Proposition 3.42]{MR2640827}).
\end{proof}

In particular, we have the following corollary when $B_0 = \M$ and $k = 1$ (the special case $u_t = tu$ of this corollary was recently proved in Perera \cite{Pe23}).

\begin{corollary} \label{Corollary 1}
Let $A_0$ be a compact symmetric subset of $\M$ with $i(A_0) = m \ge 1$, let $R > \rho > 0$, and let
\[
A = \set{u_R : u \in A_0}, \qquad X = \set{u_t : u \in A_0,\, 0 \le t \le R}.
\]
Assume that
\[
\sup_{u \in A}\, \Phi(u) \le 0 < \inf_{u \in \M_\rho}\, \Phi(u), \qquad \sup_{u \in X}\, \Phi(u) < c^\ast.
\]
Then $\Phi$ has $m$ distinct pairs of critical points at levels in $(0,c^\ast)$.
\end{corollary}

\section{Schr\"{o}dinger--Poisson--Slater equation}

\subsection{Variational setting}

In this section we prove our results for the Schr\"{o}dinger--Poisson--Slater equation by applying our abstract results. We take $W = E_r(\R^3)$ endowed with the norm
\begin{gather*}
\norm{u} = \left[\int_{\R^3} |\nabla u|^2\, dx + \left(\int_{\R^3} \int_{\R^3} \frac{u^2(x)\, u^2(y)}{|x - y|}\, dx\, dy\right)^{1/2}\right]^{1/2},\\[7.5pt]
u_t(x) = t^2\, u(tx), \quad x \in \R^3,
\intertext{and}
s = 3.
\end{gather*}
Note that
\begin{multline} \label{110}
\norm{u_t} = \left[t^3 \int_{\R^3} |\nabla u|^2\, dx + t^{3/2} \left(\int_{\R^3} \int_{\R^3} \frac{u^2(x)\, u^2(y)}{|x - y|}\, dx\, dy\right)^{1/2}\right]^{1/2}\\[7.5pt]
\le \max \set{t^{3/2},t^{3/4}} \norm{u}.
\end{multline}

\begin{lemma}
The mapping $E_r(\R^3) \times [0,\infty) \to E_r(\R^3),\, (u,t) \mapsto u_t$ is continuous.
\end{lemma}

\begin{proof}
We will show that if $u_j \to u$ in $E_r(\R^3)$ and $t_j \to t$ in $[0,\infty)$, then $(u_j)_{t_j} \to u_t$ in $E_r(\R^3)$. If $t = 0$, then $u_t = 0$ and $(u_j)_{t_j} \to 0$ by \eqref{110}, so suppose $t > 0$. Then
\[
\|(u_j)_{t_j} - u_t\| \le \|(u_j)_{t_j} - u_{t_j}\| + \|u_{t_j} - u_t\|,
\]
and \eqref{110} gives
\[
\|(u_j)_{t_j} - u_{t_j}\| = \|(u_j - u)_{t_j}\| \le \max \bgset{t_j^{3/2},t_j^{3/4}} \norm{u_j - u} \to 0
\]
and
\[
\|u_{t_j} - u_t\| = \|(u_{t_j/t})_t - u_t\| = \|(u_{t_j/t} - u)_t\| \le \max \set{t^{3/2},t^{3/4}} \norm{u_{t_j/t} - u}.
\]
So it suffices to show that $u_{t_j} \to u$ if $t_j \to 1$. This is easy to see for $u \in C^\infty_0(\R^3)$, and the general case then follows from density.
\end{proof}

Conditions $(H_1)$--$(H_3)$ are clearly satisfied, while $(H_4)$ and $(H_5)$ follow from \eqref{110}. The operators $\As$ and $\Bs$ are given by
\[
\As(u) v = \int_{\R^3} \nabla u \cdot \nabla v\, dx + \frac{1}{4 \pi} \int_{\R^3} \int_{\R^3} \frac{u^2(x)\, u(y)\, v(y)}{|x - y|}\, dx\, dy
\]
and
\[
\Bs(u) v = \int_{\R^3} |u|\, uv\, dx.
\]
Clearly, $\As, \Bs \in \A_s$ and satisfy $(H_6)$ and $(H_8)$. Condition $(H_7)$ is verified in the following lemma, while $(H_9)$ follows from the compactness of the embedding $E_r(\R^3) \hookrightarrow L^3(\R^3)$.

\begin{lemma}
Every sequence $\seq{u_j}$ in $E_r(\R^3)$ such that $u_j \wto u$ and $\As(u_j)(u_j - u) \to 0$ has a subsequence that converges strongly to $u$.
\end{lemma}

\begin{proof}
We have
\[
\As(u_j)(u_j - u) = \int_{\R^3} \nabla u_j \cdot \nabla (u_j - u)\, dx + \frac{1}{4 \pi} \int_{\R^3} \int_{\R^3} \frac{u_j^2(x)\, u_j(y)\, (u_j(y) - u(y))}{|x - y|}\, dx\, dy.
\]
Since $u_j \wto u$,
\[
\int_{\R^3} \nabla u_j \cdot \nabla u\, dx \to \int_{\R^3} |\nabla u|^2\, dx
\]
and
\[
\int_{\R^3} \int_{\R^3} \frac{u_j^2(x)\, u_j(y)\, u(y)}{|x - y|}\, dx\, dy \to \int_{\R^3} \int_{\R^3} \frac{u^2(x)\, u^2(y)}{|x - y|}\, dx\, dy
\]
by Ianni and Ruiz \cite[Lemma 2.2 and Lemma 2.3]{MR2902293}, so
\begin{multline*}
\As(u_j)(u_j - u) = \int_{\R^3} |\nabla u_j|^2\, dx - \int_{\R^3} |\nabla u|^2\, dx + \frac{1}{4 \pi}\, \bigg(\int_{\R^3} \int_{\R^3} \frac{u_j^2(x)\, u_j^2(y)}{|x - y|}\, dx\, dy\\[7.5pt]
- \int_{\R^3} \int_{\R^3} \frac{u^2(x)\, u^2(y)}{|x - y|}\, dx\, dy\bigg) + \o(1).
\end{multline*}
Moreover, $\seq{u_j}$ is bounded in $E_r(\R^3)$ and hence converges to $u$ a.e.\! in $\R^3$ for a renamed subsequence, so
\[
\int_{\R^3} |\nabla u_j|^2\, dx - \int_{\R^3} |\nabla u|^2\, dx = \int_{\R^3} |\nabla v_j|^2\, dx + \o(1)
\]
by the Brezis-Lieb lemma and
\[
\int_{\R^3} \int_{\R^3} \frac{u_j^2(x)\, u_j^2(y)}{|x - y|}\, dx\, dy - \int_{\R^3} \int_{\R^3} \frac{u^2(x)\, u^2(y)}{|x - y|}\, dx\, dy\\[7.5pt]
\ge \int_{\R^3} \int_{\R^3} \frac{v_j^{\,2}(x)\, v_j^{\,2}(y)}{|x - y|}\, dx\, dy + \o(1)
\]
by Mercuri et al.\! \cite[Proposition 4.1]{MR3568051}, where $v_j = u_j - u$. So
\[
\As(u_j)(u_j - u) \ge \int_{\R^3} |\nabla v_j|^2\, dx + \frac{1}{4 \pi} \int_{\R^3} \int_{\R^3} \frac{v_j^{\,2}(x)\, v_j^{\,2}(y)}{|x - y|}\, dx\, dy + \o(1).
\]
Since $\As(u_j)(u_j - u) \to 0$, this implies that $v_j \to 0$ in $E_r(\R^3)$.
\end{proof}

We have
\begin{gather*}
I_s(u) = \frac{1}{2} \int_{\R^3} |\nabla u|^2\, dx + \frac{1}{16 \pi} \int_{\R^3} \int_{\R^3} \frac{u^2(x)\, u^2(y)}{|x - y|}\, dx\, dy,\\[7.5pt]
\M_s = \set{u \in E_r(\R^3) : \frac{1}{2} \int_{\R^3} |\nabla u|^2\, dx + \frac{1}{16 \pi} \int_{\R^3} \int_{\R^3} \frac{u^2(x)\, u^2(y)}{|x - y|}\, dx\, dy = 1},\\[7.5pt]
J_s(u) = \frac{1}{3} \int_{\R^3} |u|^3\, dx,\\[7.5pt]
\widetilde{\Psi}(u) = \frac{1}{J_s(u)}, \quad u \in \M_s,
\intertext{and}
\Phi_\lambda(u) = \frac{1}{2} \int_{\R^3} |\nabla u|^2\, dx + \frac{1}{16 \pi} \int_{\R^3} \int_{\R^3} \frac{u^2(x)\, u^2(y)}{|x - y|}\, dx\, dy - \frac{\lambda}{3} \int_{\R^3} |u|^3\, dx.
\end{gather*}
Clearly, $(H_{10})$ holds. Since
\[
I_s(tu) = \frac{t^2}{2} \int_{\R^3} |\nabla u|^2\, dx + \frac{t^4}{16 \pi} \int_{\R^3} \int_{\R^3} \frac{u^2(x)\, u^2(y)}{|x - y|}\, dx\, dy
\]
is a strictly increasing function of $t$ for each $u \in E_r(\R^3) \setminus \set{0}$, $(H_{11})$ also holds.

If $u \in E_r(\R^3)$ is a solution of problem \eqref{55}, then testing with $u$ itself gives
\begin{equation} \label{62}
\int_{\R^3} |\nabla u|^2\, dx + \frac{1}{4 \pi} \int_{\R^3} \int_{\R^3} \frac{u^2(x)\, u^2(y)}{|x - y|}\, dx\, dy = \lambda \int_{\R^3} |u|^3\, dx,
\end{equation}
and $u$ also satisfies the Poho\v{z}aev type identity
\begin{equation} \label{63}
\frac{1}{2} \int_{\R^3} |\nabla u|^2\, dx + \frac{5}{16 \pi} \int_{\R^3} \int_{\R^3} \frac{u^2(x)\, u^2(y)}{|x - y|}\, dx\, dy = \lambda \int_{\R^3} |u|^3\, dx
\end{equation}
by Ianni and Ruiz \cite[Proposition 2.5]{MR2902293}. Multiplying \eqref{62} by $2/3$ and \eqref{63} by $1/3$ and subtracting gives
\[
\frac{1}{2} \int_{\R^3} |\nabla u|^2\, dx + \frac{1}{16 \pi} \int_{\R^3} \int_{\R^3} \frac{u^2(x)\, u^2(y)}{|x - y|}\, dx\, dy = \frac{\lambda}{3} \int_{\R^3} |u|^3\, dx,
\]
so $(H_{12})$ also holds.

We denote the operators $f$ and $g$ and their potentials $F$ and $G$ appearing in Section \ref{Section 2} by $\widetilde{f}$, $\widetilde{g}$, $\widetilde{F}$, and $\widetilde{G}$, respectively, to avoid confusing them with the nonlinearities $f$ and $g$ and their primitives $F$ and $G$ in the introduction. We have
\[
\widetilde{f}(u) v = \int_{\R^3} f(|x|,u)\, v\, dx
\]
and
\[
\widetilde{g}(u) v = \int_{\R^3} g(|x|,u)\, v\, dx.
\]

\begin{lemma} \label{Lemma 3}
We have the following asymptotic estimates on $\widetilde{f}$ and $\widetilde{g}$:
\begin{enumroman}
\item \label{Lemma 3.i} If $f$ satisfies \eqref{53} with $18/7 < q_1 < q_2 < 3$, then
    \[
    \widetilde{f}(u_t) v_t = \o(t^3) \norm{v} \text{ as } t \to \infty
    \]
    uniformly in $u$ on bounded sets for all $v \in E_r(\R^3)$.
\item \label{Lemma 3.ii} If $g$ satisfies \eqref{65}, then
    \[
    \widetilde{g}(u_t) v_t = \o(t^3) \norm{v} \text{ as } t \to 0
    \]
    uniformly in $u$ on bounded sets for all $v \in E_r(\R^3)$.
\item \label{Lemma 3.iii} If $g$ satisfies \eqref{66}, then
    \[
    \widetilde{g}(u_t) v_t = \o(t^3) \norm{v} \text{ as } t \to \infty
    \]
    uniformly in $u$ on bounded sets for all $v \in E_r(\R^3)$.
\end{enumroman}
\end{lemma}

\begin{proof}
\ref{Lemma 3.i} For $t > 0$,
\[
\widetilde{f}(u_t) v_t = \int_{\R^3} f(|x|,t^2\, u(tx))\, t^2\, v(tx)\, dx = \frac{1}{t} \int_{\R^3} f(|x|/t,t^2\, u(x))\, v(x)\, dx.
\]
If $f$ satisfies \eqref{53}, then this together with the H\"{o}lder inequality gives
\begin{multline*}
|\widetilde{f}(u_t) v_t| \le a_1\, t^{2q_1 - 3} \int_{\R^3} |u|^{q_1 - 1}\, |v|\, dx + a_2\, t^{2q_2 - 3} \int_{\R^3} |u|^{q_2 - 1}\, |v|\, dx + \frac{1}{t} \int_{\R^3} a(x)\, |v|\, dx\\[7.5pt]
\le a_1\, t^{2q_1 - 3} \pnorm[q_1]{u}^{q_1 - 1} \pnorm[q_1]{v} + a_2\, t^{2q_2 - 3} \pnorm[q_2]{u}^{q_2 - 1} \pnorm[q_2]{v} + \frac{1}{t} \pnorm[r]{a} \pnorm[r']{v},
\end{multline*}
where $r' = r/(r-1) \in (18/7,6]$. The desired conclusion follows from this since $q_1, q_2 < 3$.

\ref{Lemma 3.ii} For $t > 0$,
\[
\widetilde{g}(u_t) v_t = \int_{\R^3} g(|x|,t^2\, u(tx))\, t^2\, v(tx)\, dx = \frac{1}{t} \int_{\R^3} g(|x|/t,t^2\, u(x))\, v(x)\, dx.
\]
If $g$ satisfies \eqref{65}, then this together with the H\"{o}lder inequality gives
\begin{multline*}
|\widetilde{g}(u_t) v_t| \le a_3\, t^{2 q_3 - 3} \int_{\R^3} |u|^{q_3 - 1}\, |v|\, dx + a_4\, t^{2 q_4 - 3} \int_{\R^3} |u|^{q_4 - 1}\, |v|\, dx\\[7.5pt]
\le a_3\, t^{2 q_3 - 3} \pnorm[q_3]{u}^{q_3 - 1} \pnorm[q_3]{v} + a_4\, t^{2 q_4 - 3} \pnorm[q_4]{u}^{q_4 - 1} \pnorm[q_4]{v}.
\end{multline*}
The desired conclusion follows from this since $q_3, q_4 > 3$.

\ref{Lemma 3.iii} Same as the proof of \ref{Lemma 3.i}.
\end{proof}

\subsection{Subcritical case}

In this section we prove the results of Section \ref{1.2.1}.

\subsubsection{Proofs of Theorem \ref{Theorem 10}, Theorem \ref{Theorem 11}, and Theorem \ref{Theorem 12}}

\begin{proof}[Proof of Theorem \ref{Theorem 10}]
Since \eqref{17} holds by Lemma \ref{Lemma 3} \ref{Lemma 3.ii}, Theorem \ref{Theorem 3} gives
\[
C^l(\Phi,0) \isom C^l(\Phi_\lambda,0) \quad \forall l \ge 0.
\]
The desired conclusions now follow from Theorem \ref{Theorem 2}.
\end{proof}

\begin{proof}[Proof of Theorem \ref{Theorem 11}]
Since \eqref{35} and \eqref{38} hold by Lemma \ref{Lemma 3} \ref{Lemma 3.i} and \ref{Lemma 3.ii}, respectively, the desired conclusions follow from Theorem \ref{Theorem 8}.
\end{proof}

\begin{proof}[Proof of Theorem \ref{Theorem 12}]
Since \eqref{27} holds by Lemma \ref{Lemma 3} \ref{Lemma 3.iii}, the desired conclusion follows from Theorem \ref{Theorem 6}.
\end{proof}

\subsubsection{Proof of Theorem \ref{Theorem 13}}

We will show that Theorem \ref{Theorem 13} follows from Theorem \ref{Theorem 9}. We have
\begin{multline*}
\Phi(u) = \frac{1}{2} \int_{\R^3} |\nabla u|^2\, dx + \frac{1}{16 \pi} \int_{\R^3} \int_{\R^3} \frac{u^2(x)\, u^2(y)}{|x - y|}\, dx\, dy - \frac{\lambda}{3} \int_{\R^3} |u|^3\, dx - \int_{\R^3} G(|x|,u)\, dx,\\[7.5pt]
u \in E_r(\R^3).
\end{multline*}
First we verify the \PS{} condition.

\begin{lemma} \label{Lemma 4}
If \eqref{65} and \eqref{61} hold, then $\Phi$ satisfies the {\em \PS{c}} condition for all $c \in \R$.
\end{lemma}

\begin{proof}
We will show that every \PS{c} sequence $\seq{u_j}$ of $\Phi$ is bounded. The desired conclusion will then follow as in Proposition \ref{Proposition 3}. Suppose $\norm{u_j} \to \infty$ for a renamed subsequence. Set
\[
t_j = t_{u_j} = \frac{1}{I_s(u_j)^{1/3}}, \quad \widetilde{u}_j = (u_j)_{t_j}, \quad \widetilde{t}_j = \frac{1}{t_j} = I_s(u_j)^{1/3}
\]
(see \eqref{13}). Then $\widetilde{u}_j \in \M_s$ and
\begin{equation} \label{67}
u_j = (\widetilde{u}_j)_{\widetilde{t}_j} = \widetilde{t}_j^{\,2}\, \widetilde{u}_j(\widetilde{t}_j\, \cdot)
\end{equation}
by \eqref{18}. Since $\M_s$ is a bounded manifold, $\seq{\widetilde{u}_j}$ is bounded. Since $\norm{u_j} \to \infty$, $I_s(u_j) \to \infty$ and hence $\widetilde{t}_j \to \infty$.

We have $\Phi(u_j) = c + \o(1)$ and $\Phi'(u_j) u_j = \o(\norm{u_j})$. Using \eqref{67} and noting that
\[
\norm{u_j} = \|(\widetilde{u}_j)_{\widetilde{t}_j}\| = \O(\widetilde{t}_j^{\,3/2}),
\]
these equations can be written as
\begin{multline} \label{68}
\widetilde{t}_j^{\,3} \left(\frac{1}{2} \int_{\R^3} |\nabla \widetilde{u}_j|^2\, dx + \frac{1}{16 \pi} \int_{\R^3} \int_{\R^3} \frac{\widetilde{u}_j^{\,2}(x)\, \widetilde{u}_j^{\,2}(y)}{|x - y|}\, dx\, dy - \frac{\lambda}{3} \int_{\R^3} |\widetilde{u}_j|^3\, dx\right)\\[7.5pt]
= \int_{\R^3} G(|x|,u_j)\, dx + c + \o(1)
\end{multline}
and
\begin{multline} \label{69}
\widetilde{t}_j^{\,3} \left(\int_{\R^3} |\nabla \widetilde{u}_j|^2\, dx + \frac{1}{4 \pi} \int_{\R^3} \int_{\R^3} \frac{\widetilde{u}_j^{\,2}(x)\, \widetilde{u}_j^{\,2}(y)}{|x - y|}\, dx\, dy - \lambda \int_{\R^3} |\widetilde{u}_j|^3\, dx\right)\\[7.5pt]
= \int_{\R^3} g(|x|,u_j)\, u_j\, dx + \o(\widetilde{t}_j^{\,3/2}),
\end{multline}
respectively. By \eqref{61},
\[
\int_{\R^3} G(|x|,u_j)\, dx \ge c_0 \int_{\R^3} |u_j|^q\, dx = c_0\, \widetilde{t}_j^{\,2q-3} \int_{\R^3} |\widetilde{u}_j|^q\, dx.
\]
Since $\widetilde{t}_j \to \infty$, combining this with \eqref{68} gives
\[
c_0\, \widetilde{t}_j^{\,2\,(q-3)} \int_{\R^3} |\widetilde{u}_j|^q\, dx \le \frac{1}{2} \int_{\R^3} |\nabla \widetilde{u}_j|^2\, dx + \frac{1}{16 \pi} \int_{\R^3} \int_{\R^3} \frac{\widetilde{u}_j^{\,2}(x)\, \widetilde{u}_j^{\,2}(y)}{|x - y|}\, dx\, dy - \frac{\lambda}{3} \int_{\R^3} |\widetilde{u}_j|^3\, dx + \o(1).
\]
Since $q > 3$ and the right-hand side is bounded, this implies that $\widetilde{u}_j \to 0$ in $L^q(\R^3)$. Since $\seq{\widetilde{u}_j}$ is bounded in $L^r(\R^3)$ for $r \in (18/7,3)$, then $\widetilde{u}_j \to 0$ in $L^3(\R^3)$ by interpolation. Now multiplying \eqref{68} by $q$, subtracting \eqref{69}, and using the second inequality in \eqref{61} gives
\[
\left(\frac{q}{2} - 1\right) \int_{\R^3} |\nabla \widetilde{u}_j|^2\, dx + \frac{1}{4 \pi} \left(\frac{q}{4} - 1\right) \int_{\R^3} \int_{\R^3} \frac{\widetilde{u}_j^{\,2}(x)\, \widetilde{u}_j^{\,2}(y)}{|x - y|}\, dx\, dy \le \o(1).
\]
Since $q > 4$, this implies that $\widetilde{u}_j \to 0$ in $E_r(\R^3)$, contradicting $\widetilde{u}_j \in \M_s$.
\end{proof}

We are now ready to prove Theorem \ref{Theorem 13}.

\begin{proof}[Proof of Theorem \ref{Theorem 13}]
We have
\[
\widetilde{G}(u) = \int_{\R^3} G(|x|,u)\, dx \ge c_0 \int_{\R^3} |u|^q\, dx
\]
by \eqref{61}, so \eqref{45} holds and
\[
\widetilde{F}(u) = \frac{\lambda}{3} \int_{\R^3} |u|^3\, dx + \widetilde{G}(u) \ge \frac{\lambda}{3} \int_{\R^3} |u|^3\, dx + c_0 \int_{\R^3} |u|^q\, dx.
\]
So
\[
\frac{\widetilde{F}(u_t)}{t^3} \ge \frac{\lambda}{3} \int_{\R^3} |u|^3\, dx + c_0\, t^{2\,(q-3)} \int_{\R^3} |u|^q\, dx
\]
for $t > 0$, which implies \eqref{42} since $q > 3$. Since \eqref{43} holds by Lemma \ref{Lemma 3} \ref{Lemma 3.ii} and $\Phi$ satisfies the \PS{} condition by Lemma \ref{Lemma 4}, the desired conclusion now follows from Theorem \ref{Theorem 9}.
\end{proof}

\subsubsection{Proofs of Theorem \ref{Theorem 19} and Theorem \ref{Theorem 18}}

We have
\[
\Phi(u) = I_s(u) - \lambda\, J_s(u) - \widetilde{G}(u), \quad u \in E_r(\R^3),
\]
where
\[
\widetilde{G}(u) = \int_{\R^3} G(u)\, dx.
\]
In particular,
\begin{equation} \label{104}
\Phi(u_t) = t^3\, \bigg(1 - \frac{\lambda}{\widetilde{\Psi}(u)}\bigg) - \widetilde{G}(u_t), \quad u \in \M_s,\, t \ge 0.
\end{equation}
By \eqref{101} and Lemma \ref{Lemma 3} \ref{Lemma 3.ii}, \eqref{43} holds and hence $\widetilde{G}(u_t) = \o(t^3)$ as $t \to 0$ uniformly in $u$ on bounded sets as in \eqref{50}. So
\begin{equation} \label{111}
\Phi(u_t) = t^3\, \bigg(1 - \frac{\lambda}{\widetilde{\Psi}(u)} + \o(1)\bigg) \text{ as } t \to 0
\end{equation}
uniformly in $u$ on $\M_s$. First we show that $\Phi$ has a scaled local linking near the origin when $G$ has a definite sign.

\begin{lemma} \label{Lemma 6}
If \eqref{101} holds, then $\Phi$ has a scaled local linking near the origin in dimension $k$ in each of the following cases:
\begin{enumroman}
\item \label{Lemma 6.i} $\lambda_k < \lambda \le \lambda_{k+1}$ and $G(t) < 0$ for all $t \in \R \setminus \set{0}$,
\item \label{Lemma 6.ii} $\lambda_k \le \lambda < \lambda_{k+1}$ and $G(t) \ge 0$ for all $t \in \R$.
\end{enumroman}
\end{lemma}

\begin{proof}
By Theorem \ref{Theorem 1} \ref{Theorem 1.iii},
\[
i(\widetilde{\Psi}^{\lambda_k}) = i(\M_s \setminus \widetilde{\Psi}_{\lambda_{k+1}}) = k,
\]
so \eqref{56} holds with $A_0 = \widetilde{\Psi}^{\lambda_k}$ and $B_0 = \widetilde{\Psi}_{\lambda_{k+1}}$.

\ref{Lemma 6.i} For $u \in A_0$ and $t \ge 0$, \eqref{111} gives
\[
\Phi(u_t) \le - t^3 \left(\frac{\lambda}{\lambda_k} - 1 + \o(1)\right) \text{ as } t \to 0
\]
uniformly in $u$. Since $\lambda > \lambda_k$, it follows that
\[
\Phi(u_t) \le 0 \quad \forall u \in A_0,\, 0 \le t \le \rho
\]
if $\rho > 0$ is sufficiently small. On the other hand, we have $\widetilde{G}(u) < 0$ for all $u \in E_r(\R^3) \setminus \set{0}$, so for $u \in B_0$ and $t > 0$, \eqref{104} gives
\[
\Phi(u_t) > t^3 \left(1 - \frac{\lambda}{\lambda_{k+1}}\right) \ge 0
\]
since $\lambda \le \lambda_{k+1}$. So \eqref{57} also holds.

\ref{Lemma 6.ii} We have $\widetilde{G}(u) \ge 0$ for all $u \in E_r(\R^3)$. So for $u \in A_0$ and $t \ge 0$, \eqref{104} gives
\[
\Phi(u_t) \le - t^3 \left(\frac{\lambda}{\lambda_k} - 1\right) \le 0
\]
since $\lambda \ge \lambda_k$. On the other hand, for $u \in B_0$ and $t \ge 0$, \eqref{111} gives
\[
\Phi(u_t) \ge t^3 \left(1 - \frac{\lambda}{\lambda_{k+1}} + \o(1)\right) \text{ as } t \to 0
\]
uniformly in $u$. Since $\lambda < \lambda_{k+1}$, it follows that
\[
\Phi(u_t) > 0 \quad \forall u \in B_0,\, 0 < t \le \rho
\]
if $\rho > 0$ is sufficiently small. So \eqref{57} also holds.
\end{proof}

We are now ready to prove Theorem \ref{Theorem 19} and Theorem \ref{Theorem 18}.

\begin{proof}[Proof of Theorem \ref{Theorem 19}]
Since $\lambda > \lambda_1$, $\lambda_k < \lambda \le \lambda_{k+1}$ for some $k \ge 1$. Then $\Phi$ has a scaled local linking near the origin in dimension $k$ by Lemma \ref{Lemma 6} and hence \[
C^k(\Phi,0) \ne 0
\]
by Theorem \ref{Theorem 16}. The desired conclusions now follow as in the proof of Theorem \ref{Theorem 8}.
\end{proof}

\begin{proof}[Proof of Theorem \ref{Theorem 18}]
The functional $\Phi$ satisfies the \PS{} condition by Lemma \ref{Lemma 4}. If $\lambda < \lambda_1$, then $\Phi$ has the mountain pass geometry as in the proof of Theorem \ref{Theorem 9}, so we may assume that $\lambda \ge \lambda_1$. Then $\lambda_k \le \lambda < \lambda_{k+1}$ for some $k \ge 1$. Then $\Phi$ has a scaled local linking near the origin by Lemma \ref{Lemma 6}. We will show that $\Phi^a$ is contractible for some $a < 0$. The desired conclusion will then follow from Corollary \ref{Corollary 2}.

Fix
\[
a < \inf_{u \in \M_s,\, 0 \le t \le 1}\, \Phi(u_t).
\]
We have
\begin{equation} \label{105}
\Phi^a = \set{u_t : u \in \M_s,\, t > 1,\, \varphi_u(t) \le a},
\end{equation}
where $\varphi_u(t) = \Phi(u_t)$. Since
\[
\widetilde{G}(u_t) = \int_{\R^3} G(t^2\, u(tx))\, dx = t^{-3} \int_{\R^3} G(t^2\, u(x))\, dx,
\]
\eqref{104} gives
\[
\varphi_u(t) = t^3\, \bigg(1 - \frac{\lambda}{\widetilde{\Psi}(u)}\bigg) - t^{-3} \int_{\R^3} G(t^2 u)\, dx.
\]
Then \eqref{102} gives
\[
\varphi_u(t) \le t^3\, \bigg(1 - \frac{\lambda}{\widetilde{\Psi}(u)}\bigg) - c_0\, t^{2q-3} \int_{\R^3} |u|^q\, dx,
\]
which implies that $\varphi_u(t) \to - \infty$ as $t \to \infty$ since $q > 3$. Moreover,
\begin{multline*}
\varphi_u'(t) = 3t^2\, \bigg(1 - \frac{\lambda}{\widetilde{\Psi}(u)}\bigg) + 3t^{-4} \int_{\R^3} G(t^2 u)\, dx - 2t^{-2} \int_{\R^3} g(t^2 u)\, u\, dx\\[7.5pt]
= \frac{3}{t}\, \varphi_u(t) + 2t^{-4} \int_{\R^3} \left(3G(t^2 u) - g(t^2 u)\, t^2 u\right) dx \le \frac{3}{t}\, \varphi_u(t)
\end{multline*}
by the second inequality in \eqref{102}, so $\varphi_u'(t) < 0$ whenever $\varphi_u(t) < 0$. So it follows from the implicit function theorem that there is a $C^1$-map $\sigma : \M_s \to (1,\infty)$ such that $\varphi_u(t) > a$ for $0 \le t < \sigma(u)$, $\varphi_u(\sigma(u)) = a$, and $\varphi_u(t) < a$ for $t > \sigma(u)$. Then \eqref{105} gives
\[
\Phi^a = \set{u_t : u \in \M_s,\, t \ge \sigma(u)},
\]
which is a deformation retract of the contractible set $E_r(\R^3) \setminus \set{0}$ via
\[
(E_r(\R^3) \setminus \set{0}) \times [0,1] \to E_r(\R^3) \setminus \set{0}, \quad (u,\tau) \mapsto \begin{cases}
\widetilde{u}_{(1 - \tau)/t_u + \tau\, \sigma(\widetilde{u})} & \text{if } u \notin \Phi^a\\[5pt]
u & \text{if } u \in \Phi^a
\end{cases}
\]
(see \eqref{13} and \eqref{18}).
\end{proof}

\subsection{Critical case}

In this section we prove the results of Section \ref{1.2.2}. We have
\begin{multline*}
\Phi(u) = \frac{1}{2} \int_{\R^3} |\nabla u|^2\, dx + \frac{1}{16 \pi} \int_{\R^3} \int_{\R^3} \frac{u^2(x)\, u^2(y)}{|x - y|}\, dx\, dy - \frac{\lambda}{q} \int_{\R^3} |u|^q\, dx - \frac{1}{6} \int_{\R^3} |u|^6\, dx,\\[7.5pt]
u \in E_r(\R^3).
\end{multline*}
First we prove a local \PS{} condition. Let
\begin{equation} \label{88}
S = \inf_{u \in \D^{1,\,2}(\R^3) \setminus \set{0}}\, \frac{\dint_{\R^3} |\nabla u|^2\, dx}{\left(\dint_{\R^3} |u|^6\, dx\right)^{1/3}}
\end{equation}
be the best Sobolev constant.

\begin{lemma} \label{Lemma 5}
If $\lambda > 0$ and $q \in [3,6)$, then $\Phi$ satisfies the {\em \PS{c}} condition for $0 < c < \dfrac{1}{3}\, S^{3/2}$.
\end{lemma}

\begin{proof}
Let $\seq{u_j} \subset E_r(\R^3)$ be a \PS{c} sequence of $\Phi$. Then
\begin{equation} \label{81}
\frac{1}{2} \int_{\R^3} |\nabla u_j|^2\, dx + \frac{1}{16 \pi} \int_{\R^3} \int_{\R^3} \frac{u_j^2(x)\, u_j^2(y)}{|x - y|}\, dx\, dy - \frac{\lambda}{q} \int_{\R^3} |u_j|^q\, dx - \frac{1}{6} \int_{\R^3} |u_j|^6\, dx = c + \o(1)
\end{equation}
and
\begin{multline} \label{83}
\int_{\R^3} \nabla u_j \cdot \nabla v\, dx + \frac{1}{4 \pi} \int_{\R^3} \int_{\R^3} \frac{u_j^2(x)\, u_j(y)\, v(y)}{|x - y|}\, dx\, dy - \lambda \int_{\R^3} |u_j|^{q-2}\, u_j v\, dx - \int_{\R^3} u_j^5\, v\, dx\\[7.5pt]
= \o(\norm{v})
\end{multline}
for all $v \in E_r(\R^3)$. Taking $v = u_j$ in \eqref{83} gives
\begin{equation} \label{82}
\int_{\R^3} |\nabla u_j|^2\, dx + \frac{1}{4 \pi} \int_{\R^3} \int_{\R^3} \frac{u_j^2(x)\, u_j^2(y)}{|x - y|}\, dx\, dy - \lambda \int_{\R^3} |u_j|^q\, dx - \int_{\R^3} |u_j|^6\, dx = \o(\norm{u_j}).
\end{equation}

First we show that $\seq{u_j}$ is bounded in $E_r(\R^3)$. Suppose $\norm{u_j} \to \infty$ for a renamed subsequence. Set
\[
t_j = t_{u_j} = \frac{1}{I_s(u_j)^{1/3}}, \quad \widetilde{u}_j = (u_j)_{t_j}, \quad \widetilde{t}_j = \frac{1}{t_j} = I_s(u_j)^{1/3}
\]
(see \eqref{13}). Then $\widetilde{u}_j \in \M_s$ and
\begin{equation} \label{167}
u_j = (\widetilde{u}_j)_{\widetilde{t}_j} = \widetilde{t}_j^{\,2}\, \widetilde{u}_j(\widetilde{t}_j\, \cdot)
\end{equation}
by \eqref{18}. Since $\M_s$ is a bounded manifold, $\seq{\widetilde{u}_j}$ is bounded. Since $\norm{u_j} \to \infty$, $I_s(u_j) \to \infty$ and hence $\widetilde{t}_j \to \infty$. Using \eqref{167} and noting that
\[
\norm{u_j} = \|(\widetilde{u}_j)_{\widetilde{t}_j}\| = \O(\widetilde{t}_j^{\,3/2}),
\]
\eqref{81} and \eqref{82} can be written as
\begin{multline} \label{168}
\widetilde{t}_j^{\,3} \left(\frac{1}{2} \int_{\R^3} |\nabla \widetilde{u}_j|^2\, dx + \frac{1}{16 \pi} \int_{\R^3} \int_{\R^3} \frac{\widetilde{u}_j^{\,2}(x)\, \widetilde{u}_j^{\,2}(y)}{|x - y|}\, dx\, dy\right)\\[7.5pt]
= \frac{\widetilde{t}_j^{\,9}}{6} \int_{\R^3} |\widetilde{u}_j|^6\, dx + \frac{\lambda\, \widetilde{t}_j^{\,2q-3}}{q} \int_{\R^3} |\widetilde{u}_j|^q\, dx + c + \o(1)
\end{multline}
and
\begin{multline} \label{169}
\widetilde{t}_j^{\,3} \left(\int_{\R^3} |\nabla \widetilde{u}_j|^2\, dx + \frac{1}{4 \pi} \int_{\R^3} \int_{\R^3} \frac{\widetilde{u}_j^{\,2}(x)\, \widetilde{u}_j^{\,2}(y)}{|x - y|}\, dx\, dy\right)\\[7.5pt]
= \widetilde{t}_j^{\,9} \int_{\R^3} |\widetilde{u}_j|^6\, dx + \lambda\, \widetilde{t}_j^{\,2q-3} \int_{\R^3} |\widetilde{u}_j|^q\, dx + \o(\widetilde{t}_j^{\,3/2}),
\end{multline}
respectively. Since $\widetilde{t}_j \to \infty$, dividing \eqref{169} by $\widetilde{t}_j^{\,3}$ gives
\[
\widetilde{t}_j^{\,6} \int_{\R^3} |\widetilde{u}_j|^6\, dx + \lambda\, \widetilde{t}_j^{\,2\,(q-3)} \int_{\R^3} |\widetilde{u}_j|^q\, dx = \int_{\R^3} |\nabla \widetilde{u}_j|^2\, dx + \frac{1}{4 \pi} \int_{\R^3} \int_{\R^3} \frac{\widetilde{u}_j^{\,2}(x)\, \widetilde{u}_j^{\,2}(y)}{|x - y|}\, dx\, dy + \o(1).
\]
Since the right-hand side is bounded and $\lambda > 0$, this in turn gives
\[
\int_{\R^3} |\widetilde{u}_j|^6\, dx = \O(\widetilde{t}_j^{\,-6}).
\]
Since $\seq{\widetilde{u}_j}$ is bounded in $L^r(\R^3)$ for $r \in (18/7,3)$, then the interpolation inequality
\[
\pnorm[q]{\widetilde{u}_j} \le \pnorm[r]{\widetilde{u}_j}^{1 - \theta} \pnorm[6]{\widetilde{u}_j}^\theta,
\]
where $1/q = (1 - \theta)/r + \theta/6$, gives
\[
\widetilde{t}_j^{\,2\,(q-3)} \int_{\R^3} |\widetilde{u}_j|^q\, dx = \O(\widetilde{t}_j^{\,-2(6-q)(3-r)/(6-r)}),
\]
so
\[
\widetilde{t}_j^{\,2\,(q-3)} \int_{\R^3} |\widetilde{u}_j|^q\, dx \to 0.
\]
Now multiplying \eqref{168} by $6$, subtracting \eqref{169}, and dividing by $\widetilde{t}_j^{\,3}$ gives
\[
2 \int_{\R^3} |\nabla \widetilde{u}_j|^2\, dx + \frac{1}{8 \pi} \int_{\R^3} \int_{\R^3} \frac{\widetilde{u}_j^{\,2}(x)\, \widetilde{u}_j^{\,2}(y)}{|x - y|}\, dx\, dy = \o(1).
\]
This implies that $\widetilde{u}_j \to 0$ in $E_r(\R^3)$, contradicting $\widetilde{u}_j \in \M_s$.

Since $\seq{u_j}$ is bounded and $E_r(\R^3)$ is a reflexive Banach space, a renamed subsequence of $\seq{u_j}$ converges weakly to some $u \in E_r(\R^3)$. Then $u_j$ also converges to $u$ strongly in $L^q(\R^3)$, weakly in $L^6(\R^3)$, and a.e.\! in $\R^3$ for a further subsequence. Now passing to the limit in \eqref{83} using Lemma 2.2 and Lemma 2.3 in Ianni and Ruiz \cite{MR2902293} gives
\[
\int_{\R^3} \nabla u \cdot \nabla v\, dx + \frac{1}{4 \pi} \int_{\R^3} \int_{\R^3} \frac{u^2(x)\, u(y)\, v(y)}{|x - y|}\, dx\, dy - \lambda \int_{\R^3} |u|^{q-2}\, uv\, dx - \int_{\R^3} u^5\, v\, dx = 0
\]
for all $v \in E_r(\R^3)$. Taking $v = u$ gives
\begin{equation} \label{84}
\int_{\R^3} |\nabla u|^2\, dx + \frac{1}{4 \pi} \int_{\R^3} \int_{\R^3} \frac{u^2(x)\, u^2(y)}{|x - y|}\, dx\, dy - \lambda \int_{\R^3} |u|^q\, dx - \int_{\R^3} |u|^6\, dx = 0,
\end{equation}
and $u$ also satisfies the Poho\v{z}aev type identity
\begin{equation} \label{163}
\frac{1}{2} \int_{\R^3} |\nabla u|^2\, dx + \frac{5}{16 \pi} \int_{\R^3} \int_{\R^3} \frac{u^2(x)\, u^2(y)}{|x - y|}\, dx\, dy - \frac{3 \lambda}{q} \int_{\R^3} |u|^q\, dx - \frac{1}{2} \int_{\R^3} |u|^6\, dx = 0,
\end{equation}
which can be shown, e.g., as in Dutko et al.\! \cite[Lemma 2.4]{MR4292779}.

Set $v_j = u_j - u$. We will show that $v_j \to 0$ in $E_r(\R^3)$ for a renamed subsequence. By the Brezis-Lieb lemma and Mercuri et al.\! \cite[Proposition 4.1]{MR3568051},
\begin{equation} \label{85}
\int_{\R^3} |\nabla u_j|^2\, dx - \int_{\R^3} |\nabla u|^2\, dx = \int_{\R^3} |\nabla v_j|^2\, dx + \o(1),
\end{equation}
\begin{multline} \label{90}
\int_{\R^3} \int_{\R^3} \frac{u_j^2(x)\, u_j^2(y)}{|x - y|}\, dx\, dy - \int_{\R^3} \int_{\R^3} \frac{u^2(x)\, u^2(y)}{|x - y|}\, dx\, dy\\[7.5pt]
\ge \int_{\R^3} \int_{\R^3} \frac{v_j^{\,2}(x)\, v_j^{\,2}(y)}{|x - y|}\, dx\, dy + \o(1),
\end{multline}
and
\begin{equation} \label{86}
\int_{\R^3} |u_j|^6\, dx - \int_{\R^3} |u|^6\, dx = \int_{\R^3} |v_j|^6\, dx + \o(1).
\end{equation}
Subtracting \eqref{84} from \eqref{82} and combining with \eqref{85}--\eqref{86} and \eqref{88} gives
\begin{multline} \label{87}
\int_{\R^3} |\nabla v_j|^2\, dx + \frac{1}{4 \pi} \int_{\R^3} \int_{\R^3} \frac{v_j^{\,2}(x)\, v_j^{\,2}(y)}{|x - y|}\, dx\, dy \le \int_{\R^3} |v_j|^6\, dx + \o(1)\\[7.5pt]
\le S^{-3} \left(\int_{\R^3} |\nabla v_j|^2\, dx\right)^3 + \o(1).
\end{multline}
So it suffices to show that $\int_{\R^3} |\nabla v_j|^2\, dx \to 0$ for a renamed subsequence. Suppose this is not the case. Then \eqref{87} gives
\begin{equation} \label{89}
\int_{\R^3} |\nabla v_j|^2\, dx \ge S^{3/2} + \o(1).
\end{equation}
Dividing \eqref{82} by $6$ and subtracting from \eqref{81} gives
\begin{equation} \label{91}
c = \frac{1}{3} \int_{\R^3} |\nabla u_j|^2\, dx + \frac{1}{48 \pi} \int_{\R^3} \int_{\R^3} \frac{u_j^2(x)\, u_j^2(y)}{|x - y|}\, dx\, dy - \lambda \left(\frac{1}{q} - \frac{1}{6}\right) \int_{\R^3} |u|^q\, dx + \o(1).
\end{equation}
Since
\[
\int_{\R^3} |\nabla u_j|^2\, dx \ge S^{3/2} + \int_{\R^3} |\nabla u|^2\, dx + \o(1)
\]
by \eqref{85} and \eqref{89}, and
\[
\int_{\R^3} \int_{\R^3} \frac{u_j^2(x)\, u_j^2(y)}{|x - y|}\, dx\, dy \ge \int_{\R^3} \int_{\R^3} \frac{u^2(x)\, u^2(y)}{|x - y|}\, dx\, dy + \o(1)
\]
by \eqref{90}, \eqref{91} gives
\[
c \ge \frac{1}{3}\, S^{3/2} + \frac{1}{3} \int_{\R^3} |\nabla u|^2\, dx + \frac{1}{48 \pi} \int_{\R^3} \int_{\R^3} \frac{u^2(x)\, u^2(y)}{|x - y|}\, dx\, dy - \lambda \left(\frac{1}{q} - \frac{1}{6}\right) \int_{\R^3} |u|^q\, dx.
\]
Multiplying \eqref{84} by $1/2$ and \eqref{163} by $-1/3$ and adding to this inequality gives
\begin{equation} \label{187}
c \ge \frac{1}{3}\, S^{3/2} + 2 \lambda \left(\frac{1}{3} - \frac{1}{q}\right) \int_{\R^3} |u|^q\, dx + \frac{1}{3} \int_{\R^3} |u|^6\, dx \ge \frac{1}{3}\, S^{3/2}
\end{equation}
since $\lambda > 0$ and $q \ge 3$, contrary to assumption.
\end{proof}

\subsubsection{Proof of Theorem \ref{Theorem 25}}

In view of Lemma \ref{Lemma 5}, we apply Theorem \ref{Theorem 14} with $c^\ast = \dfrac{1}{3}\, S^{3/2}$. Let $\eps \in (0,\lambda_{k+m} - \lambda_{k+m-1})$. Then
\[
i(\M_s \setminus \widetilde{\Psi}_{\lambda_{k+m-1} + \eps}) = k + m - 1
\]
by Theorem \ref{Theorem 1} \ref{Theorem 1.iii}. Since $\M_s \setminus \widetilde{\Psi}_{\lambda_{k+m-1} + \eps}$ is an open symmetric subset of $\M_s$, then it has a compact symmetric subset $C$ of index $k + m - 1$ (see the proof of Proposition 3.1 in Degiovanni and Lancelotti \cite{MR2371112}). We apply Theorem \ref{Theorem 14} with $A_0 = C$ and $B_0 = \widetilde{\Psi}_{\lambda_k}$. We have either $\lambda_1 = \cdots = \lambda_k$, or $\lambda_{l-1} < \lambda_l = \cdots = \lambda_k$ for some $2 \le l \le k$. In the latter case,
\[
i(\M_s \setminus B_0) = i(\M_s \setminus \widetilde{\Psi}_{\lambda_l}) = l - 1 \le k - 1
\]
by Theorem \ref{Theorem 1} \ref{Theorem 1.iii}. In the former case, $B_0 = \widetilde{\Psi}_{\lambda_1} = \M_s$ by Theorem \ref{Theorem 1} \ref{Theorem 1.i} and hence
\[
i(\M_s \setminus B_0) = 0 \le k - 1
\]
by Proposition \ref{Proposition 7} $(i_1)$.

Let $R > \rho > 0$ and let
\begin{gather*}
X = \set{u_t : u \in A_0,\, 0 \le t \le R},\\
A = \set{u_R : u \in A_0},\\
B = \set{u_\rho : u \in B_0}.
\end{gather*}
For $u \in \M_s$ and $t \ge 0$, we have
\begin{equation} \label{180}
\Phi(u_t) = t^3\, \bigg(1 - \frac{\lambda}{\widetilde{\Psi}(u)} - \frac{t^6}{6} \int_{\R^3} |u|^6\, dx\bigg).
\end{equation}
Since $\M_s$ is bounded, this gives
\[
\Phi(u_t) \ge t^3 \left(1 - \frac{\lambda}{\lambda_k} - c_1\, t^6\right) \quad \forall u \in B_0,\, t \ge 0
\]
for some constant $c_1 > 0$. So
\[
\inf_{u \in B}\, \Phi(u) > 0
\]
if $\lambda < \lambda_k$ and $\rho$ is sufficiently small. For $u \in A_0 \subset \M_s \setminus \widetilde{\Psi}_{\lambda_{k+m-1} + \eps}$,
\[
J_s(u) = \frac{1}{3} \int_{\R^3} |u|^3\, dx > \frac{1}{\lambda_{k+m-1} + \eps} > \frac{1}{\lambda_{k+m}}
\]
since $\eps < \lambda_{k+m} - \lambda_{k+m-1}$. Since $\M_s$ is bounded in $L^r(\R^3)$ for $r \in (18/7,3)$, then the interpolation inequality
\[
\pnorm[3]{u} \le \pnorm[r]{u}^{1 - \theta} \pnorm[6]{u}^\theta,
\]
where $1/3 = (1 - \theta)/r + \theta/6$, implies that $\pnorm[6]{u}$ is bounded from below by a constant independent of $\eps$. Then \eqref{180} together with $\lambda_{k+m-1} = \lambda_k$ gives
\[
\Phi(u_t) \le t^3 \left(1 - \frac{\lambda}{\lambda_k + \eps} - c_2\, t^6\right) \quad \forall u \in A_0,\, t \ge 0
\]
for some constant $c_2 > 0$ independent of $\eps$. So
\[
\sup_{u \in A}\, \Phi(u) \le R^3 \left(1 - c_2\, R^6\right) \le 0
\]
if $R$ is sufficiently large, and
\begin{equation} \label{181}
\sup_{u \in X}\, \Phi(u) \le \max_{t \ge 0}\, t^3 \left(1 - \frac{\lambda}{\lambda_k + \eps} - c_2\, t^6\right) = \frac{2}{3 \sqrt{3 c_2}} \left(1 - \frac{\lambda}{\lambda_k + \eps}\right)^{3/2}.
\end{equation}
Let
\[
\delta_k = \left(\frac{3 c_2}{4}\right)^{1/3} \lambda_k\, S.
\]
Then it is easily seen that the last expression in \eqref{181} is less than $\dfrac{1}{3}\, S^{3/2}$ if $\lambda > \lambda_k - \delta_k$ and $\eps$ is sufficiently small. The desired conclusion now follows from Theorem \ref{Theorem 14}.

\subsubsection{Proof of Theorem \ref{Theorem 15}}

We apply Corollary \ref{Corollary 1} with $c^\ast = \dfrac{1}{3}\, S^{3/2}$. Let $A_0$ be a compact symmetric subset of $\M_s$ with $i(A_0) = m$, let $R > \rho > 0$, and let
\[
A = \set{u_R : u \in A_0}, \qquad X = \set{u_t : u \in A_0,\, 0 \le t \le R}.
\]
For $u \in \M_s$ and $t \ge 0$, we have
\begin{equation} \label{80}
\Phi(u_t) = t^3 \left(1 - \frac{\lambda\, t^{2\,(q-3)}}{q} \int_{\R^3} |u|^q\, dx - \frac{t^6}{6} \int_{\R^3} |u|^6\, dx\right).
\end{equation}
Since $q > 3$ and $\M_s$ is bounded, this implies that
\[
\inf_{u \in \M_s}\, \Phi(u_\rho) > 0
\]
if $\rho$ is sufficiently small. Since $A_0$ is compact and $0 \notin A_0$, \eqref{80} also gives
\[
\Phi(u_t) \le t^3 \left(1 - c_1\, \lambda\, t^{2\,(q-3)} - c_2\, t^6\right) \quad \forall u \in A_0,\, t \ge 0
\]
for some constants $c_1, c_2 > 0$. So
\[
\sup_{u \in A}\, \Phi(u) \le R^3 \left(1 - c_2\, R^6\right) \le 0
\]
if $R$ is sufficiently large and
\[
\sup_{u \in X}\, \Phi(u) \le \max_{t \ge 0}\, t^3 \left(1 - c_1\, \lambda\, t^{2\,(q-3)}\right) = 2\, (q - 3)\, \frac{(3/c_1\, \lambda)^{3/2\,(q-3)}}{(2q - 3)^{(2q-3)/2\,(q-3)}} < \frac{1}{3}\, S^{3/2}
\]
if $\lambda > 0$ is sufficiently large. The desired conclusion now follows from Corollary \ref{Corollary 1}.

\subsection{Equations subscaled near zero}

In this section we prove the results of Section \ref{1.2.3}. The proofs will be based on Corollary \ref{Corollary 1} and the following standard result in critical point theory (see, e.g., Perera et al.\! \cite[Proposition 3.36]{MR2640827}).

\begin{proposition} \label{Proposition 11}
Let $\Phi$ be an even $C^1$-functional on a Banach space $W$ such that $\Phi(0) = 0$ and $\Phi$ satisfies the {\em \PS{c}} condition for all $c < 0$. Let $\F$ denote the class of symmetric subsets of $W \setminus \set{0}$. For $k \ge 1$, let
\[
\F_k = \bgset{M \in \F : i(M) \ge k}
\]
and set
\begin{equation} \label{121}
c_k := \inf_{M \in \F_k}\, \sup_{u \in M}\, \Phi(u).
\end{equation}
If there exists a $k_0 \ge 1$ such that $- \infty < c_k < 0$ for all $k \ge k_0$, then $c_{k_0} \le c_{k_0 + 1} \le \cdots \to 0$ is a sequence of critical values of $\Phi$.
\end{proposition}

Set
\begin{equation} \label{200}
K(u) = \frac{1}{\sigma} \int_{\R^3} |u|^\sigma\, dx.
\end{equation}

\begin{lemma} \label{Lemma 7}
If $\sigma \in (18/7,3)$, then
\[
\inf_{u \in \widetilde{\Psi}^a}\, K(u) > 0 \quad \forall a \ge \lambda_1.
\]
\end{lemma}

\begin{proof}
For $u \in \widetilde{\Psi}^a$,
\[
\int_{\R^3} |u|^3\, dx = \frac{3}{\widetilde{\Psi}(u)} \ge \frac{3}{a}, \qquad \int_{\R^3} |\nabla u|^2\, dx \le 2 I_s(u) = 2.
\]
The desired conclusion now follows from the Gagliardo-Nirenberg inequality
\[
\left(\int_{\R^3} |u|^3\, dx\right)^{1/3} \le C \left(\int_{\R^3} |\nabla u|^2\, dx\right)^{\theta/2} \left(\int_{\R^3} |u|^\sigma\, dx\right)^{(1 - \theta)/\sigma},
\]
where $C > 0$ is a constant and $\theta = 2\, (3 - \sigma)/(6 - \sigma) \in (0,1)$.
\end{proof}

\subsubsection{Proofs of Theorem \ref{Theorem 20} and Corollary \ref{Corollary 3}}

\begin{proof}[Proof of Theorem \ref{Theorem 20}]
Since $\sigma, \widetilde{\sigma} \in (18/7,3)$, the operator
\[
\widetilde{f}(u) v = \int_{\R^3} \left(|u|^{\sigma - 2}\, u + \widetilde{\sigma}\, C\, |u|^{\widetilde{\sigma} - 2}\, u\right) v\, dx
\]
that corresponds to the nonlinearity $f(t) = |t|^{\sigma - 2}\, t + \widetilde{\sigma}\, C\, |t|^{\widetilde{\sigma} - 2}\, t$ satisfies
\[
\widetilde{f}(u_t) v_t = \o(t^3) \norm{v} \text{ as } t \to \infty
\]
uniformly in $u$ on bounded sets for all $v \in E_r(\R^3)$ by Lemma \ref{Lemma 3} \ref{Lemma 3.i}. So the corresponding functional
\begin{multline*}
\widetilde{\Phi}(u) = \frac{1}{2} \int_{\R^3} |\nabla u|^2\, dx + \frac{1}{16 \pi} \int_{\R^3} \int_{\R^3} \frac{u^2(x)\, u^2(y)}{|x - y|}\, dx\, dy - \frac{1}{\sigma} \int_{\R^3} |u|^\sigma\, dx - C \int_{\R^3} |u|^{\widetilde{\sigma}}\, dx,\\[7.5pt]
u \in E_r(\R^3)
\end{multline*}
is coercive by Lemma \ref{Lemma 2}. Since $\Phi \ge \widetilde{\Phi}$ by \eqref{305}, then so is $\Phi$. Since $\Phi$ is bounded on bounded sets, this implies that it is bounded from below.

We have
\[
\Phi(u) = I_s(u) - K(u) - \widetilde{G}(u), \quad u \in E_r(\R^3),
\]
where $K$ is as in \eqref{200} and
\[
\widetilde{G}(u) = \int_{\R^3} G(|x|,u)\, dx.
\]
In particular,
\[
\Phi(u_t) = t^3 - t^{2 \sigma - 3}\, K(u) - \widetilde{G}(u_t), \quad u \in \M_s,\, t \ge 0.
\]
By \eqref{301},
\[
|G(|x|,t)| \le \frac{a_6}{q_6}\, |t|^{q_6} + \frac{a_7}{q_7}\, |t|^{q_7} \quad \text{for a.a.\! } x \in \R^3 \text{ and all } t \in \R
\]
and hence
\begin{multline*}
|\widetilde{G}(u_t)| = \abs{\int_{\R^3} G(|x|,t^2\, u(tx))\, dx} \le \frac{1}{t^3} \int_{\R^3} |G(|x|/t,t^2\, u(x))|\, dx\\[7.5pt]
\le \frac{a_6}{q_6}\, t^{2 q_6 - 3} \pnorm[q_6]{u}^{q_6} + \frac{a_7}{q_7}\, t^{2 q_7 - 3} \pnorm[q_7]{u}^{q_7} = \o(t^{2 \sigma - 3}) \text{ as } t \to 0
\end{multline*}
uniformly in $u \in \M_s$ since $q_6, q_7 > \sigma$. This together with the assumption that $\sigma < 3$ now gives
\begin{equation} \label{304}
\Phi(u_t) = - t^{2 \sigma - 3}\, [K(u) + \o(1)] \text{ as } t \to 0
\end{equation}
uniformly in $u \in \M_s$.

Let $c_k$ be as in Proposition \ref{Proposition 11} and note that $c_k > - \infty$ since $\Phi$ is bounded from below. Let $M_t = \bgset{u_t : u \in \widetilde{\Psi}^{\lambda_k}}$ for $t > 0$. Since the mapping $\widetilde{\Psi}^{\lambda_k} \to M_t,\, u \mapsto u_t$ is an odd homeomorphism, $i(M_t) = i(\widetilde{\Psi}^{\lambda_k})$ by Proposition \ref{Proposition 7} $(i_2)$. We have $\lambda_k = \dotsb = \lambda_{k+m-1} < \lambda_{k+m}$ for some $m \ge 1$ and hence
\[
i(\widetilde{\Psi}^{\lambda_k}) = i(\widetilde{\Psi}^{\lambda_{k+m-1}}) = k + m - 1 \ge k
\]
by Theorem \ref{Theorem 1} \ref{Theorem 1.iii}, so $i(M_t) \ge k$. So $M_t \in \F_k$ and hence
\begin{equation} \label{303}
c_k \le \sup_{u \in M_t}\, \Phi(u) = \sup_{u \in \widetilde{\Psi}^{\lambda_k}}\, \Phi(u_t).
\end{equation}
Since
\[
\inf_{u \in \widetilde{\Psi}^{\lambda_k}}\, K(u) > 0
\]
by Lemma \ref{Lemma 7}, taking $t$ sufficiently small in \eqref{304} and combining with \eqref{303} gives $c_k < 0$. The desired conclusion now follows from Proposition \ref{Proposition 11}.
\end{proof}

\begin{proof}[Proof of Corollary \ref{Corollary 3}]
As in the proof of Theorem \ref{Theorem 20}, the functional
\[
\Phi(u) = \frac{1}{2} \int_{\R^3} |\nabla u|^2\, dx + \frac{1}{16 \pi} \int_{\R^3} \int_{\R^3} \frac{u^2(x)\, u^2(y)}{|x - y|}\, dx\, dy - \frac{1}{\sigma} \int_{\R^3} |u|^\sigma\, dx, \quad u \in E_r(\R^3)
\]
is coercive. So every \PS{} sequence of $\Phi$ is bounded and hence $\Phi$ satisfies the \PS{} condition by Proposition \ref{Proposition 3}.
\end{proof}

\subsubsection{Proof of Theorem \ref{Theorem 29}}

Let
\[
\widetilde{g}(|x|,t) = \begin{cases}
g(|x|,t) & \text{if } |t| < t_0\\[7.5pt]
- |t|^{\sigma - 2}\, t & \text{if } |t| \ge t_0.
\end{cases}
\]
Then $\widetilde{g}$ is a Carath\'{e}odory function on $[0,\infty) \times \R$ by \eqref{307}, and \eqref{301}--\eqref{305} hold with $\widetilde{g}$ and $\widetilde{G}(|x|,t) = \int_0^t \widetilde{g}(|x|,\tau)\, d\tau$ in place of $g$ and $G$, respectively. As in the proof of Theorem \ref{Theorem 20}, the variational functional
\begin{multline*}
\widetilde{\Phi}(u) = \frac{1}{2} \int_{\R^3} |\nabla u|^2\, dx + \frac{1}{16 \pi} \int_{\R^3} \int_{\R^3} \frac{u^2(x)\, u^2(y)}{|x - y|}\, dx\, dy - \frac{1}{\sigma} \int_{\R^3} |u|^\sigma\, dx - \int_{\R^3} \widetilde{G}(|x|,u)\, dx,\\[7.5pt]
u \in E_r(\R^3)
\end{multline*}
associated with the truncated equation
\begin{equation} \label{308}
- \Delta u + \left(\frac{1}{4 \pi |x|} \star u^2\right) u = |u|^{\sigma - 2}\, u + \widetilde{g}(|x|,u) \quad \text{in } \R^3
\end{equation}
is coercive. So every \PS{} sequence of $\widetilde{\Phi}$ is bounded and hence $\widetilde{\Phi}$ satisfies the \PS{} condition by Proposition \ref{Proposition 3}. Thus, equation \eqref{308} has a sequence of solutions $\seq{u_k}$ such that $\widetilde{\Phi}(u_k) < 0$ for all $k$ and $\widetilde{\Phi}(u_k) \nearrow 0$ by Theorem \ref{Theorem 20}. Since $|t|^{\sigma - 2}\, t + \widetilde{g}(|x|,t) = 0$ for $|t| \ge t_0$, $|u_k| < t_0$ by the strong maximum principle, so $u_k$ is also a solution of \eqref{300} and $\widetilde{\Phi}(u_k) = \Phi(u_k)$.

\subsubsection{Proof of Theorem \ref{Theorem 24}}

Since $\sigma \in (18/7,3)$, the operator
\[
\widetilde{g}(u) v = \int_{\R^3} |u|^{\sigma - 2}\, uv\, dx
\]
that corresponds to the term $g(t) = |t|^{\sigma - 2}\, t$ satisfies
\[
\widetilde{g}(u_t) v_t = \o(t^3) \norm{v} \text{ as } t \to \infty
\]
uniformly in $u$ on bounded sets for all $v \in E_r(\R^3)$ by Lemma \ref{Lemma 3} \ref{Lemma 3.iii}. Since $\lambda$ is not an eigenvalue of problem \eqref{55}, then the associated variational functional
\begin{multline*}
\Phi(u) = \frac{1}{2} \int_{\R^3} |\nabla u|^2\, dx + \frac{1}{16 \pi} \int_{\R^3} \int_{\R^3} \frac{u^2(x)\, u^2(y)}{|x - y|}\, dx\, dy - \frac{1}{\sigma} \int_{\R^3} |u|^\sigma\, dx - \frac{\lambda}{3} \int_{\R^3} |u|^3\, dx,\\[7.5pt]
u \in E_r(\R^3)
\end{multline*}
satisfies the \PS{} condition by Lemma \ref{Lemma 1}.

We apply Proposition \ref{Proposition 11}, choosing $k_0$ as follows. Since $\lambda$ is not an eigenvalue of problem \eqref{55}, either $\lambda < \lambda_1$, in which case we take $k_0 = 1$, or $\lambda_{k_0 - 1} < \lambda < \lambda_{k_0}$ for some $k_0 \ge 2$. In the latter case,
\begin{equation} \label{120}
i(\M_s \setminus \widetilde{\Psi}_{\lambda_{k_0}}) = k_0 - 1
\end{equation}
by Theorem \ref{Theorem 1} \ref{Theorem 1.iii}. In the former case, $\widetilde{\Psi}_{\lambda_{k_0}} = \M_s$ by Theorem \ref{Theorem 1} \ref{Theorem 1.i}, so \eqref{120} holds in this case also by Proposition \ref{Proposition 7} $(i_1)$. For $k \ge k_0$, let $c_k$ be as in \eqref{121} and note that $c_k < 0$ as in the proof of Theorem \ref{Theorem 20}. It only remains to show that $c_k > - \infty$.

Let $M \in \F_k$. Then
\begin{equation} \label{122}
i(M) \ge k \ge k_0.
\end{equation}
We claim that $M$ intersects the set
\[
N = \bgset{u_t : u \in \widetilde{\Psi}_{\lambda_{k_0}},\, t \ge 0}.
\]
Suppose this is not the case. Then the restriction to $M$ of the projection $\pi : W \setminus \set{0} \to \M_s,\, u \mapsto \widetilde{u}$ is an odd continuous map from $M$ to $\M_s \setminus \widetilde{\Psi}_{\lambda_{k_0}}$ (see \eqref{13}), so
\[
i(M) \le i(\M_s \setminus \widetilde{\Psi}_{\lambda_{k_0}})
\]
by Proposition \ref{Proposition 7} $(i_2)$. This together with \eqref{120} contradicts \eqref{122}. So $M \cap N \ne \emptyset$ and hence
\[
\sup_{u \in M}\, \Phi(u) \ge \sup_{u \in M \cap N}\, \Phi(u) \ge \inf_{u \in M \cap N}\, \Phi(u) \ge \inf_{u \in N}\, \Phi(u).
\]
Since $M \in \F_k$ is arbitrary, it follows that
\[
c_k \ge \inf_{u \in N}\, \Phi(u).
\]
So it suffices to show that $\Phi$ is bounded from below on $N$.

We have
\[
\Phi(u) = I_s(u) - \lambda\, J_s(u) - K(u), \quad u \in E_r(\R^3),
\]
where $K$ is as in \eqref{200}. In particular,
\[
\Phi(u_t) = t^3\, \bigg(1 - \frac{\lambda}{\widetilde{\Psi}(u)}\bigg) - t^{2 \sigma - 3}\, K(u), \quad u \in \M_s,\, t \ge 0.
\]
Since $K$ is bounded on the bounded set $\M_s$, this gives
\[
\Phi(u_t) \ge t^3 \left(1 - \frac{\lambda^+}{\lambda_{k_0}}\right) - Ct^{2 \sigma - 3}
\]
for $u \in \widetilde{\Psi}_{\lambda_{k_0}}$, where $\lambda^+ = \max \set{\lambda,0} < \lambda_{k_0}$ and $C > 0$ is a constant. This implies that $\Phi$ is bounded from below on $N$ since $\sigma \in (18/7,3)$.

\subsubsection{Proof of Theorem \ref{Theorem 26}}

The variational functional associated with equation \eqref{179} is
\[
\Phi_\mu(u) = I_s(u) - \frac{\mu}{\sigma} \int_{\R^3} |u|^\sigma\, dx - \frac{1}{6} \int_{\R^3} |u|^6\, dx, \quad u \in E_r(\R^3).
\]
We will use a truncation of $\Phi_\mu$ inspired by Liu et al.\! \cite{MR3912770} (see also Garc{\'{\i}}a Azorero and Peral Alonso \cite{MR1083144}).

By Ianni and Ruiz \cite[Lemma 3.1]{MR2902293}, there are constants $c_1, c_2 > 0$ such that
\[
\Phi_\mu(u) \ge I_s(u) - c_1\, \mu\, I_s(u)^{(2 \sigma - 3)/3} - c_2\, I_s(u)^3 = g_\mu(I_s(u)),
\]
where
\[
g_\mu(t) = t - c_1\, \mu\, t^{(2 \sigma - 3)/3} - c_2\, t^3, \quad t \ge 0.
\]
Since $\sigma < 3$, $\exists \mu^\ast > 0$ such that for all $\mu \in (0,\mu^\ast)$, there are $0 < R_1(\mu) < R_2(\mu)$ such that
\[
g_\mu(t) < 0 \quad \forall t \in (0,R_1(\mu)) \cup (R_2(\mu),\infty), \qquad g_\mu(t) > 0 \quad \forall t \in (R_1(\mu),R_2(\mu)).
\]
Since $g_\mu \ge g_{\mu^\ast}$,
\begin{equation} \label{186}
R_1(\mu) \le R_1(\mu^\ast), \qquad R_2(\mu) \ge R_2(\mu^\ast).
\end{equation}

Take a smooth function $\xi_\mu : [0,\infty) \to [0,1]$ such that $\xi_\mu(t) = 1$ for $t \in [0,R_1(\mu)]$ and $\xi_\mu(t) = 0$ for $t \in [R_2(\mu),\infty)$, and set
\[
\widetilde{\Phi}_\mu(u) = \xi_\mu(I_s(u))\, \Phi_\mu(u).
\]
If $\widetilde{\Phi}_\mu(u) < 0$, then $I_s(u) \in (0,R_1(\mu))$ and hence $I_s(v) \in (0,R_1(\mu))$ for all $v$ in some neighborhood of $u$ by continuity. Then $\widetilde{\Phi}_\mu$ coincides with $\Phi_\mu$ in a neighborhood of $u$, so
\[
\widetilde{\Phi}_\mu(u) = \Phi_\mu(u), \qquad \widetilde{\Phi}_\mu'(u) = \Phi_\mu'(u).
\]
In particular, critical points of $\widetilde{\Phi}_\mu$ at negative levels are also critical points of $\Phi_\mu$ at the same levels.

First we show that $\widetilde{\Phi}_\mu$ satisfies the \PS{c} condition for all $c < 0$ if $\mu \in (0,\mu^\ast)$ is sufficiently small. Let $c < 0$ and let $\seq{u_j} \subset E_r(\R^3)$ be a \PS{c} sequence of $\widetilde{\Phi}_\mu$. Since $\widetilde{\Phi}_\mu(u_j) \to c$, for all sufficiently large $j$, $\widetilde{\Phi}_\mu(u_j) < 0$ and hence $\widetilde{\Phi}_\mu(u_j) = \Phi_\mu(u_j)$ and $\widetilde{\Phi}_\mu'(u_j) = \Phi_\mu'(u_j)$. So $\seq{u_j}$ has a renamed subsequence for which $\Phi_\mu(u_j) \to c$ and $\Phi_\mu'(u_j) \to 0$. Moreover, $I_s(u_j) \in (0,R_1(\mu))$, which together with the first inequality in \eqref{186} implies that $\seq{u_j}$ is bounded independently of $\mu \in (0,\mu^\ast)$. We may now proceed as in the proof of Lemma \ref{Lemma 5} to arrive at the inequality
\[
c \ge \frac{1}{3}\, S^{3/2} - 2 \mu \left(\frac{1}{\sigma} - \frac{1}{3}\right) \int_{\R^3} |u|^\sigma\, dx + \frac{1}{3} \int_{\R^3} |u|^6\, dx
\]
if $\seq{u_j}$ does not converge to its weak limit $u$ (see \eqref{187}). Since $c < 0$, this gives
\[
2 \mu \left(\frac{1}{\sigma} - \frac{1}{3}\right) \int_{\R^3} |u|^\sigma\, dx > \frac{1}{3}\, S^{3/2}.
\]
However, since $\seq{u_j}$ is bounded independently of $\mu$, so is $u$, and hence this inequality cannot hold for sufficiently small $\mu$.

Now we apply Proposition \ref{Proposition 11} to $\widetilde{\Phi}_\mu$. By taking $\mu^\ast$ smaller if necessary, we may assume that $\widetilde{\Phi}_\mu$ satisfies the \PS{c} condition for all $c < 0$ for $\mu \in (0,\mu^\ast)$. Set
\[
c_k := \inf_{M \in \F_k}\, \sup_{u \in M}\, \widetilde{\Phi}_\mu(u),
\]
where $\F_k$ is as in Proposition \ref{Proposition 11}, and note that $c_k > - \infty$ since $\widetilde{\Phi}_\mu$ is clearly bounded from below. It only remains to show that $c_k < 0$.

Let $M_t = \bgset{u_t : u \in \widetilde{\Psi}^{\lambda_k}}$ for $t > 0$. As in the proof of Theorem \ref{Theorem 20}, $M_t \in \F_k$ and
\[
\sup_{u \in M_t}\, \Phi_\mu(u) < 0
\]
if $t$ is sufficiently small. Since $I_s(u_t) = t^3$ for $u \in \M_s$ by \eqref{12}, if $t \le R_1(\mu)^{1/3}$, then $I_s(u) \in [0,R_1(\mu)]$ and hence $\widetilde{\Phi}_\mu(u) = \Phi_\mu(u)$ for $u \in M_t$. So taking $t$ sufficiently small gives
\[
c_k \le \sup_{u \in M_t}\, \widetilde{\Phi}_\mu(u) = \sup_{u \in M_t}\, \Phi_\mu(u) < 0.
\]

\subsubsection{Proof of Theorem \ref{Theorem 27}}

The variational functional associated with equation \eqref{201} is
\[
\Phi(u) = I_s(u) - \lambda\, J_s(u) + K(u), \quad u \in E_r(\R^3),
\]
where $K$ is as in \eqref{200}. Since $\lambda$ is not an eigenvalue of problem \eqref{55}, $\Phi$ satisfies the \PS{c} condition for all $c \in \R$ as in the proof of Theorem \ref{Theorem 24}. So we apply Corollary \ref{Corollary 1} with $c^\ast = \infty$. Since $\lambda > \lambda_k$ is not an eigenvalue, by taking $k$ larger if necessary, we may assume that $\lambda_k < \lambda < \lambda_{k+1}$. Then
\[
i(\M_s \setminus \widetilde{\Psi}_\lambda) = k
\]
by Theorem \ref{Theorem 1} \ref{Theorem 1.iii}. Since $\M_s \setminus \widetilde{\Psi}_\lambda$ is an open symmetric subset of $\M_s$, then it has a compact symmetric subset $C$ of index $k$ (see the proof of Proposition 3.1 in Degiovanni and Lancelotti \cite{MR2371112}). We take $A_0 = C$.

Let $R > \rho > 0$ and let
\[
A = \set{u_R : u \in A_0}, \qquad X = \set{u_t : u \in A_0,\, 0 \le t \le R}.
\]
For $u \in \M_s$ and $t \ge 0$, we have
\begin{equation} \label{280}
\Phi(u_t) = t^3\, \bigg(1 - \frac{\lambda}{\widetilde{\Psi}(u)}\bigg) + t^{2 \sigma - 3}\, K(u).
\end{equation}
This together with Lemma \ref{Lemma 7} gives
\[
\Phi(u_t) \ge c_1\, t^{2 \sigma - 3} - \bigg(\frac{\lambda}{\lambda_1} - 1\bigg)\, t^3 \quad \forall u \in \widetilde{\Psi}^{\lambda_{k+1}},\, t \ge 0
\]
for some constant $c_1 > 0$ and
\[
\Phi(u_t) \ge \bigg(1 - \frac{\lambda}{\lambda_{k+1}}\bigg)\, t^3 \quad \forall u \in \M_s \setminus \widetilde{\Psi}^{\lambda_{k+1}},\, t \ge 0.
\]
Since $\sigma < 3$ and $\lambda < \lambda_{k+1}$, it follows that
\[
\inf_{u \in \M_s}\, \Phi(u_\rho) > 0
\]
if $\rho$ is sufficiently small. Since $A_0$ is compact and $\widetilde{\Psi} < \lambda$ on $A_0$, \eqref{280} also gives
\[
\Phi(u_t) \le - t^3\, \bigg(\frac{\lambda}{c_2} - 1\bigg) + c_3\, t^{2 \sigma - 3} \quad \forall u \in A_0,\, t \ge 0
\]
for some constants $0 < c_2 < \lambda$ and $c_3 > 0$. So
\[
\sup_{u \in A}\, \Phi(u) \le 0
\]
if $R$ is sufficiently large and
\[
\sup_{u \in X}\, \Phi(u) < \infty.
\]
The desired conclusion now follows from Corollary \ref{Corollary 1}.

\subsubsection{Proof of Theorem \ref{Theorem 28}}

The variational functional associated with equation \eqref{202} is
\begin{multline*}
\Phi(u) = \frac{1}{2} \int_{\R^3} |\nabla u|^2\, dx + \frac{1}{16 \pi} \int_{\R^3} \int_{\R^3} \frac{u^2(x)\, u^2(y)}{|x - y|}\, dx\, dy - \frac{\lambda}{3} \int_{\R^3} |u|^3\, dx\\[7.5pt]
+ \frac{\mu}{\sigma} \int_{\R^3} |u|^\sigma\, dx - \frac{1}{6} \int_{\R^3} |u|^6\, dx, \quad u \in E_r(\R^3).
\end{multline*}
First we prove the following local \PS{} condition, where $S$ is as in \eqref{88}.

\begin{lemma} \label{Lemma 8}
If $\lambda, \mu > 0$ and $\sigma \in (18/7,3)$, then $\Phi$ satisfies the {\em \PS{c}} condition for $0 < c < \dfrac{1}{3}\, S^{3/2}$.
\end{lemma}

\begin{proof}
Let $\seq{u_j} \subset E_r(\R^3)$ be a \PS{c} sequence of $\Phi$. Then
\begin{multline} \label{381}
\frac{1}{2} \int_{\R^3} |\nabla u_j|^2\, dx + \frac{1}{16 \pi} \int_{\R^3} \int_{\R^3} \frac{u_j^2(x)\, u_j^2(y)}{|x - y|}\, dx\, dy - \frac{\lambda}{3} \int_{\R^3} |u_j|^3\, dx\\[7.5pt]
+ \frac{\mu}{\sigma} \int_{\R^3} |u_j|^\sigma\, dx - \frac{1}{6} \int_{\R^3} |u_j|^6\, dx = c + \o(1)
\end{multline}
and
\begin{multline} \label{383}
\int_{\R^3} \nabla u_j \cdot \nabla v\, dx + \frac{1}{4 \pi} \int_{\R^3} \int_{\R^3} \frac{u_j^2(x)\, u_j(y)\, v(y)}{|x - y|}\, dx\, dy - \lambda \int_{\R^3} |u_j|\, u_j v\, dx\\[7.5pt]
+ \mu \int_{\R^3} |u_j|^{\sigma - 2}\, u_j v\, dx - \int_{\R^3} u_j^5\, v\, dx = \o(\norm{v})
\end{multline}
for all $v \in E_r(\R^3)$. Taking $v = u_j$ in \eqref{383} gives
\begin{multline} \label{382}
\int_{\R^3} |\nabla u_j|^2\, dx + \frac{1}{4 \pi} \int_{\R^3} \int_{\R^3} \frac{u_j^2(x)\, u_j^2(y)}{|x - y|}\, dx\, dy - \lambda \int_{\R^3} |u_j|^3\, dx\\[7.5pt]
+ \mu \int_{\R^3} |u_j|^\sigma\, dx - \int_{\R^3} |u_j|^6\, dx = \o(\norm{u_j}).
\end{multline}

First we show that $\seq{u_j}$ is bounded in $E_r(\R^3)$. Suppose $\norm{u_j} \to \infty$ for a renamed subsequence. Set
\[
t_j = t_{u_j} = \frac{1}{I_s(u_j)^{1/3}}, \quad \widetilde{u}_j = (u_j)_{t_j}, \quad \widetilde{t}_j = \frac{1}{t_j} = I_s(u_j)^{1/3}
\]
(see \eqref{13}). Then $\widetilde{u}_j \in \M_s$ and
\begin{equation} \label{367}
u_j = (\widetilde{u}_j)_{\widetilde{t}_j} = \widetilde{t}_j^{\,2}\, \widetilde{u}_j(\widetilde{t}_j\, \cdot)
\end{equation}
by \eqref{18}. Since $\M_s$ is a bounded manifold, $\seq{\widetilde{u}_j}$ is bounded. Since $\norm{u_j} \to \infty$, $I_s(u_j) \to \infty$ and hence $\widetilde{t}_j \to \infty$. Using \eqref{367} and noting that
\[
\norm{u_j} = \|(\widetilde{u}_j)_{\widetilde{t}_j}\| = \O(\widetilde{t}_j^{\,3/2}),
\]
\eqref{381} and \eqref{382} can be written as
\begin{multline} \label{368}
\widetilde{t}_j^{\,3} \left(\frac{1}{2} \int_{\R^3} |\nabla \widetilde{u}_j|^2\, dx + \frac{1}{16 \pi} \int_{\R^3} \int_{\R^3} \frac{\widetilde{u}_j^{\,2}(x)\, \widetilde{u}_j^{\,2}(y)}{|x - y|}\, dx\, dy - \frac{\lambda}{3} \int_{\R^3} |\widetilde{u}_j|^3\, dx\right)\\[7.5pt]
= \frac{\widetilde{t}_j^{\,9}}{6} \int_{\R^3} |\widetilde{u}_j|^6\, dx - \frac{\mu\, \widetilde{t}_j^{\,2 \sigma - 3}}{\sigma} \int_{\R^3} |\widetilde{u}_j|^\sigma\, dx + c + \o(1)
\end{multline}
and
\begin{multline} \label{369}
\widetilde{t}_j^{\,3} \left(\int_{\R^3} |\nabla \widetilde{u}_j|^2\, dx + \frac{1}{4 \pi} \int_{\R^3} \int_{\R^3} \frac{\widetilde{u}_j^{\,2}(x)\, \widetilde{u}_j^{\,2}(y)}{|x - y|}\, dx\, dy - \lambda \int_{\R^3} |\widetilde{u}_j|^3\, dx\right)\\[7.5pt]
= \widetilde{t}_j^{\,9} \int_{\R^3} |\widetilde{u}_j|^6\, dx - \mu\, \widetilde{t}_j^{\,2 \sigma - 3} \int_{\R^3} |\widetilde{u}_j|^\sigma\, dx + \o(\widetilde{t}_j^{\,3/2}),
\end{multline}
respectively. Since $\widetilde{t}_j \to \infty$, dividing \eqref{369} by $\widetilde{t}_j^{\,3}$ gives
\begin{multline*}
\widetilde{t}_j^{\,6} \int_{\R^3} |\widetilde{u}_j|^6\, dx - \frac{\mu}{\widetilde{t}_j^{\,2\,(3 - \sigma)}} \int_{\R^3} |\widetilde{u}_j|^\sigma\, dx = \int_{\R^3} |\nabla \widetilde{u}_j|^2\, dx + \frac{1}{4 \pi} \int_{\R^3} \int_{\R^3} \frac{\widetilde{u}_j^{\,2}(x)\, \widetilde{u}_j^{\,2}(y)}{|x - y|}\, dx\, dy\\[7.5pt]
- \lambda \int_{\R^3} |\widetilde{u}_j|^3\, dx + \o(1).
\end{multline*}
Since the right-hand side is bounded and $\sigma < 3$, this implies that $\widetilde{u}_j \to 0$ in $L^6(\R^3)$. Since $\seq{\widetilde{u}_j}$ is bounded in $L^r(\R^3)$ for $r \in (18/7,3)$, then $\widetilde{u}_j \to 0$ in $L^3(\R^3)$ by interpolation. Now multiplying \eqref{368} by $6$, subtracting \eqref{369}, and dividing by $\widetilde{t}_j^{\,3}$ gives
\[
2 \int_{\R^3} |\nabla \widetilde{u}_j|^2\, dx + \frac{1}{8 \pi} \int_{\R^3} \int_{\R^3} \frac{\widetilde{u}_j^{\,2}(x)\, \widetilde{u}_j^{\,2}(y)}{|x - y|}\, dx\, dy = \o(1).
\]
This implies that $\widetilde{u}_j \to 0$ in $E_r(\R^3)$, contradicting $\widetilde{u}_j \in \M_s$.

We may now proceed as in the proof of Lemma \ref{Lemma 5} to arrive at the inequality
\[
c \ge \frac{1}{3}\, S^{3/2} + 2 \mu \left(\frac{1}{\sigma} - \frac{1}{3}\right) \int_{\R^3} |u|^\sigma\, dx + \frac{1}{3} \int_{\R^3} |u|^6\, dx
\]
if $\seq{u_j}$ does not converge to its weak limit $u$ (see \eqref{187}). Since $\mu > 0$ and $\sigma < 3$, this gives $c \ge \dfrac{1}{3}\, S^{3/2}$, contrary to assumption.
\end{proof}

We are now ready to prove Theorem \ref{Theorem 28}.

\begin{proof}[Proof of Theorem \ref{Theorem 28}]
In view of Lemma \ref{Lemma 8}, we apply Corollary \ref{Corollary 1} with $c^\ast = \dfrac{1}{3}\, S^{3/2}$. By taking $k$ larger if necessary, we may assume that $\lambda_k < \lambda_{k+1}$. Fix $\widetilde{\lambda} \le \lambda$ such that $\lambda_k < \widetilde{\lambda} < \lambda_{k+1}$. Then
\[
i(\M_s \setminus \widetilde{\Psi}_{\widetilde{\lambda}}) = k
\]
by Theorem \ref{Theorem 1} \ref{Theorem 1.iii}. Since $\M_s \setminus \widetilde{\Psi}_{\widetilde{\lambda}}$ is an open symmetric subset of $\M_s$, then it has a compact symmetric subset $C$ of index $k$ (see the proof of Proposition 3.1 in Degiovanni and Lancelotti \cite{MR2371112}). We take $A_0 = C$.

Let $R > \rho > 0$ and let
\[
A = \set{u_R : u \in A_0}, \qquad X = \set{u_t : u \in A_0,\, 0 \le t \le R}.
\]
For $u \in \M_s$ and $t \ge 0$, we have
\begin{equation} \label{380}
\Phi(u_t) = t^3\, \bigg(1 - \frac{\lambda}{\widetilde{\Psi}(u)}\bigg) + \mu\, t^{2 \sigma - 3}\, K(u) - \frac{t^9}{6} \int_{\R^3} |u|^6\, dx.
\end{equation}
Taking $a > \lambda$ in Lemma \ref{Lemma 7}, combining with \eqref{380}, and noting that
\[
S \left(\int_{\R^3} |u|^6\, dx\right)^{1/3} \le \int_{\R^3} |\nabla u|^2\, dx \le 2 I_s(u) = 2
\]
gives
\[
\Phi(u_t) \ge c_1\, \mu\, t^{2 \sigma - 3} - \bigg(\frac{\lambda}{\lambda_1} - 1\bigg)\, t^3 - \frac{4}{3 S^3}\, t^9 \quad \forall u \in \widetilde{\Psi}^a,\, t \ge 0
\]
for some constant $c_1 > 0$ and
\[
\Phi(u_t) \ge \bigg(1 - \frac{\lambda}{a}\bigg)\, t^3 - \frac{4}{3 S^3}\, t^9 \quad \forall u \in \M_s \setminus \widetilde{\Psi}^a,\, t \ge 0.
\]
Since $\sigma < 3$ and $a > \lambda$, it follows that
\[
\inf_{u \in \M_s}\, \Phi(u_\rho) > 0
\]
if $\rho$ is sufficiently small. Since $\widetilde{\Psi} < \widetilde{\lambda} \le \lambda$ on $A_0$ and $0$ is not in the compact set $A_0$, \eqref{380} also gives
\[
\Phi(u_t) \le c_2\, \mu\, t^{2 \sigma - 3} - c_3\, t^9 \quad \forall u \in A_0,\, t \ge 0
\]
for some constants $c_2, c_3 > 0$. So
\[
\sup_{u \in A}\, \Phi(u) \le 0
\]
if $R$ is sufficiently large and
\[
\sup_{u \in X}\, \Phi(u) \le \max_{t \ge 0} \left(c_2\, \mu\, t^{2 \sigma - 3} - c_3\, t^9\right) < \frac{1}{3}\, S^{3/2}
\]
if $\mu > 0$ is sufficiently small. The desired conclusion now follows from Corollary \ref{Corollary 1}.
\end{proof}

\section*{Acknowledgements}

Carlo Mercuri is a member of the group GNAMPA of Istituto Nazionale di Alta Matematica (INdAM); he would like to thank GNAMPA for the support within the programme Progetti GNAMPA 2024, CUP E53C23001670001. Both authors would like to thank GNAMPA for the support within the programme “Bandi Professori Visitatori gruppi Indam 2024”, GNAMPA - CUP E53C22001930001.

This work was partially completed while Kanishka Perera was visiting the Department of Physical, Computer and Mathematical Sciences at the University of Modena and Reggio Emilia, and he is grateful for the kind hospitality of the host department. His visit was partially supported by the Simons Foundation Award 962241 “Local and nonlocal variational problems with lack of compactness”.

\end{document}